\documentclass[11pt]{article} %{amsart}
\usepackage{amsmath}
\usepackage{amssymb,bbm}
\usepackage{amsthm}
\usepackage[top=1in, bottom=1in, left=0.75in, right=0.75in]{geometry}
\usepackage{xcolor}
\usepackage{dsfont}
\usepackage{mathrsfs}
\usepackage{hyperref}
\usepackage{tikz-cd}
\usepackage{quiver}
\hypersetup{
  colorlinks=true,
  linkcolor=blue!50!black,
  citecolor=blue!50!black,
  urlcolor=blue!50!black,
  pdfauthor={David Owen Horace Cutler, Mel Deaton},
  pdftitle={Metrics from valuations on convex bodies}
}
\usepackage{microtype}
\usepackage{graphicx}
\usepackage{float}
\usepackage[style=alphabetic]{biblatex}
\usepackage{nicefrac}
\usepackage{tikz}
\usepackage{wrapfig}
\usepackage{eulervm}
\usepackage[normalem]{ulem}
\usepackage{physics}

\usepackage[shortlabels]{enumitem}
\setlist[itemize]{topsep=2pt, itemsep=1pt, parsep=0pt, partopsep=0pt}
\setlist[enumerate]{topsep=2pt, itemsep=1pt, parsep=0pt, partopsep=0pt}

\usepackage{caption}
\usetikzlibrary{calc,positioning,patterns,decorations,decorations.pathreplacing}

\definecolor{mycolor}{RGB}{99, 76, 217}

\newtheorem{theorem}{Theorem}[section]
\newtheorem{lemma}[theorem]{Lemma}
\newtheorem{prop}[theorem]{Proposition}
\newtheorem{coro}[theorem]{Corollary}

\newtheorem{defn}[theorem]{Definition}

\newtheorem*{theoremA}{Theorem A}
\newtheorem*{coroB}{Corollary B}
\newtheorem*{theoremC}{Theorem C}
\newtheorem*{theoremD}{Theorem D}
\newtheorem*{theoremE}{Theorem E}
\newtheorem*{coroF}{Corollary F}
\newtheorem*{coroG}{Corollary G}
\newtheorem*{theoremH}{Theorem H}

\newtheorem{example}[theorem]{Example}
\newtheorem{openq}[theorem]{Open Question}

\newcommand{\R}{\mathbb{R}}

\newcommand{\K}{\mathcal{K}}
\newcommand{\Gr}{\textup{Gr}}

\DeclareMathOperator{\supp}{supp}
\DeclareMathOperator{\conv}{conv}

\newcommand{\ol}[1]{\overline{#1}}

\renewcommand{\epsilon}{\varepsilon}
\newcommand{\bigast}{\mathop{\vcenter{\hbox{\Large$\ast$}}}\displaylimits}

\makeatletter
\newcommand{\mynameis}[1]{%
  \phantomsection
  \def\@currentlabel{#1}%
  \def\@currentlabelname{#1}%
}
\makeatother

\usepackage[utf8]{inputenc} 

\DeclareUnicodeCharacter{0327}{} 
\begin{document}
\begin{center}

{\LARGE\bfseries Metrics from valuations on convex bodies}\\[10pt]

\begin{tabular}{c@{\hspace{1.2cm}}c}
\textsc{David Owen Horace Cutler} & \textsc{Mel Deaton}\\[-1pt]
\small University of Michigan & \small Johns Hopkins University\\
\small \href{mailto:cutler@umich.edu}{\tt cutler@umich.edu} & \small \href{mailto:rdeaton3@jh.edu}{\tt rdeaton3@jh.edu}
\end{tabular}

\end{center}

\begin{abstract}
We associate to each valuation $\phi$ on $\mathcal{K}^n$ a pseudometric $D_{\phi}$, generalizing the classical symmetric difference and mean width metrics. Broad criteria are given under which $D_{\phi}$ is a bona fide metric, in which case it is automatically intrinsic. Moreover, a necessary and sufficient condition for the $D_{\phi}$-balls to be bounded in the Hausdorff sense is given; as a corollary, we obtain a condition for $D_{\phi}$ to be proper, and hence complete and geodesic.  
The topology induced by $D_{\phi}$ is also investigated: we provide conditions relating different forms of convergence, showing for instance that the intrinsic volumes $\left \{\textup{V}_j : j \in \{1, \dots, n\} \right\}$ induce metrics that are topologically equivalent to the Hausdorff distance. Finally, we apply our construction to the metric geometry of singularity types by exhibiting, on the subspace of $\mathcal{S}(\mathbb{CP}^n,\omega_{FS})$ corresponding to toric model potentials, a geodesic metric that is bilipschitz equivalent to the $d_{\mathcal{S}}$ metric of Darvas, Di Nezza, and Lu.
\end{abstract}

\newpage

\tableofcontents

\newpage

\section{Introduction}

In convex geometry, various notions of ``distance'' are considered on $\mathcal{K}^n$, the space of nonempty compact convex subsets of $\mathbb{R}^n$.
 Most commonly equipped is the \emph{Hausdorff distance} $d_\mathcal{H}$, which may be written 
\begin{align}\label{haus-dist}d_\mathcal{H}(K,L) := \inf \left \{\delta > 0 : K \subseteq L + \delta B^n \quad \text{and} \quad L \subseteq K + \delta B^n \right\},\quad \text{for all } K, L \in \mathcal{K}^n.\end{align}

\begin{figure}[H]
\centering
\begin{tikzpicture}[scale=0.66,line cap=round,line join=round]
\tikzset{
  outline/.style={black!100,line width=0.95pt},
  inner/.style={black!100,dashed,line width=0.8pt},
  Kfill/.style={fill=gray!18,opacity=.50},
  Lfill/.style={fill=gray!40,opacity=.65},
  Ksumfill/.style={fill=gray!18,opacity=.28},
  Lsumfill/.style={fill=gray!40,opacity=.35},
  lab/.style={black!100,font=\footnotesize}
}

\pgfmathsetmacro{\delone}{0.37883}
\pgfmathsetmacro{\deltwo}{0.89904}

\newcommand{\Kpath}[3]{
  plot[smooth cycle,samples=361,domain=0:360,variable=\t]
    ({#1 + ((1.05 + 0.13*cos(2*\t-20) + 0.07*sin(3*\t+10)) + #3)*cos(\t)
          - (-0.26*sin(2*\t-20) + 0.21*cos(3*\t+10))*sin(\t)},
     {#2 + ((1.05 + 0.13*cos(2*\t-20) + 0.07*sin(3*\t+10)) + #3)*sin(\t)
          + (-0.26*sin(2*\t-20) + 0.21*cos(3*\t+10))*cos(\t)})
}

\newcommand{\Lpath}[3]{
  plot[smooth cycle,samples=361,domain=0:360,variable=\t]
    ({#1 + ((0.82 + 0.11*cos(2*\t+25) + 0.05*sin(4*\t-10)) + #3)*cos(\t)
          - (-0.22*sin(2*\t+25) + 0.20*cos(4*\t-10))*sin(\t)},
     {#2 + ((0.82 + 0.11*cos(2*\t+25) + 0.05*sin(4*\t-10)) + #3)*sin(\t)
          + (-0.22*sin(2*\t+25) + 0.20*cos(4*\t-10))*cos(\t)})
}

\begin{scope}[xshift=-6.1cm]
  \fill[Kfill] \Kpath{-0.20}{-0.05}{0};
  \fill[Lfill] \Lpath{0.45}{0.18}{0};
  \draw[outline] \Kpath{-0.20}{-0.05}{0};
  \draw[outline] \Lpath{0.45}{0.18}{0};
  \node[lab] at (-1.00,-1.18) {$K$};
  \node[lab] at (1.02,1.00) {$L$};
\end{scope}

\begin{scope}
  \fill[Ksumfill] \Kpath{-0.20}{-0.05}{\delone};
  \fill[Lfill] \Lpath{0.45}{0.18}{0};
  \draw[outline] \Kpath{-0.20}{-0.05}{\delone};
  \draw[inner] \Kpath{-0.20}{-0.05}{0};
  \draw[outline] \Lpath{0.45}{0.18}{0};
  \node[lab] at (0.00,-2.02) {$K+\delta_1 B$};
\end{scope}

\begin{scope}[xshift=6.1cm]
  \fill[Lsumfill] \Lpath{0.45}{0.18}{\deltwo};
  \fill[Kfill] \Kpath{-0.20}{-0.05}{0};
  \draw[outline] \Lpath{0.45}{0.18}{\deltwo};
  \draw[inner] \Lpath{0.45}{0.18}{0};
  \draw[outline] \Kpath{-0.20}{-0.05}{0};
  \node[lab] at (0.05,-2.02) {$L+\delta_2 B$};
\end{scope}

\end{tikzpicture}
\caption{A visual representation of the Hausdorff distance --- in this case, $d_\mathcal{H}(K,L) = \delta_2$.}
\end{figure}

Another similarly characteristic --- albeit seemingly unrelated --- construction on $\mathcal{K}^n$ comes from the polynomial behavior of volume on outer parallel bodies, described in the classical \emph{Steiner formula}
\begin{align}\label{steiner-form}\textup{vol}(K + tB^n) = \sum_{j = 0}^n t^{n-j} \kappa_{n-j}\textup{V}_j(K) \quad \text{for all } K \in \mathcal{K}^n, \quad t \geq 0,\end{align}
where $\kappa_{m}$ denotes the volume of the unit ball $B^m$ in $\mathbb{R}^m$. Of particular interest are the involved coefficients, which encode geometric data: for a convex body $K$, the $\textup{V}_j(K)$, known as \emph{intrinsic volumes}, satisfy
\begin{align*}
  &\textup{V}_n(K)=\textup{vol}(K), \quad
  \textup{V}_{n-1}(K)=\frac{1}{2}s(K),
  \quad \text{where} \quad
  s(K):=\mathcal{H}^{n-1}(\partial K), \\
  &\text{and} \quad
  \textup{V}_1(K)\varpropto w(K),
  \quad \text{where} \quad
  w(K):=\frac{2}{n\kappa_n}
  \int_{S^{n-1}}h_K(u)\,d\omega(u)
  =\frac{2\kappa_{n-1}}{n\kappa_n}\textup{V}_1(K);
\end{align*}
here $\omega$ is the standard surface measure on $S^{n-1}$, and so $w(K)$ may be thought of as the \emph{mean width of} $K$.

Intriguingly, some of these intrinsic volumes may be studied for how they furnish alternatives to the Hausdorff distance: the seminal example is perhaps the \emph{symmetric difference metric} 
\begin{align}\label{sym-dif}
  \Delta_\textup{vol}(K,L) := \textup{vol}(K \, \triangle \, L) &= \textup{vol}(K) + \textup{vol}(L) - 2\textup{vol}(K \cap L) \\ &= \int_{\mathbb{R}^n} |\chi_K(x) - \chi_L(x)| d\mu(x), \notag
\end{align}
for $K, L \in \mathcal{K}^n$, which admits a comparably less-known ``dual'' in the form of the $L^1$-type \emph{mean width metric}
\begin{align}  
    \label{mean-width}\rho_w(K,L) := 2w(\textup{conv}(K \cup L)) - w(K) - w(L) = \frac{2}{n \kappa_n}\int_{S^{n-1}} |h_K(u) - h_L(u)| d\omega
\end{align}
for all $K, L \in \mathcal{K}^n$, a construction due to Florian in \cite{florian}. Both are metrics (topologically) equivalent to the Hausdorff distance on their respective loci of nondegeneracy: $\mathcal{K}^n$ for the latter and 
$$\mathcal{K}^n_n := \{K \in \mathcal{K}^n : \textup{int}(K) \neq \varnothing\} = \{K \in \mathcal{K}^n : \textup{vol}(K) > 0\},$$
the set of ``convex bodies'', for the former --- note we use this nomenclature in the nonempty interior case. 

\begin{figure}[H]
\centering
\begin{tikzpicture}[scale=0.8,line cap=round,line join=round]
\tikzset{
  outline/.style={black!100,line width=1pt},
  aux/.style={black!100,dashed,line width=.85pt},
  dist/.style={black!100,line width=.95pt,<->},
  dir/.style={black!100,line width=.95pt,->},
  Kfill/.style={fill=gray!18,opacity=.5},
  Lfill/.style={fill=gray!40,opacity=.65},
  lab/.style={black!100,font=\footnotesize}
}

\pgfmathsetmacro{\xKsup}{0.984}
\pgfmathsetmacro{\xLsup}{1.361}

\newcommand{\Kpath}[3]{
  plot[smooth cycle,samples=181,domain=0:360,variable=\t]
    ({#1 + ((1.05 + 0.13*cos(2*\t-20) + 0.07*sin(3*\t+10)) + #3)*cos(\t)
          - (-0.26*sin(2*\t-20) + 0.21*cos(3*\t+10))*sin(\t)},
     {#2 + ((1.05 + 0.13*cos(2*\t-20) + 0.07*sin(3*\t+10)) + #3)*sin(\t)
          + (-0.26*sin(2*\t-20) + 0.21*cos(3*\t+10))*cos(\t)})
}

\newcommand{\Lpath}[3]{
  plot[smooth cycle,samples=181,domain=0:360,variable=\t]
    ({#1 + ((0.82 + 0.11*cos(2*\t+25) + 0.05*sin(4*\t-10)) + #3)*cos(\t)
          - (-0.22*sin(2*\t+25) + 0.20*cos(4*\t-10))*sin(\t)},
     {#2 + ((0.82 + 0.11*cos(2*\t+25) + 0.05*sin(4*\t-10)) + #3)*sin(\t)
          + (-0.22*sin(2*\t+25) + 0.20*cos(4*\t-10))*cos(\t)})
}

\begin{scope}[xshift=-3.8cm]
\path[use as bounding box] (-2.35,-2.05) rectangle (2.35,2.05);

\node[anchor=east] at (-1.45,-0.08)
  {$\mathrm{vol}\!\left(\vphantom{\rule[-0.97cm]{0pt}{1.94cm}}\right.$};
\node[anchor=west] at (1.45,-0.08)
  {$\left.\vphantom{\rule[-0.97cm]{0pt}{1.94cm}}\right)$};

\fill[Kfill] \Kpath{-0.20}{-0.05}{0};
\fill[Lfill] \Lpath{0.45}{0.18}{0};

\begin{scope}
  \clip \Kpath{-0.20}{-0.05}{0};
  \fill[white] \Lpath{0.45}{0.18}{0};
\end{scope}

\draw[outline] \Kpath{-0.20}{-0.05}{0};
\draw[outline] \Lpath{0.45}{0.18}{0};

\node[lab] at (0,-1.72) {$\Delta_{\mathrm{vol}}$};
\end{scope}

\begin{scope}[xshift=3.8cm]
\path[use as bounding box] (-2.0,-2.05) rectangle (2.0,2.05);

\fill[Kfill] \Kpath{-0.20}{-0.05}{0};
\fill[Lfill] \Lpath{0.45}{0.18}{0};

\begin{scope}
  \clip \Kpath{-0.20}{-0.05}{0};
  \clip \Lpath{0.45}{0.18}{0};
  \fill[white] (-3.2,-3.2) rectangle (3.2,3.2);
  \begin{scope}[rotate around={45:(0,0)}]
    \foreach \x in {-6.00,-5.64,...,6.24} {
      \fill[Kfill] (\x,-6) rectangle ++(0.18,12);
    }
    \foreach \x in {-5.82,-5.46,...,6.42} {
      \fill[Lfill] (\x,-6) rectangle ++(0.18,12);
    }
  \end{scope}
\end{scope}

\draw[outline] \Kpath{-0.20}{-0.05}{0};
\draw[outline] \Lpath{0.45}{0.18}{0};

\draw[aux] (\xKsup,-1.15) -- (\xKsup,1.60);
\draw[aux] (\xLsup,-1.15) -- (\xLsup,1.60);

\draw[dist] (\xKsup,1.28) -- (\xLsup,1.28);

\draw[dir] (-1.55,1.48) -- (-0.45,1.48);
\node[lab,above=2pt] at (-1.00,1.48) {$u$};

\node[lab] at (0,-1.72) {$\rho_w$};
\end{scope}

\end{tikzpicture}
\caption{Visual representations of the symmetric difference metric and the mean width metric --- in particular, $\rho_w$ averages, over $u \in S^{n-1}$, the distance between the frontmost positions of $K$ and $L$ when viewed in the direction $u$.}
\end{figure}

Constructions analogous to $\Delta_{\textup{vol}}$ and $\rho_w$ and incorporating the intermediate intrinsic volumes $\{\textup{V}_2, \dots, \textup{V}_{n-1}\}$ moreover exist: na\"ively, one may just replace $\textup{vol}$ in (\ref{sym-dif}) and $w$ in (\ref{mean-width}) with a given $\textup{V}_j$, yielding
\begin{align*}&\Delta_j(K,L) := \textup{V}_j(K) + \textup{V}_j(L) - 2\textup{V}_j(K \cap L) \quad \text{for all } K, L \in \mathcal{K}^n,\end{align*}
the so-called \emph{intrinsic volume deviations}, and
\begin{align*}
    \rho_j(K,L) := 2\textup{V}_j(\textup{conv}(K \cup L)) - \textup{V}_j(K) - \textup{V}_j(L) \quad \text{for all } K, L \in \mathcal{K}^n.
\end{align*}
Special cases of both of these have been considered since \cite{florian} --- Groemer proved in \cite{groemer} explicit estimates comparing the Hausdorff distance to
\begin{align}\label{surf-area-dev}
    \Delta_s(K,L) &:= s(K) + s(L) - 2s(K \cap L) \\ &= \frac{1}{2}\Delta_{n-1}(K,L), \notag
\end{align}
the \emph{surface area deviation}, whenever $K$ and $L$ are convex bodies. Later, Besau, Hoehner, and Kur thoroughly studied the $\Delta_j$ from the perspective of best approximation in \cite{Besau_2019}, additionally noting that
\begin{align*}
\Delta_j(B^n, C_\epsilon) + \Delta_j(B^n, C_{-\epsilon}) < \Delta_j(C_\epsilon, C_{-\epsilon}), \quad \text{where} \quad C_{\epsilon} := \{x \in B^n : x_n \geq \epsilon\}, \quad C_{-\epsilon} := -C_\epsilon,
\end{align*}
for suitable $j \in \{1, \dots, n-1\}$ and sufficiently large $\epsilon > 0$ 
 --- that is, the $\{\Delta_j\}_{j = 1}^{n-1}$ are \emph{not} metrics, not even on $\mathcal{K}^n_n$. Similar is the case for $\{\rho_j\}_{j = 2}^n$, which are also not metrics --- see \cite{florian}. Some questions thus arise:
 \begin{enumerate}[label=(\arabic*)]
     \item is there some meaningful general construction that associates a metric to \emph{each} $\textup{V}_j$, i.e.
     $$\textup{V}_j \mapsto D_j \quad \text{for } j \in \{1, \dots, n\}, \quad D_1 = \rho_1 \varpropto \rho_w, \quad D_n = \Delta_{\textup{vol}},$$
     and, if so, 
     \item do the resultant metrics induce the same topology as the Hausdorff distance?
 \end{enumerate}

If the first question may be resolved, it would moreover be interesting to see if one may get metrics out of more general \emph{valuations} --- indeed, note the present authors and collaborators considered in \cite{us} 
\begin{align}\label{dG}d_G(K,L) := \sum_{j = 1}^n\big[ 2\textup{V}_j(G, \textup{conv}(K \cup L)) - \textup{V}_j(G, K) - \textup{V}_j(G, L)\big] \quad \text{for all } K, L \in \mathcal{K}(G),\end{align}
given some $G \in \mathcal{K}^n_n$ --- here the $\textup{V}_j$ denote mixed volumes; see (\ref{intrinsic-mixed}) for notation. It is clear $d_G$ is essentially just $\rho_j$ with the role of $\textup{V}_j$ replaced by the valuation $\sum_{j = 0}^n \textup{V}_j(G, \cdot) \varpropto \textup{vol}(G + \cdot)$ --- furthermore, it is 
itself \emph{not} a metric, though it is a ``quasimetric'' in the particular case where $G$ is the unit simplex $$ \Delta := \left \{ x \in \mathbb{R}^n : x_j \geq 0 \text{ for all } j ,\, \sum_{j = 1}^n x_j \leq 1\right \},$$
a type of distance function weaker than a metric that nonetheless generates a metrizable topology --- here, this happens to be that of the Hausdorff distance. 

That all notwithstanding, such quasimetrics are not conducive to geometric investigation, given that they fail to satisfy a triangle inequality --- accordingly, they are not guaranteed to yield nontrivial length structures.

\begin{figure}[H]
\centering
\begin{tikzpicture}[scale=1.05, line cap=round, line join=round]
\tikzset{
  outline/.style={black,line width=0.95pt},
  faint/.style={black,opacity=.40,line width=0.85pt},
  hop/.style={black,line width=0.9pt,->},
  bigarrow/.style={black,line width=1.05pt,->},
  pt/.style={circle,fill=black,inner sep=1.25pt,transform shape=false},
  lab/.style={black,font=\small,transform shape=false},
}

\def\xL{0}
\def\xR{9}
\def\n{8}
\def\Hup{0.78}
\def\Hdown{0.90}
\def\braceY{-1.08}
\def\capPad{0.55}
\def\topShift{0.16}
\def\botShift{0.16}
\def\xstretch{1.35}

\pgfmathtruncatemacro{\nMinusOne}{\n-1}

\begin{scope}[xscale=\xstretch]

\path[use as bounding box] (-0.7,{(\braceY-0.28)-\capPad}) rectangle (9.7,{(\Hup+0.9)+\topShift});

\draw[outline] (\xL,0) -- (\xR,0);

\coordinate (x0) at (\xL,0);
\coordinate (x1) at (\xR,0);
\node[pt] at (x0) {};
\node[pt] at (x1) {};
\node[lab,below=3pt,xshift=-7pt] at (x0) {$x_0$};
\node[lab,below=3pt,xshift=10pt]  at (x1) {$x_1$};

\draw[bigarrow]
  ($(x0)+(0,\topShift)$) .. controls ($(x0)+(0,\Hup+\topShift)$) and ($(x1)+(0,\Hup+\topShift)$) .. ($(x1)+(0,\topShift)$);
\node[lab,above=0pt] at ($(x0)!0.5!(x1)+(0,\Hup+\topShift)$) {$|x_1-x_0|^{p}$};

\foreach \k in {1,...,\nMinusOne}{
  \coordinate (xi\k) at ($ (x0)!\k/\n!(x1) $);
  \node[pt] at (xi\k) {};
}

\foreach \k [evaluate=\k as \km using int(\k-1)] in {1,...,\n}{
  \coordinate (A) at ($ (x0)!\km/\n!(x1) $);
  \coordinate (B) at ($ (x0)!\k/\n!(x1) $);
  \draw[hop]
    ($(A)+(0,-\botShift)$) .. controls ($(A)+(0,-\Hdown-\botShift)$) and ($(B)+(0,-\Hdown-\botShift)$) .. ($(B)+(0,-\botShift)$);
}

\draw[decorate,decoration={brace,amplitude=5pt,mirror}]
  ($(x0)+(0,\braceY)$) -- ($(x1)+(0,\braceY)$);

\node[lab,below=4pt,align=center] at ($(x0)!0.5!(x1)+(0,\braceY)$)
{$\displaystyle \sum_{i=1}^{n}\bigl|x_i-x_{i-1}\bigr|^{p}
= n\left|\frac{x_1-x_0}{n}\right|^{p} \longrightarrow 0 \quad \text{as} \quad n \rightarrow +\infty$};

\end{scope}

\end{tikzpicture}
\vspace{6pt}
\caption{$|x-y|^p$ is a quasimetric on $\mathbb{R}$ for $p > 1$, its topology equivalent to the Euclidean one. However, refining partitions causes distances to collapse, so it admits no good notion of an induced intrinsic metric.}
\end{figure}

The goal of the present paper is thus to provide a general construction $D_\phi$ that takes as input \emph{any} valuation $\phi$ on $\mathcal{K}^n$ and produces a bona fide metric under reasonably general hypotheses.

To this end then, we make the following definition.

\begin{defn}\label{Dphi}
    Let $\phi \in \textup{Val}$ (see Section \ref{2} for notation). We write 
    $$D_\phi(K, L) := \inf \left \{\sum_{i = 1}^N |\phi(K_{i-1}) - \phi(K_i)|  : \{K_i\}_{i = 0}^N  \text{ is a comparable chain between } K \text{ and } L  \right \}$$
    for all $K, L \in \mathcal{K}^n$; here a comparable chain between $K$ and $L \in \mathcal{K}^n$ is a collection $\{K_i\}_{i = 0}^N \subseteq \mathcal{K}^n$ satisfying 
    $$K_0 = K, \quad K_N = L, \quad \text{and} \quad K_{i-1} \subseteq K_i \quad \text{or} \quad K_i \subseteq K_{i-1} \quad \text{for all } i \in \{1, \dots, N\}.$$
\end{defn}

Similar constructions have appeared in lattice-theoretic literature, but to the authors' knowledge, a definition like Definition \ref{Dphi} has not been considered in the convex-geometric setting.

We moreover generalize the intrinsic volume deviations to \emph{valuation deviations}. 

\begin{defn}\label{val-dev}
    Let $\phi \in \textup{Val}$. We define 
    \begin{align*}\Delta_\phi(K,L) &:= \phi(K) + \phi(L) - 2\phi(K \cap L) \quad \text{for all } K, L \in \mathcal{K}^n \end{align*} and \begin{align*} \rho_\phi(K,L) &:= 2\phi(\textup{conv}(K \cup L)) - \phi(K) - \phi(L) \quad \text{for all } K, L \in \mathcal{K}^n,\end{align*}
    both termed valuation deviations. The two naturally agree when $K, L \in \mathcal{K}^n$ are such that $K \cup L$ is convex. 
    
\end{defn}

In Section \ref{2} then, we offer a targeted survey of the valuation theory needed to work with Definitions \ref{Dphi} and \ref{val-dev}. We show in Section \ref{3} some basic properties of $D_\phi$, in particular that $D_\phi$ generalizes $\Delta_{\textup{vol}}$ and $\rho_w$, i.e.
$$D_{\textup{vol}}(K,L) = \Delta_\textup{vol}(K,L) \quad \text{for all } K, L \in \mathcal{K}^n$$
when $n \geq 2$ and 
$$D_{w}(K,L) = \rho_w(K,L) \quad \text{for all } K, L \in \mathcal{K}^n$$
always --- in this capacity the $D_\phi$ construction differs from say the metrics $\{\delta_j\}_{j = 1}^n$ of Besau and Hoehner, recently introduced in \cite{besau2024}.

Following this, in Section \ref{4}, we prove the following result giving broad conditions for when $D_\phi$ is a bona fide metric.

\begin{theoremA}\mynameis{A}\label{metricity-conditions}
Let $\phi \in \textup{Val}$ be monotone. We have that $D_\phi$ is nondegenerate (and hence a metric)
\begin{enumerate}[label=(\roman*)]
  \item on $\mathcal{K}^n_n$ when $\phi$ has a nontrivial $n$-homogeneous component, i.e. $\phi_n \not \equiv 0$, 
  \item on $\mathcal{K}^n$ when, in the context of the representation
  \[
    \phi_1(K)= nV(K,G_1, \dots, G_{n-1}) = \int_{S^{n-1}} h_K(u) \, dS(G_1, \dots, G_{n-1}; u) \quad \text{for all } K \in \mathcal{K}^n
  \]
  given some fixed $G_1, \dots, G_{n-1} \in \mathcal{K}^n $, it is that the mixed area measure $\mu := S(G_1, \dots, G_{n-1}; \cdot)$ has full support,
    \item on $\mathcal{K}^n_k$ when $\phi \in \textup{Val}^{+,\infty}$ and we may choose nonnegative representing functions $f_j$ in (\ref{grassman-comp}), with at least one $f_k>0$ on $\Gr_{n,k}$ for some $k\in\{1,\dots,n\}$, and 
  \item on $\mathcal{K}^n$ when $\phi = \textup{vol}(G + K)$ for some $G \in \mathcal{K}^n_n$.
\end{enumerate}
\end{theoremA}

As a corollary of this result, we obtain that every nonconstant, monotone, $\textup{SO}(n)$-invariant $\phi \in \textup{Val}$ has $D_\phi$ a metric.

\begin{coroB}\mynameis{B}\label{intrinsic-volumes-give-metrics}
        Let $\phi \in \textup{Val}$ be monotone, nonconstant, and $\textup{SO}(n)$-invariant. Then $D_\phi$ is a metric on $\mathcal{K}^n_k$ for some $k \in \{0, \dots, n\}$.
        In particular, this result holds when $$\phi = \sum_{j = 0}^n c_j \textup{V}_j, \quad \{c_j\}_{j = 0}^n \subseteq \mathbb{R}_{\geq 0}, \quad \text{and} \quad c_k > 0 \quad  \text{for some } k \in \{1, \dots, n\},$$
        and hence for all the intrinsic volumes $\{\textup{V}_j\}_{j = 1}^n$ individually. 
\end{coroB}

We then move to  rule out $\Delta_\phi$ and $\rho_\phi$ as candidates for a closed form of $D_\phi$, except in the extremal homogeneous cases. In the process of doing so, we generalize work in \cite{Besau_2019}.

\begin{theoremC}\mynameis{C}\label{closed-forms}
        Let $\phi \in \textup{Val}$ be monotone, and write 
    $\phi(K) = \sum_{i = 0}^n \phi_i(K)$ for all $K \in \mathcal{K}^n,$ where
    $\phi_i \in \textup{Val}_i$. If $\phi_i \not \equiv 0$ for any $i \in \{0, \dots, n-1\}$, then there exists $K, L, M \in \mathcal{K}^n$ with 
     $$\Delta_\phi(K,L) > \Delta_\phi(K,M) + \Delta_\phi(M,L).$$
     Similarly, if $\phi_i \not \equiv 0$ for any $i \in \{2, \dots, n\}$, there exists $K, L, M \in \mathcal{K}^n$ with 
     $$\rho_\phi(K,L) > \rho_\phi(K,M) + \rho_\phi(M,L).$$
     Hence $D_\phi = \Delta_\phi$ implies $\phi \in \textup{Val}_n$, and $D_\phi = \rho_\phi$ forces $\phi \in \textup{Val}_1 \oplus \textup{Val}_0$.
\end{theoremC}

Moving on, we consider in Section \ref{5} some geometric aspects of $D_\phi$, assured now that there is an abundance of $\phi$ for which $D_\phi$ is a metric. To start, we show that if $D_\phi$ is a metric, it is automatically intrinsic.

\begin{theoremD}\mynameis{D}\label{intrinsic}
    Let $\phi \in \textup{Val}$ be monotone. If $D_\phi$ is a metric on $\mathcal{K}^n_k$ for $k \in \{0, \dots, n\}$, then it is intrinsic/length.
\end{theoremD}

Following this, we provide a necessary and sufficient condition for the $D_\phi$-balls to be bounded in the Hausdorff sense. 

\begin{theoremE}\mynameis{E}\label{bounded-balls}
    Let $\phi \in \textup{Val}$ be monotone, and write 
    $$\phi_1(K) = \int_{S^{n-1}} h_K(u) \, dS(G_1, \dots, G_{n-1}; u) \quad \text{for all } K \in \mathcal{K}^n$$
    for fixed $G_1, \dots, G_{n-1} \in \mathcal{K}^n$. Then the support of the mixed area measure $\mu := S(G_1, \dots, G_{n-1}; \cdot)$ is not contained in any closed hemisphere, i.e. 
    \begin{align}\mu\left(\left \{u \in S^{n-1} : \langle u, v \rangle > 0  \right\}\right) > 0 \quad \text{for all } v \in S^{n-1}, \end{align}
    if and only if for each $K \in \mathcal{K}^n$ and $r > 0$, the ball $B(K,r) := \{L \in \mathcal{K}^n : D_\phi(K, L) \leq r\}$ is Hausdorff bounded, i.e. there is $R \geq 0$ such that 
    $B(K,r) \subseteq \{L \in \mathcal{K}^n : d_\mathcal{H}(K,L) \leq R\}.$
\end{theoremE}

Applying a Blaschke selection theorem argument then, this result immediately percolates to a condition under which $D_\phi$ is proper, and consequently complete and geodesic. 

\begin{coroF}\mynameis{F}\label{geodesicity}
Let $\phi \in \textup{Val}$ be monotone, and assume $D_\phi$ is a metric on $\mathcal{K}^n$. Write
    $$\phi_1(K) = \int_{S^{n-1}} h_K(u) \, dS(G_1, \dots, G_{n-1}) \quad \text{for all } K \in \mathcal{K}^n$$
    and fixed $G_1, \dots, G_{n-1} \in \mathcal{K}^n$. If the mixed area measure $\mu := S(G_1, \dots, G_{n-1}; \cdot)$ is such that its support is not contained within any closed hemisphere, then $D_\phi$ is proper, and consequently complete and geodesic. 
\end{coroF}

Switching gears, in Section \ref{6} we investigate some topological aspects of $D_\phi$: we compare various notions of convergence, and show, for instance, that the $\{D_{\, \textup{V}_j}\}_{j = 1}^n$ are all topologically equivalent to $d_\mathcal{H}$.

\begin{coroG}\mynameis{G}\label{intrinsic-volumes-give-haus-top}
Let $\phi \in \textup{Val}$ be monotone, nonconstant, and $\textup{SO}(n)$-invariant. Then for the same $k \in \{1, \dots n\}$ in Corollary \ref{intrinsic-volumes-give-metrics} for which $D_\phi$ is a metric on $\mathcal{K}^n_k$, $D_\phi$ induces the topology of the Hausdorff distance on $\mathcal{K}^n_k$, i.e.
    $$D_\phi(K_i, K) \rightarrow 0 \quad \text{as} \quad i \rightarrow +\infty \quad \Longleftrightarrow \quad d_\mathcal{H}(K_i, K) \rightarrow 0 \quad \text{as} \quad i \rightarrow +\infty,$$
    where $K \in \mathcal{K}^n_k$ and $\{K_i\} \subseteq \mathcal{K}^n_k$. In particular, each $D_j := D_{\textup{V}_j}$ induces the topology of the Hausdorff distance on $\mathcal{K}^n_{f(j)}$, where $f(j) = j - \delta_{j,1}$.
\end{coroG}

In the process of doing so, we moreover affirmatively resolve an open question of Besau and Hoehner in \cite{besau2024} (Open Question 2.2), which asks if the analogous holds for their metrics $\{\delta_j\}_{j =1 }^n$.

An analogous result holds for the metrics induced by $$\{\mu_G := \textup{vol}(G + \cdot) : G \in \mathcal{K}^n_n\},$$ which become recurrent examples. Indeed, Corollary \ref{geodesicity} shows them to be geodesic --- this allows us to apply them to a question from the native setting of $d_G$: that of toric singularity types.

\begin{theoremH}\mynameis{H}\label{geo-bilip-to-dS}
    Define 
        \begin{align*}\mathcal{D}(K,L) := \inf \left \{
        \sum_{i = 1}^N |\textup{M}_\Delta(K_{i-1}) - \textup{M}_\Delta(K_i)| : \{K_i\}_{i = 0}^N \subseteq \mathcal{K}(\Delta) \text{ is a comparable chain between } K \text{ and } L \right \}\end{align*}
        for each $K, L \in \mathcal{K}(\Delta)$, i.e. $\mathcal{D}$ is $D_{\textup{M}_\Delta}$ with its defining chains restricted to $\mathcal{K}(\Delta)$, where $\textup{M}_\Delta(K) := \sum_{j = 0}^n \textup{V}_j(\Delta, K)$ for all $K \in \mathcal{K}^n$. Then the pushforward metric $\mathcal{D}_\ast$ on $\mathcal{S}_{\textup{tor}}(\mathbb{CP}^n, \omega_{FS})$, that is $$\mathcal{D}_\ast := \frac{n!}{2^n(n + 1)} (\Phi)_\ast \mathcal{D}, \quad \Phi : \mathcal{K}(\Delta) \rightarrow \mathcal{S}_{\textup{tor}}(\mathbb{CP}^n, \omega_{FS}), \quad K \mapsto [\phi_K],$$
        is geodesic and is moreover bilipschitz with the $d_\mathcal{S}$ of Darvas, Di Nezza and Lu.
\end{theoremH}

\subsection*{Acknowledgments}

The authors are indebted to Tam\'as Darvas for comments which were instrumental in the development of the current paper. The first author is also thankful for conversations with Nima Hoda and Kasso Okoudjou.

\textbf{Statement on AI usage.} Various large language model (LLM) tools were utilized during the research process for literature searches, assistance with minor technical and computational details, and manipulation of TikZ code.

Importantly, \emph{no generative AI tools were used to write the manuscript}. The central construction of the paper ($D_\phi$), as well as the key ideas underlying the proofs of all major theorems, are due to the authors. 

We maintain ethical reservations about disclosing which particular models were used, given how the preeminent AI companies have engaged with mathematics largely as an avenue for advertisement. Such information is nonetheless available upon request.

\section{Relevant background on valuation theory}\label{2}

\subsection{Valuations and mixed volumes}

Let $n \geq 1$, and write $\mathcal{K}^n$ for the set of nonempty compact convex subsets of $\mathbb{R}^n$. For $j \in \{0, \dots, n\}$, take $\mathcal{K}^n_j$ to denote the subset of $\mathcal{K}^n$ consisting of sets $K$ with $\dim(K) \geq j$, i.e. those sets $K$ with a $(k \geq j)$-dimensional affine hull. In this notation
$$\mathcal{K}^n = \mathcal{K}^n_0 \supset \mathcal{K}^n_1 \supset \cdots \supset \mathcal{K}^n_{n-1} \supset \mathcal{K}^n_n =: \{\text{convex bodies in } \mathbb{R}^n\};$$
for $S \subseteq \mathbb{R}^n$ we will also write $\mathcal{K}(S)$ for $\{K \in \mathcal{K}^n : K \subseteq S\}$, and in the particular case where $H$ is a $j$-dimensional linear subspace of $\mathbb{R}^n$, $\mathcal{K}_j(H)$ will denote $\mathcal{K}^n_j \cap \mathcal{K}(H)$.

The ambient topology on $\mathcal{K}^n$ used throughout the paper is the one induced by the Hausdorff distance (\ref{haus-dist}), inasmuch as generic continuity is meant in the Hausdorff sense. We recall now the definition of a valuation on $\mathcal{K}^n$.
\begin{defn}
    A (real-valued) valuation on $\mathcal{K}^n$ is a mapping $\phi : \mathcal{K}^n \rightarrow \mathbb{R}$ satisfying 
    \begin{align}\label{val-prop}\phi(K) + \phi(L) = \phi(K \cup L) + \phi(K \cap L) \quad \text{for all } K, L \in \mathcal{K}^n,\end{align}
    whenever $K \cup L \in \mathcal{K}^n$. Conventionally, the domain of $\phi$ is expanded by the assignment $\phi(\varnothing) = 0$.

\begin{figure}[H]
\centering

\newcommand{\IconU}{\vcenter{\hbox{%
\begin{tikzpicture}[scale=0.54,line cap=round,line join=round]
\tikzset{
  outline/.style={black!100,line width=1pt},
  Kfill/.style={fill=gray!9},
  Lfill/.style={fill=gray!9},
  Ifill/.style={fill=gray!9},
  lab/.style={black!100,font=\scriptsize},
}
\path[use as bounding box] (-2.35,-2.55) rectangle (2.35,2.35);
\begin{scope}[shift={(0.4,-0.9)}]
\fill[Kfill] (-2.4,-0.4)--(0.2,-0.1)--(0.6,1.4)--(-0.6,2.2)--(-2.2,1.2)--cycle;
\fill[Lfill] (-0.8,-0.8)--(1.6,0.0)--(1.4,1.8)--(0.0,2.6)--(-1.2,1.6)--cycle;
\begin{scope}
\clip (-2.4,-0.4)--(0.2,-0.1)--(0.6,1.4)--(-0.6,2.2)--(-2.2,1.2)--cycle;
\fill[Ifill] (-0.8,-0.8)--(1.6,0.0)--(1.4,1.8)--(0.0,2.6)--(-1.2,1.6)--cycle;
\end{scope}
\draw[outline] (-2.4,-0.4)--(0.2,-0.1)--(0.6,1.4)--(-0.6,2.2)--(-2.2,1.2)--cycle;
\draw[outline] (-0.8,-0.8)--(1.6,0.0)--(1.4,1.8)--(0.0,2.6)--(-1.2,1.6)--cycle;
\end{scope}
\node[lab] at (0,-2.25) {$K\cup L$};
\end{tikzpicture}}}}

\newcommand{\IconI}{\vcenter{\hbox{%
\begin{tikzpicture}[scale=0.54,line cap=round,line join=round]
\tikzset{
  outline/.style={black!100,line width=1pt},
  Ifill/.style={fill=gray!9},
  lab/.style={black!100,font=\scriptsize},
}
\path[use as bounding box] (-2.35,-2.55) rectangle (2.35,2.35);
\begin{scope}[shift={(0.4,-0.9)}]
\begin{scope}
\clip (-2.4,-0.4)--(0.2,-0.1)--(0.6,1.4)--(-0.6,2.2)--(-2.2,1.2)--cycle;
\fill[Ifill] (-0.8,-0.8)--(1.6,0.0)--(1.4,1.8)--(0.0,2.6)--(-1.2,1.6)--cycle;
\end{scope}
\draw[outline] (-2.4,-0.4)--(0.2,-0.1)--(0.6,1.4)--(-0.6,2.2)--(-2.2,1.2)--cycle;
\draw[outline] (-0.8,-0.8)--(1.6,0.0)--(1.4,1.8)--(0.0,2.6)--(-1.2,1.6)--cycle;
\end{scope}
\node[lab] at (0,-2.25) {$K\cap L$};
\end{tikzpicture}}}}

\newcommand{\IconK}{\vcenter{\hbox{%
\begin{tikzpicture}[scale=0.54,line cap=round,line join=round]
\tikzset{
  outline/.style={black!100,line width=1pt},
  Kfill/.style={fill=gray!9},
  lab/.style={black!100,font=\scriptsize},
}
\path[use as bounding box] (-1.95,-2.55) rectangle (1.95,2.35);
\begin{scope}[shift={(0.9,-0.9)}]
\fill[Kfill] (-2.4,-0.4)--(0.2,-0.1)--(0.6,1.4)--(-0.6,2.2)--(-2.2,1.2)--cycle;
\draw[outline] (-2.4,-0.4)--(0.2,-0.1)--(0.6,1.4)--(-0.6,2.2)--(-2.2,1.2)--cycle;
\end{scope}
\node[lab] at (0,-2.25) {$K$};
\end{tikzpicture}}}}

\newcommand{\IconL}{\vcenter{\hbox{%
\begin{tikzpicture}[scale=0.54,line cap=round,line join=round]
\tikzset{
  outline/.style={black!100,line width=1pt},
  Lfill/.style={fill=gray!9},
  lab/.style={black!100,font=\scriptsize},
}
\path[use as bounding box] (-1.95,-2.55) rectangle (1.95,2.35);
\begin{scope}[shift={(-0.2,-0.9)}]
\fill[Lfill] (-0.8,-0.8)--(1.6,0.0)--(1.4,1.8)--(0.0,2.6)--(-1.2,1.6)--cycle;
\draw[outline] (-0.8,-0.8)--(1.6,0.0)--(1.4,1.8)--(0.0,2.6)--(-1.2,1.6)--cycle;
\end{scope}
\node[lab] at (0,-2.25) {$L$};
\end{tikzpicture}}}}

\[
\mathrm{vol}\!\left(\vphantom{\IconU}\IconK\right)+\mathrm{vol}\!\left(\vphantom{\IconU}\IconL\right)
=\mathrm{vol}\!\left(\IconU\right)+\mathrm{vol}\!\left(\vphantom{\IconU}\IconI\right)
\]

\caption{Volume/the Lebesgue measure is the prototypical valuation on $\mathcal{K}^n$, as is visually clear.}

\end{figure}
\end{defn}

Stipulating $K \cup L \in \mathcal{K}^n$ is nonessential in the sense that every continuous valuation on $\mathcal{K}^n$ might be extended to a valuation on the convex ring $\mathcal{R}^n$ --- here $\mathcal{R}^n$ is the set of subsets of $\mathbb{R}^n$ that are finite unions of sets in $\mathcal{K}^n$, i.e.
$$\mathcal{R}^n := \left \{\bigcup_{i = 0}^N K_i : \{K_i\}_{i = 0}^N  \subseteq \mathcal{K}^n \right \}.$$
For a discussion of the relevant extension theorem, see for instance \cite{schneider2014} Theorem 6.2.5.

If $G$ is a group acting on $\mathcal{K}^n$, we say a valuation $\phi$ on $\mathcal{K}^n$ is \textit{G-invariant} if $$\phi(gK) = \phi(K) \quad  \text{for all } g \in G, \quad K \in \mathcal{K}^n,$$ and in the specific case of $G = \mathbb{R}^n$ with its action on $\mathcal{K}^n$ given by translation, we say $\phi$ is \textit{translation-invariant}; we will take $\textup{Val}$ to denote the space of translation invariant, continuous valuations on $\mathcal{K}^n$. An interesting class of such valuations is provided by the \emph{mixed volumes}.

\begin{prop}[\cite{schneider2014} Theorem 5.1.7]\label{mixed-volume}
    There exists a nonnegative, multilinear, symmetric function $\textup{V} : (\mathcal{K}^n)^n \rightarrow \mathbb{R}$ such that, for any $K_1, \dots, K_n \in \mathcal{K}^n$ and $t_1, \dots, t_n \geq 0$, we have that
    \begin{align}\label{mixed-vol-eq}\textup{vol}(t_1K_1 + \dots + t_n K_n) = \sum_{i_1, \dots, i_n = 1}^n t_{i_1} \dots t_{i_n} \textup{V}(K_{i_1}, \dots, K_{i_n}).\end{align}
    The coefficient $\textup{V}(K_{1}, \dots, K_{n})$ is termed the mixed volume of $K_1, \dots, K_n$. 
\end{prop}

Clearly, the mixed volumes generalize the intrinsic volumes --- one sees that (\ref{steiner-form}) is a particular example of (\ref{mixed-vol-eq}). Indeed, reorganizing, we deduce that for all $K \in \mathcal{K}^n$, the $j$-th intrinsic volume $\textup{V}_j$ satisfies \begin{align}\label{intrinsic-mixed}
\textup{V}_j(K)
= \frac{\binom n j}{\kappa_{n -j}}\textup{V}_j(B^n,K)
:= \frac{\binom n j}{\kappa_{n -j}} \textup{V}(B^n[n-j], K[j]) :=  \frac{\binom n j}{\kappa_{n -j}}\textup{V} (
\underbrace{B^n,\ldots,B^n}_{n-j\ \text{times}},
\underbrace{K,\ldots,K}_{j\ \text{times}}
 ).
\end{align}
The intrinsic volumes are in fact continuous, $(\textup{SO}(n) \ltimes \mathbb{R}^n)$-invariant valuations on $\mathcal{K}^n$: as noted earlier, for convex body $K \in \mathcal{K}^n$, one has that  $\textup{V}_n(K) = \textup{vol}(K)$, that $\textup{V}_{n-1}(K)$ is half the surface area $s(K) = \mathcal{H}^{n-1}(\partial K)$, that $\textup{V}_1(K)$ is proportional to the mean width $w(K)$, and $\textup{V}_0(K)$ is the Euler characteristic $\chi(K) = 1$.

\begin{figure}[H]
\centering
\begin{tikzpicture}[scale=0.54,line cap=round,line join=round]
\tikzset{
  outline/.style={black!100,line width=1.00pt},
  Kfill/.style={fill=gray!18,opacity=.5},
  facefaint/.style={fill=gray!18,opacity=.0},
  perim/.style={black!100,line width=1.50pt},
  edge/.style={black!100,line width=0.95pt},
  vtx/.style={circle,fill=black!100,inner sep=1.35pt},
  lab/.style={black!100,font=\scriptsize},
}

\def\pan{8.4}
\def\labx{-0.5}

\newcommand{\MakePoly}{
  \coordinate (A) at (-2.4,-0.4);
  \coordinate (B) at (0.2,-0.1);
  \coordinate (C) at (0.6,1.4);
  \coordinate (D) at (-0.6,2.2);
  \coordinate (E) at (-2.2,1.2);
}
\newcommand{\PolyPath}{(A)--(B)--(C)--(D)--(E)--cycle}

\begin{scope}[xshift=0*\pan cm]
  \path[use as bounding box] (-2.55,-2.65) rectangle (2.55,2.45);
  \begin{scope}[shift={(0.4,-0.9)}]
    \MakePoly
    \fill[Kfill] \PolyPath;
    \draw[outline] \PolyPath;
  \end{scope}
  \node[lab,align=center] at (\labx,-2.35) {$V_2(K)=\mathrm{area}(K)$};
\end{scope}

\begin{scope}[xshift=1*\pan cm]
  \path[use as bounding box] (-2.55,-2.65) rectangle (2.55,2.45);
  \begin{scope}[shift={(0.4,-0.9)}]
    \MakePoly
    \draw[perim] \PolyPath;
  \end{scope}
  \node[lab,align=center] at (\labx,-2.35) {$V_1(K)=\nicefrac{\mathrm{perim}(K)}{2}$};
\end{scope}

\begin{scope}[xshift=2*\pan cm]
  \path[use as bounding box] (-2.55,-2.65) rectangle (2.55,2.45);
  \begin{scope}[shift={(0.4,-0.9)}]
    \MakePoly
    \draw[edge] \PolyPath;
    \node[vtx] at (A) {};
    \node[vtx] at (B) {};
    \node[vtx] at (C) {};
    \node[vtx] at (D) {};
    \node[vtx] at (E) {};
    \node[lab,anchor=west] at (2.05,1.75) {$V=5$};
    \node[lab,anchor=west] at (2.05,1.30) {$E=5$};
    \node[lab,anchor=west] at (2.05,0.85) {$F=1$};
  \end{scope}
  \node[lab,align=center] at (\labx,-2.35) {$V_0(K)=\chi(K)=1$};
\end{scope}

\end{tikzpicture}
\caption{Intrinsic volumes for a planar convex polygon $K$.}
\end{figure}

More generally, we can fix $n - j$ arguments of $\textup{V}$ to get a $j$-homogeneous functional $\mathcal{K}^n \rightarrow \mathbb{R}$. It will be itself translation-invariant, as $\textup{V}$ is invariant under independent translations of its parameters, i.e.,
$$\textup{V}(K_1, \dots, K_n) = \textup{V}(K_1 + x_1, \dots, K_n + x_n) \quad \text{for all } K_1, \dots, K_n \in \mathcal{K}^n, \quad x_1, \dots, x_n \in \mathbb{R}^n;$$
see for instance \cite{hug2020lectures} Theorem 3.9. One can then show it to be a valuation on $\mathcal{K}^n$ by working with the left side of (\ref{mixed-vol-eq}), which also lends it continuity --- hence it lies in $\textup{Val}$.

\subsection{Homogeneity and monotonicity}

A natural question thus arises in terms of how adequately mixed volumes account for \emph{all} valuations in $\textup{Val}$, especially given the classic result of Hadwiger that intrinsic volumes alone are sufficient in the presence of $\textup{SO}(n)$-invariance.

\begin{prop}[Hadwiger, \cite{https://doi.org/10.1112/S0025579300014625} Theorem 2.1]\label{hadwiger-char}
    Let $\phi \in \textup{Val}$ be $\textup{SO}(n)$-invariant. Then there exists constants $c_0, \dots, c_n \in \mathbb{R}$ such that 
    $$\phi(K) = \sum_{j = 0}^n c_j \textup{V}_j(K), \quad \text{for all } K \in \mathcal{K}^n.$$
\end{prop}

This is that the space of continuous valuations on $\mathcal{K}^n$ invariant under (proper) rigid motions is finite-dimensional and is spanned by the intrinsic volumes. As one may guess, the situation for $\textup{Val}$ is less simple --- it is decidedly infinite-dimensional --- though one may recover analogs of two ``facets'' of Proposition \ref{hadwiger-char}. 

The first is Alesker's irreducibility theorem, \cite{ca8d7fa2304c4b91a128f4574c71e5ed} Theorem 1.3, which implies these mixed volume functionals span a dense subspace of $\textup{Val}$ --- for how $\textup{Val}$ is topologized, see (\ref{banach}). As
$$\textup{vol}(G + K) = \sum_{j = 0}^n \binom n j \, \textup{V}_j(G, K) \quad \text{for all } K \in \mathcal{K}^n,$$
we can restate this result as saying the class of valuations $\{\textup{vol}(G + K) : G \in \mathcal{K}^n\}$ spans a dense subspace of $\textup{Val}$; we will return to this formulation later. The second, more pertinent to us, is the following. 

\begin{prop}[McMullen, \cite{alesker2014} Corollary 1.1.7]\label{mcmullen-decomp}

We have 
\begin{align}\label{mcmullen-decomp-eq}\textup{Val} = \bigoplus_{j = 0}^n \textup{Val}_j,\end{align}
where $\textup{Val}_j$ is the subspace of $\textup{Val}$ consisting of $j$-homogeneous valuations, i.e. 
    \[\textup{Val}_j := \big\{\phi \in \textup{Val} : \phi(\lambda K) = \lambda ^j\phi(K) \text{ for all } K \in \mathcal{K}^n, \, \lambda \geq 0\big\}.\] In particular, given $\phi \in \textup{Val}$, one may write $\phi = \sum_{j = 0}^n c_j \phi_j$ for some constants $c_0, \dots, c_n \in \mathbb{R}$. 
\end{prop}

This result is analogous to Proposition \ref{hadwiger-char} specifically because $\textup{V}_j$ is $j$-homogeneous, again owing to the multilinearity of the function $\textup{V}$. We emphasize, however, that the spaces $\textup{Val}_j$ for $0<j<n$ are far richer and significantly less well-behaved than mere multiples of $\textup{V}_j$.

As to make working with $\textup{Val}$ tractable, it is useful to introduce some properties of valuations that are compatible with the decomposition (\ref{mcmullen-decomp-eq}), first of which is \emph{monotonicity}.

\begin{defn}
    Let $\phi \in \textup{Val}$. We say $\phi$ is monotone if 
    $$\phi(K) \leq \phi(L) \quad \text{for all } K, L \in \mathcal{K}^n \cup \{\varnothing\} \quad \text{with} \quad  K \subseteq L.$$

\end{defn}

Monotonicity is a broadly satisfied and yet deeply regularizing condition. As to the broadness in particular, we note mixed volumes can be shown to be monotone with respect to inclusion in each of their parameters, i.e. 
\begin{align*}\textup{V}(K_1,\dots,K_n)\le \textup{V}(L_1,\dots,L_n) \quad \text{when} \quad  \{K_i\}_{j = 1}^n, \{L_i\}_{j = 1}^n \subseteq \mathcal{K}^n, \quad K_j \subseteq L_j \quad \text{for all } j \in \{1, \dots, n\},
\end{align*}
and hence the earlier mixed volume functionals are monotone. This is supplemented by the fact that monotonicity is indeed respectful of the grading in (\ref{mcmullen-decomp-eq}), i.e. the homogeneous components of a monotone $\phi \in \textup{Val}$ are themselves monotone, a deeply nontrivial result of Bernig and Fu in \cite{bernig2010hermitianintegralgeometry}.

Speaking then to the regularizing properties of monotonicity, we note the following result.

\begin{prop}[Firey, \cite{firey1976functional}]\label{firey-rep}
    Let $\psi \in \textup{Val}_1$ be monotone. Then there exists $G_1, \dots, G_{n-1} \in \mathcal{K}^n$ such that
    \begin{align}\label{firey-rep-eq}\psi(K) = \textup{V}(K, G_1, \dots, G_{n-1}) = \frac{1}{n}\int_{S^{n-1}} h_K(u) \, dS(G_1, \dots, G_{n-1}; u) \quad \text{for all } K \in \mathcal{K}^n.\end{align}
    The measure $S(G_1, \dots, G_{n-1}; \cdot)$ is termed the mixed area measure of $G_1, \dots, G_{n-1}$, a finite (nonnegative) Borel measure uniquely determined by (\ref{firey-rep-eq}). 
\end{prop}

Here $h_K$ here is the support function associated to $K$ --- we refer the reader to \cite{schneider2014} Chapter 1 for the basics of the support function calculus. We similarly refer to \cite{schneider2014} Chapter 5 for the construction of the mixed area measure, though we do recount its basic properties now.

\begin{prop}\label{mixed-area-prop}
    Let $G_1, \dots, G_{n-1} \in \mathcal{K}^n$. Then the mixed area measure $S(G_1, \dots, G_{n-1}; \cdot)$ satisfies
    \begin{enumerate}[label=(\roman*)]
        \item $S(G_1, \dots, G_{n-1}; \cdot)$ is symmetric and multilinear in $G_1, \dots, G_{n-1}$,
        \item $S(G_1, \dots, G_{n-1}; \cdot)$ is nonnegative,
        \item $S(G_1, \dots, G_{n-1}; \cdot)$ is invariant under translation of any argument, and
        \item $S(G_1, \dots, G_{n-1};\cdot)$ has its center of mass at the origin, i.e. 
        $$\int_{S^{n-1}} \langle u, v \rangle \, dS(G_1, \dots, G_{n-1}; u) =0 \quad \text{for all } v \in S^{n-1}.$$
    \end{enumerate}
\end{prop}

It is standard to have a $\nicefrac{1}{n}$ term written with the integral in (\ref{firey-rep-eq}), but Proposition \ref{mixed-area-prop} assures us
$$\alpha S(G_1, \dots, G_{n-1}; \cdot) = S(\alpha G_1, \dots, G_{n-1}; \cdot) \quad \text{for all } \alpha \geq 0, \quad G_1, \dots, G_{n-1} \in \mathcal{K}^n,$$
so, in practice, we will freely absorb the $\nicefrac{1}{n}$ factor to reduce notational bloat. Occasionally, it will also be useful to make that same integral representation more explicit, i.e. write it as instead integrating against the normal surface measure $\omega$ on $S^{n-1}$. To this end we note one has 
\begin{align}\label{firey-rep-eq-2}\textup{V}(K, G_1, \dots, G_{n-1}) = \frac{1}{n} \int_{S^{n-1}} h_K(u) \textup{D}(D^2 h_{G_1}(u), \dots, D^2 h_{G_{n-1}}(u)) \, d\omega(u) \quad \text{for all } K \in \mathcal{K}^n,\end{align}
given any $G_{1}, \dots, G_{n-1} \in \mathcal{K}^n$. $\textup{D}$ here denotes the mixed discriminant --- see for instance \cite{Shenfeld_2019} --- and $D^2 h_{G_i}(u)$ is the restriction of Hessian of $h_{G_i}$ to the tangent space of $S^{n-1}$ at the point $u$.

\subsection{Smooth and even valuations}

We move now then to two other properties, themselves also compatible with McMullen's decomposition, that will be useful to us inasmuch as they grant us additional useful integral representations. The first is \emph{smoothness}.

\begin{defn}\label{smooth}
    Let $\phi \in \textup{Val}$. We say $\phi$ is smooth if the map $\textup{GL}(n) \rightarrow \textup{Val}$ given by $g \mapsto g^\ast \phi$ is infinitely differentiable, i.e. $\phi$ is a smooth vector under the action of $\textup{GL}(n)$ on $\textup{Val}$.
\end{defn}
We will use $\textup{Val}^\infty$ to denote the subspace of $\textup{Val}$ consisting of smooth valuations, and in general an ``$\infty$'' in superscript will indicate we are intersecting with $\textup{Val}^\infty$, e.g. $\textup{Val}_j^\infty$ denotes $j$-homogeneous smooth valuations in $\textup{Val}$. In this notation we have 
$$\textup{Val}^\infty = \bigoplus_{j = 0}^n \textup{Val}_j^\infty.$$

$\textup{Val}^\infty$ is moreover dense in $\textup{Val}$ --- note the topology on the later is that of uniform convergence on compact subsets of $\mathcal{K}^n$. A consequence of Proposition \ref{mcmullen-decomp} is that $\textup{Val}$ with this topology is Banach space, its norm given by
\begin{align}\label{banach}||\phi|| := \sup_{K \subseteq B^n} |\phi(K)| \quad \text{for all } \phi \in \textup{Val},\end{align}
and so smoothness in the context of Definition \ref{smooth} is intelligible. As was the case then with monotonicity, we remark the intrinsic volumes $\textup{V}_j$ are smooth. Smoothness is most useful to us when the valuations we consider are additionally \textit{even}, defined now.

\begin{defn}
    Let $\phi \in \textup{Val}$. We say $\phi$ is even if $\phi(K) = \phi(-K)$ for all $\mathcal{K}^n$. Conversely, $\phi$ is odd if $\phi(K) = -\phi(-K)$.
\end{defn}

Note that, analogous to what one can do with functions $f : \mathbb{R} \rightarrow \mathbb{R}$, we might split $\phi \in \textup{Val}$ into even and odd parts via the decomposition
$$\phi(K) = \frac{\phi(K) + \phi(-K)}{2} + \frac{\phi(K) - \phi(-K)}{2}.$$
This decomposition is respectful of homogeneity and compatible with smoothness, in that we have 
\begin{align}\label{compatibility}\textup{Val} = \bigoplus_{j = 0}^n \textup{Val}^+_j \oplus \textup{Val}_j^- \quad \text{and} \quad \textup{Val}^\infty = \bigoplus_{j = 0}^n \textup{Val}^{+,\infty}_j \oplus \textup{Val}_j^{-,\infty};\end{align}
of course, the ``$+$'' in the superscript denotes taking the intersection with the subspace of $\textup{Val}$ consisting of even valuations, and dually ``$-$'' in the superscript denotes taking the intersection with the subspace of odd valuations. 

Intrinsic volumes again provide a good example for us here, as they might easily be seen to be even. Indeed, given that we have for any $K \in \mathcal{K}^n$ that $$\textup{vol}(K + tB^n)  =  \textup{vol}((-K) + tB^n)$$
as polynomials in $t \geq 0$, it follows their coefficients, the intrinsic volumes of $K$ and $-K$ respectively, must agree. Alternatively, one might deduce this same fact merely by inspecting the Kubota formula 
\begin{align}\label{kubota}\textup{V}_j(K) = \binom n j \frac{\kappa_n}{\kappa_j \kappa_{n-j}} \int_{\Gr_{n,j}} \textup{vol}_j(\pi_H(K)) \, d\nu_j(H) \quad \text{for all } K \in \mathcal{K}^n, \quad j \in \{0, \dots, n\};\end{align}
here $\Gr_{n,j}$ is the Grassmannian of $j$-dimensional subspaces of $\mathbb{R}^n$, $\nu_j$ is the uniquely determined Haar probability measure on $\Gr_{n,j}$, and $\pi_H(K)$ denotes the orthogonal projection of $K$ onto $H \in \Gr_{n,j}$.

\begin{figure}[H]
\centering
\begin{tikzpicture}[scale=0.72,line cap=round,line join=round]
\tikzset{
  outline/.style={black!100,line width=1.00pt},
  Kfill/.style={fill=gray!18,opacity=.5},
  HdashL/.style={black!100,dashed,->,line width=1.00pt},
  HdashR/.style={black!100,dashed,->,line width=1.00pt},
  projseg/.style={black!100,line width=1pt},
  support/.style={black!100,line width=1.00pt},
  faint/.style={black!100,line width=1pt},
  meas/.style={black!100,<->,line width=1.0pt},
}

\def\th{25}
\pgfmathsetmacro{\ux}{cos(\th)}
\pgfmathsetmacro{\uy}{sin(\th)}
\pgfmathsetmacro{\vx}{-\uy}
\pgfmathsetmacro{\vy}{\ux}

\def\aH{0.09}
\def\bH{0.04}
\def\sc{1.75}

\def\Rline{3.45}
\def\Rsup{2.55}
\def\off{1.95}

\def\Kcoords{}
\foreach \t in {0,4,...,356}{
  \pgfmathsetmacro{\theta}{\t*pi/180}
  \pgfmathsetmacro{\h}{1 + \aH*cos(2*\theta r) + \bH*cos(4*\theta r)}
  \pgfmathsetmacro{\hp}{-2*\aH*sin(2*\theta r) - 4*\bH*sin(4*\theta r)}
  \pgfmathsetmacro{\xx}{\sc*(\h*cos(\theta r) - \hp*sin(\theta r))}
  \pgfmathsetmacro{\yy}{\sc*(\h*sin(\theta r) + \hp*cos(\theta r))}
  \xdef\Kcoords{\Kcoords (\xx,\yy)}
}

\pgfmathsetmacro{\thrad}{\th*pi/180}
\pgfmathsetmacro{\hth}{1 + \aH*cos(2*\thrad r) + \bH*cos(4*\thrad r)}
\pgfmathsetmacro{\d}{\sc*\hth}

\pgfmathsetmacro{\hpTh}{-2*\aH*sin(2*\thrad r) - 4*\bH*sin(4*\thrad r)}
\pgfmathsetmacro{\Tx}{\sc*(\hth*cos(\thrad r) - \hpTh*sin(\thrad r))}
\pgfmathsetmacro{\Ty}{\sc*(\hth*sin(\thrad r) + \hpTh*cos(\thrad r))}

\path[use as bounding box] (-4.1,-3.2) rectangle (4.1,3.2);

\fill[Kfill] plot[smooth cycle] coordinates {\Kcoords};
\draw[outline] plot[smooth cycle] coordinates {\Kcoords};

\coordinate (Splus)  at ({ \d*\ux},{ \d*\uy});
\coordinate (Sminus) at ({-\d*\ux},{-\d*\uy});

\draw[support] ($(Splus)  - \Rsup*(\vx,\vy)$) -- ($(Splus)  + \Rsup*(\vx,\vy)$);
\draw[support] ($(Sminus) - \Rsup*(\vx,\vy)$) -- ($(Sminus) + \Rsup*(\vx,\vy)$);

\draw[faint] (\Tx,\Ty) -- (Splus);
\draw[faint] (-\Tx,-\Ty) -- (Sminus);

\fill[black!100] (\Tx,\Ty) circle (1.15pt);
\fill[black!100] (-\Tx,-\Ty) circle (1.15pt);

\draw[HdashL] (Sminus) -- ({-\Rline*\ux},{-\Rline*\uy});
\draw[HdashR] (Splus) -- ({\Rline*\ux},{\Rline*\uy});

\draw[projseg] (Sminus) -- (Splus);

\draw[meas, dashed] ($(Sminus)+\off*(\vx,\vy)$) -- ($(Splus)+\off*(\vx,\vy)$);

\end{tikzpicture}
\caption{The Kubota formula for $\textup{V}_1$ is particularly intuitive: for a fixed $K \in \mathcal{K}^n$, the distance between the opposite supporting hyperplanes of $K$ perpendicular to a direction $u \in S^{n-1}$ agrees with the length of the orthogonal projection of $K$ onto the one-dimensional subspace $\mathrm{span}(u)$, and $\textup{V}_1(K)$ is exactly proportional to the mean width of $K$.}
\end{figure} 

As it turns out, the ability to consider a valuation as (an integral) over a Grassmannian is intimately related to evenness and, to an extent, smoothness. Indeed, it is known for $\psi \in \textup{Val}_j^+$ that there is a continuous mapping $\textup{Kl}_\psi : \Gr_{n,j} \rightarrow \mathbb{R}$, termed the \textit{Klain function}, such that
\begin{align}\label{klain}
\psi(K) = \textup{Kl}_\psi(H)\textup{vol}_j(K) \quad \text{for all } H \in \Gr_{n,j}, \quad K \in \mathcal{K}(H);
\end{align}
 a result of Klain himself confirms that the mapping $\textup{Val}_j^+ \rightarrow \textup{C}(\Gr_{n,j})$ is injective. If one then adds smoothness, an integral representation of the form in (\ref{kubota}) falls out.

\begin{prop}[Alesker, \cite{alesker2003hardlefschetztheoremvaluations}]\label{grassman-rep-prop}
    Let $\psi \in \textup{Val}^{+,\infty}_j$ for $j \in \{0, \dots, n\}$. Then there exists $f \in \textup{C}^\infty(\Gr_{n,j})$ such that 
    \begin{align}\label{grassman-rep}\psi(K) = \int_{\Gr_{n,j}} f(H) \, \textup{vol}_j(\pi_H(K)) \, d\nu_j(H) \quad \text{for all } K \in \mathcal{K}^n.\end{align}
\end{prop}
This result is typically stated for $j \in \{1, \dots, n-2\}$ --- see for instance \cite{weil2017grassmannmeasuresconvexbodies} Theorem 1.3 --- but is true in the generality claimed. 

Proposition \ref{grassman-rep-prop} will be particularly useful given the aforementioned compatibility in (\ref{compatibility}); in particular, for $\phi \in \textup{Val}^{+,\infty}$, we may write
\begin{align}\label{grassman-comp}\phi(K) = \sum_{j = 0}^n \int_{\Gr_{n,j}} f_j(H) \, \textup{vol}_j(\pi_H(K)) \, d\nu_j(H) \quad \text{for all } K \in \mathcal{K}^n\end{align}
where $f_j \in \textup{C}^\infty(\Gr_{n,j})$ for each $j \in \{0, \dots, n\}$. This remark concludes our preliminaries. 

\section{$D_\phi$ and its basic properties}\label{3}

\subsection{Definition and basic properties}

We define now the $D_\phi$ construction and establish some of its basic properties. Note first the definition of a comparable chain. 

\begin{defn}
    Let $K, L \in \mathcal{K}^n$. We say a collection $\{K_i\}_{i = 0}^N \subseteq K^n$ is a comparable chain between $K$ and $L$ if 
    $$K_0 = K, \quad K_N = L, \quad \text{and} \quad K_{i-1} \subseteq K_i \quad \text{or} \quad K_i \subseteq K_{i-1} \quad \text{for all } i \in \{1, \dots, N\}.$$
\end{defn}

We may think of a comparable chain between $K$ and $L \in \mathcal{K}^n$ as the lattice-theoretic analog of a rectifiable path between $K$ and $L$ --- this inspires an intrinsic-metric-like construction, which we call $D_\phi$.

\begin{defn}
    Let $\phi \in \textup{Val}$. We define for all $K,L\in\K^n$, 
    $$D_\phi(K, L) := \inf \left \{\sum_{i = 1}^N \left |\phi(K_{i-1}) - \phi(K_i) \right|  : \{K_i\}_{i = 0}^N  \text{ is a comparable chain between } K \text{ and } L  \right \}$$
\end{defn}

The pertinent infimum always exists, since for any $K$ and $L \in \mathcal{K}^n$ one may consider the comparable chain $\{K, \textup{conv}(K \cup L), L\}$. Moreover, it is clear from the definition that $D_\phi$ is a pseudometric on $\mathcal{K}^n$, but for the sake of completeness, we record this as a proposition.

\begin{prop}
  For $\phi \in \textup{Val}$, $D_\phi$ is a pseudometric on $\mathcal{K}^n$.
  \begin{proof}
      Let $\phi \in \textup{Val}$. We verify the pseudometric axioms, nonnegativity being obvious.
          \begin{enumerate}[label=(\roman*)]
 \item \textit{(Symmetry)} Let $K,L\in\mathcal{K}^n$, and let $\{K_i\}_{i=0}^N$ be a comparable chain between $K$ and $L$.
Define the reversed sequence 
$$\{L_i\}_{i = 0}^N \subseteq \mathcal{K}^n, \quad \text{where} \quad L_i := K_{N-i} \quad \text{for all } i \in \{0, \dots, N\};$$
then $L_0=L$, $L_N=K$, and for each $i \in \{1, \dots, N\}$ we have
\[
L_{i-1}\subseteq L_i \quad  \text{or} \quad  L_i\subseteq L_{i-1}
\quad\Longleftrightarrow\quad
K_{N-i+1}\subseteq K_{N-i} \quad  \text{or} \quad  K_{N-i}\subseteq K_{N-i+1}.
\]
Given that $\{K_i\}_{i = 0}^N$ is a comparable chain, the right-hand side of this expression holds, and thus $\{L_i\}_{i = 0}^N$ is exactly a comparable chain between $L$ and $K$. Moreover, one has
\begin{align*}
\sum_{i=1}^N \left|\phi(L_{i-1})-\phi(L_i)\right|
&=\sum_{i=1}^N \left|\phi(K_{N-i+1})-\phi(K_{N-i})\right| \\
&=\sum_{j=1}^N \left|\phi(K_{j-1})-\phi(K_j)\right|,
\end{align*}
where the last equality follows after reindexing $j=N-i+1$.
Taking the appropriate infima then shows $D_\phi(K,L)=D_\phi(L,K)$, as desired. 
 
\item \textit{(Triangle inequality)} Let now $K, L, M \in \mathcal{K}^n$, and take $\{K_i\}_{i = 0}^N$ and $\{L_j\}_{j = 0}^M$ to be comparable chains between $K$ and $M$ and between $M$ and $L$, respectively. We define the concatenation
$$\{M_k\}_{k = 0}^{N + M}, \quad \text{where}  \quad M_k := \begin{cases}
    K_k & \text{for } k \in \{0, \dots, N\}\\
    L_{k - N} & \text{for } k \in \{N, \dots, N+ M\}
\end{cases};$$
it is immediately clear that $\{M_k\}_{k = 0}^{N + M}$ gives a comparable chain between $K$ and $L$, where moreover
\begin{align*}\sum_{k = 1}^{N + M} \left|\phi(M_{k - 1}) - \phi(M_k)\right| &= \sum_{i = 1}^N \left|\phi(K_{i-1}) - \phi(K_i)\right| + \sum_{k = N + 1}^{N + M} \left|\phi(L_{k - 1 - N}) - \phi(L_{k - N})\right| \\ &= \sum_{i = 1}^N \left|\phi(K_{i-1}) - \phi(K_i)\right| + \sum_{j = 1}^M \left|\phi(L_{j- 1}) - \phi(L_j)\right|. \end{align*}
Taking infima then, one deduces
\begin{align*}
    &D_\phi(K,L) - \sum_{i = 1}^N \left |\phi(K_{i-1}) - \phi(K_i) \right | \leq \sum_{j = 1}^M \left |\phi(L_{j-1}) - \phi(L_j) \right | \\
    \Longrightarrow\quad&D_\phi(K,L) - D_\phi(M,L) \leq \sum_{i = 1}^N \left |\phi(K_{i-1}) - \phi(K_i) \right | \\
    \Longrightarrow\quad&D_\phi(K,L) \leq D_\phi(K,M) + D_\phi(K,L),
\end{align*}
as desired.
     \end{enumerate}
  \end{proof}
\end{prop}

We note that having $\phi$ specifically in $\textup{Val}$ is not essential for establishing that $D_\phi$ is a pseudometric --- that is, $D_\phi$ would still be a pseudometric if $\phi$ were merely an arbitrary valuation on $\mathcal{K}^n$. Nonetheless, taking $\phi \in \textup{Val}$ gives $D_\phi$ some desirable properties, which we now record.

\begin{prop}[Basic properties of $D_\phi$]\label{basic-props}
    Let $\phi \in \textup{Val}$. Then $D_\phi$ satisfies
    \begin{enumerate}[label=(\roman*)]
        \item $$\left |\phi(K) - \phi(L) \right| \leq D_\phi(K,L) \quad \text{for all } K, L \in \mathcal{K}^n,$$
        with equality when $K \subseteq L$ or $L \subseteq K$,
        \item continuity with respect the Hausdorff distance, i.e. if $\{K_i\}_{i = 0}^\infty \subseteq \mathcal{K}^n$ and $K \in \mathcal{K}^n$,
        $$d_\mathcal{H}(K_i, K) \rightarrow 0 \quad \text{as} \quad i \rightarrow +\infty \quad  \Longrightarrow \quad D_\phi(K_i, K) \rightarrow 0 \quad \text{as} \quad i \rightarrow +\infty,$$
        \item invariance under the same subgroup of affine transformations $\phi$ is invariant under, i.e.
        $$D_\phi(gK, gL) = D_\phi(K, L) \quad \text{for all } K,L \in \mathcal{K}^n, \quad g \in G$$
        if $G \leqslant \textup{Aff}(n) := \textup{GL}(n) \ltimes \mathbb{R}^n$ is such that $\phi(gM) = \phi(M)$ for all $M \in \mathcal{K}^n$, $g \in G$. 
        \item and $j$-homogeneity if $\phi \in \textup{Val}_j,$ i.e. 
        $$D_\phi(tK, tL) = t^j D_\phi(K,L) \quad \text{for all } K, L \in \mathcal{K}^n, \quad t \geq 0.$$
    \end{enumerate}
    \begin{proof}
        We start with (i), which is immediate. 
        Indeed, take $K, L \in \mathcal{K}^n$ and let $\{K_i\}_{i = 0}^N$ be a comparable chain between $K$ and $L$; applying the (Euclidean) triangle inequality, one gets  
        \begin{align*}& \left|\phi(K) - \phi(L)\right| \leq \sum_{i = 1}^N \left|\phi(K_{i-1}) - \phi(K_i) \right| \\
        \Longrightarrow\quad&\left |\phi(K) - \phi(L) \right | \leq D_\phi(K,L) 
        \end{align*}
        by taking an infimum, which gives the first part of the statement. If we had $K \subseteq L$ or $L \subseteq K$ specifically, summing over the comparable chain $\{K, L\}$ would achieve this lower bound, giving the second part. 
        
        For (ii) then, recall that for any $K, L \in \mathcal{K}^n$ we may always consider the comparable chain $\{K, \, \textup{conv}(K \cup L), \, L\}$. Consequently,
        $$D_\phi(K,L) \leq \left|\phi(\textup{conv}(K \cup L)) - \phi(K)\right| + \left|\phi(\textup{conv}(K \cup L)) - \phi(L)\right| \quad \text{for all } K, L \in \mathcal{K}^n.$$
        Fix then $K \in \mathcal{K}^n$, and take a sequence $\{K_i\}_{i = 0}^\infty$ with $d_\mathcal{H}(K_i, K) \rightarrow 0$ as $i \rightarrow +\infty$. We have that $d_\mathcal{H}(\textup{conv}(K_i \cup K), K) \rightarrow 0$ as $i \rightarrow +\infty$ --- indeed, for $\delta > 0$ and any $i \geq 0$ we note
        \begin{align*}&K_i \subseteq K + \delta B^n \quad \Longleftrightarrow \quad \textup{conv}(K_i \cup K) \subseteq K + \delta B^n\end{align*}
        by the definition of the convex hull, and hence $d_\mathcal{H}(\textup{conv}(K_i \cup K), K) \leq d_\mathcal{H}(K_i, K)$. The desired result thus follows as
        \begin{align*}
            D_\phi(K_i, K) &\leq \left|\phi(\textup{conv}(K_i \cup K)) - \phi(K_i)\right| + \left|\phi(\textup{conv}(K_i \cup K)) - \phi(K)\right| \\
            &\leq 2\left|\phi(\textup{conv}(K_i \cup K)) - \phi(K) \right| + \left|\phi(K_i) - \phi(K)\right|,
        \end{align*}
        and since $\phi$ is continuous. 
        
        Part (iii) is nearly immediate --- indeed, if $\{K_i\}_{i = 0}^N \subseteq \mathcal{K}^n$ is a comparable chain joining some $K$ and $L$ in $\mathcal{K}^n$, and $G \leqslant \textup{Aff}(n)$ is such that $\phi$ is invariant under the natural action $G \curvearrowright \mathcal{K}^n$, we note for any $g \in G$ that $\{g K_i\}_{i = 1}^N$ is a comparable chain joining $gK$ and $gL$ with
        \begin{align*}
        \sum_{i = 1}^N \left |\phi(gK_{i-1}) - \phi(gK_i) \right| = \sum_{i = 1}^N \left |\phi(K_{i-1}) - \phi(K_i) \right|;
        \end{align*}
        hence $D_\phi(gK, gL) \leq D_\phi(K, L)$, and the other direction of inequality follows similarly as $g^{-1} \in G$. 
        
        Finally, for part (iv), assume $\phi \in \textup{Val}_j$ for $j \in \{0, \dots, n\}$, and take $K, L \in \mathcal{K}^n$ and $t > 0$. If $\{K_i\}_{i = 0}^N \subseteq \mathcal{K}^n$ is a comparable chain joining $K$ and $L$, then $\{tK_i\}_{i = 1}^N$ is a comparable chain joining $tK$ and $tL$ with
        $$\sum_{i = 1}^N \left|\phi(tK_{i-1}) - \phi(tK_i)\right| = t^j \sum_{i = 1}^N \left |\phi(K_{i-1}) - \phi(K_i) \right|;$$
        as we may pull the $t^j$ factor out of the relevant infimum, this yields $D_\phi(tK, tL) \leq t^jD_\phi(K,L)$, with the other direction of inequality following similarly so long as $t \neq 0$. If instead $t = 0$, the result follows trivially. 
    \end{proof}
\end{prop}

Given that we assume $\phi \in \textup{Val}$ for Proposition \ref{basic-props}, we note that part (iii) of the proposition has that $D_\phi$ is translation-invariant, in the sense that
\begin{align}\label{trans-invariant}
D_\phi(K, L) = D_\phi(K + \{x\}, L + \{x\}) \quad \text{for all } K, L \in \mathcal{K}^n,\ x \in \mathbb{R}^n
\end{align}
whenever $\phi \in \textup{Val}$. Though this and the other implications of Proposition \ref{basic-props} will be useful, they will not provide enough to study $D_\phi$ in more detail --- we will need to assume monotonicity which, as we now see, will provide a basis for comparison with the valuation deviations $\Delta_\phi$ and $\rho_\phi$.

\subsection{Monotonicity and comparison with the valuation deviations}

\begin{prop}[Properties of $D_\phi$ for monotone $\phi \in \textup{Val}$]\label{monotone-props}
    Let $\phi \in \textup{Val}$ be monotone. Then $D_\phi$ satisfies for all $K, L \in \mathcal{K}^n$
    \begin{enumerate}[label=(\roman*)]
        \item $D_\phi(K,L) = \rho_\phi(K,L) = \Delta_\phi(K,L) = \left |\phi(K) - \phi(L) \right | \quad \text{if} \quad K \subseteq L$ or $L \subseteq K$, and consequently 
        \begin{align*}
            D_\phi(K,L) &= \inf \left \{\sum_{i = 1}^N \rho_\phi(K_{i-1}, K_i)  : \{K_i\}_{i = 0}^N  \text{ is a comparable chain between } K \text{ and } L  \right \} \\
            &= \inf \left \{\sum_{i = 1}^N \Delta_\phi(K_{i-1}, K_i)  : \{K_i\}_{i = 0}^N  \text{ is a comparable chain between } K \text{ and } L  \right \},
        \end{align*}
        \item $\rho_{\phi_1}(K,L) \leq D_\phi(K,L),$
        \item $\Delta_{\phi_n}(K,L) \leq D_\phi(K,L),$
        \item $D_\phi(K,L) \leq \rho_\phi(K, L),$
        \item and $D_\phi(K,L) \leq \Delta_\phi(K,L)$
        if and only if $\phi_1 \equiv 0$.
    \end{enumerate}
    \begin{proof}
        We start with (i), which follows nearly immediately from Proposition \ref{basic-props} part (i). Indeed, if $K, L \in \mathcal{K}^n$ are such that $K \subseteq L$ or $L \subseteq K$, we get
        $D_\phi(K,L) = |\phi(K) - \phi(L)|$
        by applying the aforementioned result.  Moreover, 
        \begin{align*}
            &\Delta_\phi(K,L) = \phi(K) + \phi(L) - 2\phi(K \cap L) = \left |\phi(K) - \phi(L) \right | \quad \text{and}  \\
            &\rho_\phi(K,L) = 2\phi(\textup{conv}(K \cup L)) - \phi(K) - \phi(L) = \left | \phi(K) - \phi(L) \right |
        \end{align*}
        by monotonicity, and so the desired result follows. 
    
        For (ii), first note that for any $K$ and $L \in \mathcal{K}^n$ we have
        \begin{align*}
            \rho_\phi(K,L) &= 2\phi(\textup{conv}(K \cup L)) - \phi(K) - \phi(L) \\
            &= \sum_{j = 1}^n \big(2\phi_j(\textup{conv}(K \cup L)) - \phi_j(K) - \phi_j(L) \big) \\
            &= \sum_{j = 1}^n \rho_{\phi_j}(K,L);
        \end{align*}
        the $\rho_{\phi_0}(K,L)$ term vanishes in the expansion as $\phi_0$ is constant. As we've assumed monotonicity, each summand is nonnegative --- hence, if $K$ and $L \in \mathcal{K}^n$ and $\{K_i\}_{i = 0}^N \subseteq \mathcal{K}^n$ is a comparable chain between them,
        \begin{align*}
            \sum_{i = 1}^N \rho_\phi(K_{i-1}, K_i) &= \sum_{i = 1}^N \sum_{j = 1}^n \rho_{\phi_i}(K_{i-1},K_{i}) \\
            &\geq \sum_{i = 1}^N \rho_{\phi_1}(K_{i-1},K_i) \\
            &\geq \rho_{\phi_1}(K,L),
        \end{align*}
        the last step following as $\rho_{\phi_1}$ satisfies a triangle inequality. Indeed, given Proposition \ref{firey-rep}, we might write  $$\phi_1(M) = \int_{S^{n-1}} h_M(u) \, dS(G_1, \dots G_{n-1}; u) \quad \text{for all } M \in \mathcal{K}^n,$$
        where $G_1, \dots, G_{n-1} \in \mathcal{K}^n$ are fixed. It would follow then that
        \begin{align*}
            \rho_{\phi_1}(K_{i-1}, K_i) &= 2\phi_1(\textup{conv}(K_{i-1} \cup K_i)) - \phi_1(K_{i-1}) - \phi_1(K_i) \\
            &= \int_{S^{n-1}} \left (2\max \{h_{K_{i-1}}(u), h_{K_i}(u) \} - h_{K_{i-1}}(u) - h_{K_i}(u) \right ) \, dS(G_1, \dots, G_{n-1}; u) \\
            &= \int_{S^{n-1}} \left |h_{K_{i-1}}(u) - h_{K_i}(u) \right| \, dS(G_1, \dots, G_{n-1}; u),
        \end{align*}
        from which the triangle inequality is clear. Applying part (i) --- that is, taking an infimum --- we thus obtain $\rho_{\phi_1}(K,L) \leq D_\phi(K,L)$.

        Part (iii) is similar, albeit simpler. Indeed, we note that due to a result of Hadwiger in \cite{MR102775} we have 
        $$\phi_n(M) = c\cdot \textup{vol}(M),\quad \text{for all } M \in \mathcal{K}^n,$$
        where $c \geq 0$ is fixed. Consequently, for any $K$ and $L \in \mathcal{K}^n$, and any comparable chain $\{K_i\}_{i = 0}^N$ between $K$ and $L$, we have
\begin{align*}
    \sum_{i = 1}^N \Delta_\phi(K_{i-1}, K_i) &= \sum_{i = 1}^N \sum_{j = 1}^n \Delta_{\phi_j}(K_{i-1}, K_i) \\
    &\geq \sum_{i = 1}^N \Delta_{\phi_n}(K_{i-1}, K_i) \\
    &\geq \Delta_{\phi_n}(K,L),
\end{align*}
following the same reasoning as in the previous part. Indeed, $\Delta_{\phi_n} = c \cdot \Delta_\textup{vol}$ here, and hence it satisfies the triangle inequality on $\mathcal{K}^n$. It follows then that, after taking an infimum, part (i) has $\Delta_{\phi_n}(K,L) \leq D_\phi(K,L)$.

        Part (iv) is demonstrated immediately given an earlier observation: if we have $K, L \in \mathcal{K}^n$, we can consider the comparable chain $\{K_i\}_{i = 0}^2$ where
        $$K_0 = K, \quad K_1 = \textup{conv}(K \cup L), \quad \text{and} \quad K_2 = L.$$
        Summing over this gets exactly $\rho_\phi(K, L)$ by monotonicity, and so we have the desired result. 
        
        Finally, for (v), we first recall that if $S$ is a line segment in $\mathbb{R}^n$, $\phi_j(S) = 0$ for $j \geq 2$. We recount here an elementary proof of this fact: given a segment $S \subseteq \mathbb{R}^n$, we may write it as
$$
S = S_1 \cup S_2, \qquad \text{where} \qquad S_i := \frac{1}{2}S + \{x_i\},
$$
for some $x_1, x_2 \in \mathbb{R}^n$, where, furthermore, the $S_i$ meet only at a point $x' \in \mathbb{R}^n$.

\begin{figure}[H]
\centering
\begin{tikzpicture}[scale=0.8,line cap=round,line join=round]
\tikzset{
  seg/.style={black,line width=1pt},
  pt/.style={circle,fill=black,inner sep=1.4pt},
  lab/.style={font=\small}
}

\begin{scope}[xshift=0cm]
  \draw[seg] (0,0) -- (2,2);
  \node[pt] at (0,0) {};
  \node[pt] at (2,2) {};
  \node[lab,above=3pt,xshift=-4pt] at (1,1) {$S$};
\end{scope}

\begin{scope}[xshift=6cm]
  \draw[seg] (0,0) -- (2,2);
  \node[pt] at (0,0) {};
  \node[pt] at (1,1) {};
  \node[pt] at (2,2) {};
  \node[lab,above=3pt,xshift=-4pt] at (0.5,0.5) {$S_1$};
  \node[lab,above=3pt,xshift=-4pt] at (1.5,1.5) {$S_2$};
  \node[lab,below=3pt,xshift=4pt] at (1,1) {$x'$};
\end{scope}
\end{tikzpicture}
\caption{A line segment $S$ written as the union of two congruent subsegments $S_1$ and $S_2$ meeting at a point $x' \in \mathbb{R}^n$.}
\end{figure}
We recall then $\phi_j$ vanishes on points for $j \geq 1$ --- indeed, one notes that by translation-invariance and homogeneity we have 
\begin{align}\label{vanish-on-pt}\phi_j(\{x\}) = \phi_j(\{0\}) = \phi_j(t\{0\}) = t^j\phi(\{0\}) = t^j\phi_j(\{x\}) \quad \text{for all } x \in \mathbb{R}^n, \quad t \geq 0,\end{align}
which necessitates $\phi_j(\{x\})=0$ for all $x \in \mathbb{R}^n$ when $j \geq 1$. Applying the valuation property for $\phi_j$ with $S_1$ and $S_2$ thus yields 
\begin{align*}
    \phi_j(S) &= \phi_j(S_1) + \phi_j(S_2) - \phi_j(\{x'\}) \\
    &= \frac{1}{2^{j-1}}\phi(S),
\end{align*}
as $\phi_j(S_1)$, $\phi_j(S_2) = \nicefrac{1}{2^j}$ given translation-invariance and homogeneity. In a similar vein to (\ref{vanish-on-pt}), we thus conclude $\phi_j(S) = 0$ for $j \geq 2$. Returning then to proving part (v), take $K, L \in \mathcal{K}^n$ --- there are two cases.

\begin{enumerate}[label=(\arabic*)]
\item ($K \cap L \neq \varnothing$) Here summing over the comparable chain $\{K, K \cap L, L\}$
gets precisely $\Delta_\phi(K,L)$ by monotonicity, and hence $D_\phi(K,L) \leq \Delta_\phi(K,L).$ \item ($K \cap L = \varnothing$) Take $x \in K$ and $y \in L$ and consider the comparable chain $\{K_i\}_{i = 0}^4$ where 
$$K_0 = K, \quad K_1 = K \cap [x,y], \quad K_2 = [x,y], \quad K_3 = L \cap [x,y], \quad K_4 = L;$$
here $[x,y] := \textup{conv}(\{x\} \cup \{y\})$, that is it is the closed line segment joining $x$ and $y \in \mathbb{R}^n$.  Summing over this yields 
        \begin{align*}
            &\phi(K) + \phi(L)  + 2\phi([x,y]) - 2\phi(K \cap [x,y]) - 2\phi(L \cap [x,y]) \\
            &= \sum_{j = 1}^n \big( \phi_j(K) + \phi_j(L) + 2\phi_j([x,y]) - 2\phi_j(K \cap [x,y]) - 2\phi_j(L \cap [x,y]) \big)
        \end{align*}
        given monotonicity --- the $j = 0$ terms all cancel as $\phi_0$ is constant. 
        Under the assumption then that $\phi_1 \equiv 0$, this expression furthermore reduces to 
        $$\sum_{j = 2}^n \big( \phi_j(K) + \phi_j(L) \big) \leq \phi(K) + \phi(L) = \Delta_\phi(K,L)$$
        given that $\phi_j$ vanishes on line segments for $j \geq 2$, and hence we also have $D_\phi(K,L) \leq \Delta_\phi(K,L)$ in this case. 
        \end{enumerate}
        
        We show then that $\phi_1 \equiv 0$ is \emph{necessary} for this bound --- assume for that $\phi_1 \not \equiv 0$; by part (ii) we have that 
        $$\int_{S^{n-1}} |h_K(u) - h_L(u)| \, dS(G_1, \dots, G_{n-1}) = \rho_{\phi_1}(K, L) \leq D_\phi(K,L) \quad \text{for all } K,L \in \mathcal{K}^n$$
        for some fixed $G_1, \dots, G_{n-1} \in \mathcal{K}^n$. The assumption $\phi_1 \not \equiv 0$ implies that the support of the mixed area measure $S(G_1, \dots, G_{n-1}; \cdot)$ is nonempty, so pick $v \in \supp S(G_1, \dots, G_{n-1}; \cdot)$. Then 
        \begin{align*}
            \rho_{\phi_1}(B^n, B^n + t\{v\}) &= \int_{S^{n-1}} |h_{B^n}(u) - h_{B^n + t\{v\}}(u)| \, dS(G_1, \dots, G_{n-1}; u) \\
            &= t\int_{S^{n-1}} |\langle u, v \rangle | \, dS(G_1, \dots, G_{n-1}; u),
        \end{align*}
        where $\int_{S^{n-1}} |\langle u, v \rangle | \, dS(G_1, \dots, G_{n-1}; u) > 0$ as $v \in \supp(G_1, \dots, G_{n-1}; u)$, $| \langle u, v \rangle |$ is a continuous function of $u$ on $S^{n-1}$, and $| \langle v, v \rangle | = 1$.
        The bound $D_\phi(K,L) \leq \Delta_\phi(K,L)$ thus cannot hold for all $K, L \in \mathcal{K}^n$, as we would have
        $$t\int_{S^{n-1}} |\langle u, v \rangle| \, dS(G_1, \dots, G_{n-1}; u) \leq \Delta_\phi(B^n, B^n + t\{v\}) \leq 2\phi(B^n),$$
        where taking $t \rightarrow + \infty$ gives a contradiction. The desired result thus follows.
    \end{proof}
\end{prop}

Proposition \ref{monotone-props} immediately leads to a corollary which justifies regarding $D_\phi$ as a generalization of $\Delta_{\textup{vol}}$ and $\rho_w$ --- we note this now.

\begin{coro}\label{generalize}
    Let $\phi \in \textup{Val}$ be monotone. We have $D_{\phi_1}(K,L) = \rho_{\phi_1}(K,L)$ for all $K, L \in \mathcal{K}^n$. If $n \geq 2$, we have $D_{\phi_n}(K,L) = \Delta_{\phi_n}(K,L)$ for all $K, L \in \mathcal{K}^n$. Consequently,
\begin{align}
D_{\textup{vol}}(K,L)=\Delta_{\textup{vol}}(K,L) \quad \text{for all } K, L \in \mathcal{K}^n
\end{align}
when $n \geq 2$, and, regardless of any conditions on dimension,
\begin{align}
D_w(K,L) = \rho_w(K,L) \quad \text{for all } K, L \in \mathcal{K}^n.
\end{align}
    \begin{proof}
        The result follows immediately from combining Proposition \ref{monotone-props} parts (ii) -- (v).
    \end{proof}
\end{coro}

\section{Conditions for metricity}\label{4}

Corollary \ref{generalize} shows us that, at least in these particular cases of $\phi = \textup{vol}$ and $\phi = w$, the $D_\phi$ construction yields metrics --- metrics which take the form of valuation deviations. This raises some natural questions:
\begin{enumerate}[label=(\arabic*)]
    \item \label{goal1} are there general conditions on $\phi \in \textup{Val}$ such that $D_\phi$ is a metric (on $\mathcal{K}^n_k$ for some $k \in \{0, \dots, n\})$, and
    \item \label{goal2} when does $D_\phi$ happen to have a closed form such as $\rho_\phi$ or $\Delta_\phi$?
\end{enumerate}

In particular, it would be desirable to affirmatively resolve \ref{goal1} to at least the degree of being able to accommodate the intermediate intrinsic volumes $\{\textup{V}_j\}_{j = 2}^{n-1}$ --- we do so now. 

\subsection{Nondegeneracy}

\begin{theoremA}[Nondegeneracy conditions]
\label{thm:A}
Let $\phi \in \textup{Val}$ be monotone. We have that $D_\phi$ is nondegenerate (and hence a metric)
\begin{enumerate}[label=(\roman*)]
  \item on $\mathcal{K}^n_n$ when $\phi$ has a nontrivial $n$-homogeneous component, i.e. $\phi_n \not \equiv 0$, 
  \item on $\mathcal{K}^n$ when, in the context of the representation
  \[
    \phi_1(K)= nV(K,G_1, \dots, G_{n-1}) = \int_{S^{n-1}} h_K(u) \, dS(G_1, \dots, G_{n-1}; u) \quad \text{for all } K \in \mathcal{K}^n
  \]
  given some fixed $G_1, \dots, G_{n-1} \in \mathcal{K}^n $, it is that the mixed area measure $\mu := S(G_1, \dots, G_{n-1}; \cdot)$ has full support,
    \item on $\mathcal{K}^n_k$ when $\phi \in \textup{Val}^{+,\infty}$ and we may choose nonnegative representing functions $f_j$ in (\ref{grassman-comp}), with at least one $f_k>0$ on $\Gr_{n,k}$ for some $k\in\{1,\dots,n\}$, and 
  \item on $\mathcal{K}^n$ when $\phi = \textup{vol}(G + K)$ for some $G \in \mathcal{K}^n_n$.
\end{enumerate}
\begin{proof}
    For each part of the theorem, all we need to do is verify that $D_\phi$ distinguishes points on the subspace it is claimed to be a metric on, i.e. we need to show 
    $$D_\phi(K,L) = 0 \quad \Longrightarrow \quad K = L \quad \text{for all } K, L \in \mathcal{K}^n_k,$$
    given the appropriate choice of $k \in \{0, \dots, n\}$ --- as a pseudometric on $\mathcal{K}^n$, this metric axiom is the only one $D_\phi$ \textit{a priori} lacks.

    With this in mind, we focus on (i) and (ii) --- first recall then that, by Proposition \ref{monotone-props} parts (ii) and (iii), we have that 
    \begin{align}\label{lower-bounds}\Delta_{\phi_n}(K,L) \leq D_\phi(K,L) \quad \text{and} \quad \rho_{\phi_1}(K,L) \leq D_\phi(K,L) \quad \text{for all } K,L \in \mathcal{K}^n.\end{align}
    Using this in conjunction with the aforementioned result of Hadwiger in \cite{MR102775} immediately resolves (i); indeed, in the context of this result, the assumption $\phi_n$ is monotone \emph{and} nontrivial is that 
    $$\phi_n(K) = c \cdot \textup{vol}(K) \quad \text{for all } K \in \mathcal{K}^n,$$
    where, in particular, $c > 0$. Hence, using (\ref{lower-bounds}), we get for all for $K, L \in \mathcal{K}^n$ that 
    \begin{align*}
        D_\phi(K,L) &\geq \Delta_{\phi_n}(K,L) \\
        &= c \cdot \Delta_{\textup{vol}}(K,L),
    \end{align*}
    and so $D_\phi$ distinguishes points on $\mathcal{K}^n_n$ as $\Delta_{\textup{vol}}$ does. For (ii) then, instead recall our earlier derivation that, for all $K, L \in \mathcal{K}^n$, one has
\begin{align*}
    \rho_{\phi_1}(K, L) &= 2\phi_1(\textup{conv}(K \cup L)) - \phi_1(K) - \phi_1(L) \\
    &= \int_{S^{n-1}} \left(2\max\{h_{K}(u), h_{L}(u)\} - h_{K}(u) - h_{L}(u)\right) \, dS(G_1, \dots, G_{n-1}; u) \\
    &= \int_{S^{n-1}} \left|h_{K}(u) - h_{L}(u)\right| \, dS(G_1, \dots, G_{n-1}; u),
\end{align*}
where $S(G_1, \dots, G_{n-1}; \cdot)$ is the mixed area measure as in (\ref{firey-rep-eq}). Consequently, if we had $K, L \in \mathcal{K}^n$ such that $K \neq L$, $\{u \in S^{n-1} : |h_K(u) - h_L(u)| > 0\}$ would be open and nonempty, and thus the assumption this measure has full support would yield
\begin{align*}
    \rho_{\phi_1}(K,L) = \int_{S^{n-1}} \left|h_{K}(u) - h_{L}(u)\right| \, dS(G_1, \dots, G_{n-1}; u) > 0.
\end{align*}
Applying this in conjunction with (\ref{lower-bounds}) thus shows here that $D_\phi$ distinguishes points on $\mathcal{K}^n$.
    
We consider (iii) then. For $K, L \in \mathcal{K}^n$ with $K \subseteq L$ or $L \subseteq K$, via (\ref{grassman-comp}) we have
    \begin{align*}
        \left |\phi(K) - \phi(L) \right| 
        &= \left | \sum_{j = 0}^n \int_{\Gr_{n,j}} f_j(H) \, \textup{vol}_j(\pi_H(K)) \, d\nu_j(H) - \sum_{j = 0}^n \int_{\Gr_{n,j}} f_j(H)\textup{vol}_j(\pi_H(L)) \, d\nu_j(H)\right | \\
        &= \sum_{j = 1}^n \int_{\Gr_{n,j}} f_j(H) \left |\textup{vol}_j(\pi_H(K)) - \textup{vol}_j(\pi_H(L)) \right | \, d\nu_j(H)  \\
        &= \sum_{j = 1}^n \int_{\Gr_{n,j} } f_j(H) \, \textup{vol}_j(\pi_H(K) \, \triangle \, \pi_H(L)) \, d\nu_j(H).
    \end{align*}
    Let $m = \min_{H \in \Gr_{n,k}} f_k(H) > 0$, existent given compactness of the Grassmannian $\Gr_{n,k}$ --- given nonnegativity of the $f_j$, it is clear from our above work that one has $$\left | \phi(K) - \phi(L) \right | \geq m \cdot \int_{\Gr_{n,k}} \textup{vol}_k(\pi_H(K) \, \triangle \, \pi_H(L)) \, d\nu_k(H);$$
    it is known, however, that the expression $\int_{\Gr_{n,k}} \textup{vol}_k(\pi_H(K) \, \triangle \, \pi_H(L)) \, d\nu_k(H)$ gives a pseudometric on $\mathcal{K}^n$, which is metric once restricted to $\mathcal{K}^n_k$ (see \cite{besau2024} Theorem 2.1). Consequently, if $K, L \in \mathcal{K}^n_k$ and $\{K_i\}_{i = 0}^N \subseteq \mathcal{K}^n$ is a comparable chain between $K$ and $L$, we have 
    \begin{align*}
        \sum_{i = 1}^N \left |\phi(K_{i-1}) - \phi(K_i) \right| &\geq m \cdot \sum_{i = 1}^N \int_{\Gr_{n,k}} \textup{vol}_k(\pi_H(K_{i-1}) \, \triangle \, \pi_H(K_i)) \, d\nu_k(H) \\
        &\geq m \cdot \int_{\Gr_{n,k}} \textup{vol}_k(\pi_H(K) \, \triangle \, \pi_H(L)) \, d\nu_k(H).
    \end{align*}
    Thus, taking an infimum, $D_\phi(K,L) \geq m \cdot \int_{\Gr_{n,k}} \textup{vol}_k(\pi_H(K) \, \triangle \, \pi_H(L)) \, d\nu_k(H)$ --- leveraging nondegeneracy on $\mathcal{K}^n_k$ of the latter term then shows $D_\phi$ is a metric on $\mathcal{K}^n_k$. 

    We finish with (iv), the proof of which again operates by a similar strategy. Let $G \in \mathcal{K}^n_n$ --- we consider the pullback of $\Delta_{\textup{vol}}$ along the injective mapping $\mathcal{K}^n \rightarrow \mathcal{K}^n_n$ given by $K \mapsto G + K$, i.e.
    \begin{equation}
    \label{eqn:partialG}
    \partial_G(K,L) := \textup{vol}(G + K \, \triangle \, G + L) \quad \text{for all } K, L \in \mathcal{K}^n.
    \end{equation}
    As such a pullback, this is a metric on $\mathcal{K}^n$; let then any $K, L \in \mathcal{K}^n$. We have that
    \begin{align*}
        \Delta_{\phi}(K, L) &= \textup{vol}(G + K) + \textup{vol}(G + L) - 2\textup{vol}(G + (K \cap L)) \\
        &\geq \textup{vol}(G + K) + \textup{vol}(G + L) - 2\textup{vol}((G + K) \cap (G + L))
        \\ &= \partial_G(K,L),
    \end{align*}
    as $G + (K \cap L) \subseteq (G + K) \cap (G + L)$. For any $K, L \in \mathcal{K}^n$ and comparable chain $\{K_i\}_{i = 0}^N \subseteq \mathcal{K}^n$ between $K$ and $L$ then, we have
    \begin{align*}
        \sum_{i = 1}^N |\phi(K_{i-1}) - \phi(K_i)| &= \sum_{i = 1}^N \Delta_\phi(K_{i-1}, K_i) \\
        &\geq \sum_{i = 1}^N \partial_G(K_{i-1},K_i) \\
        &\geq \partial_G(K,L).
    \end{align*}
    by applying Proposition \ref{monotone-props} part (i). Taking an infimum thus has $D_\phi(K,L) \geq \partial_G(K,L)$, so $D_\phi$ is again a metric in this case.
\end{proof}
\end{theoremA}

Theorem \ref{metricity-conditions} leads immediately to a significant corollary.

\begin{coroB}
    Let $\phi \in \textup{Val}$ be monotone, nonconstant, and $\textup{SO}(n)$-invariant. Then $D_\phi$ is a metric on $\mathcal{K}^n_k$ for some $k \in \{0, \dots, n\}$. 
    
    In particular, this result holds for any linear combination of the intrinsic volumes $\textup{V}_1, \dots, \textup{V}_n$ with nonnegative coefficients, not all zero. 
    \begin{proof}
        By Hadwiger's characterization of $\phi \in \textup{Val}$ that are invariant under (proper) rigid motions, i.e. Proposition \ref{hadwiger-char}, we have
        $$\phi(K) = \sum_{j = 0}^n c_j \textup{V}_j(K) \quad \text{for all } K \in \mathcal{K}^n$$
        and some $c_0, \dots, c_n \in \mathbb{R}$. Monotonicity assures us $c_j \geq 0$ for each $j \in \{0, \dots n\}$, and the fact that $\phi$ is nonconstant improves this to $c_k > 0$ for some $k \in \{1, \dots n\}$. We recall then the Kubota formula
        $$\textup{V}_j(K) = \binom n j \frac{\kappa_n}{\kappa_j \kappa_{n-j}} \int_{\Gr_{n,j}} \textup{vol}_j(\pi_H(K)) \, d\nu_j(H) \quad \text{for all } K \in \mathcal{K}^n, \quad j \in \{0, \dots, n\};$$
        this shows then that the choice of (constant) representing functions $$f_j :\equiv c_j \binom n j \frac{\kappa_n}{\kappa_j \kappa_{n-j}} \quad \text{for all } j \in \{0, \dots, n\}$$ exactly satisfies the conditions of part (iii) of Theorem \ref{metricity-conditions}, and so $D_\phi$ is exactly a metric on $\mathcal{K}^n_{\min \left \{j \in \{1, \dots, n\} : c_j > 0 \right\}}$, giving the desired result --- note also that if $\phi_1 \not \equiv 0$, then by Theorem \ref{metricity-conditions} part (ii), $D_\phi$ will be a metric on \emph{all} of $\mathcal{K}^n$ as opposed to just $\mathcal{K}^n_1$.
    \end{proof}
\end{coroB}

The main takeaway of the corollary is that the intermediate intrinsic volumes $\left \{\textup{V}_2, \dots, \textup{V}_{n-1}  \right \}$ do indeed yield metrics --- this motivates adopting some shorthand.

\begin{defn}
    Let $j \in \{1, \dots, n\}$. Define 
    $$D_j(K,L) := \inf \left \{\sum_{i = 1}^N \left |\textup{V}_j(K_{i-1}) - \textup{V}_j(K_i) \right | : \{K_i\}_{i = 0}^N \text{ is a comparable chain between } K \text{ and } L  \right \}$$ for all $K, L \in \mathcal{K}^n$. 
    Then $D_j$ is a pseudometric on $\mathcal{K}^n$ and a metric on $\mathcal{K}^n_{f(j)}$, where $f(j) := j - \delta_{j,1}$ --- here $\delta_{j,1}$ is the Kronecker delta. In particular, $D_n = \Delta_{\textup{vol}}$ and $D_1 \varpropto \rho_w$.
\end{defn}

It will be interesting to study properties of the family of metrics $\{D_j\}_{j = 1}^n$, especially the nature of their induced topologies. Indeed, in Section \ref{6}, we will show that \emph{each} $D_j$ induces the Hausdorff topology.

Now then that we have then broad conditions in terms of $\phi \in \textup{Val}$ for when $D_\phi$ is a metric, we move on to question \ref{goal2} --- to which our response will be \emph{rarely ever}.

\subsection{Obstruction to valuation deviations as closed forms}

\begin{theoremC}
    Let $\phi \in \textup{Val}$ be monotone, and write 
    $\phi(K) = \sum_{i = 0}^n \phi_i(K)$ for all $K \in \mathcal{K}^n,$ where
    $\phi_i \in \textup{Val}_i$. If $\phi_i \not \equiv 0$ for any $i \in \{0, \dots, n-1\}$, then there exists $K, L, M \in \mathcal{K}^n$ with 
     $$\Delta_\phi(K,L) > \Delta_\phi(K,M) + \Delta_\phi(M,L).$$
     Similarly, if $\phi_i \not \equiv 0$ for any $i \in \{2, \dots, n\}$, there exists $K, L, M \in \mathcal{K}^n$ with 
     $$\rho_\phi(K,L) > \rho_\phi(K,M) + \rho_\phi(M,L).$$
     Hence $D_\phi = \Delta_\phi$ implies $\phi \in \textup{Val}_n$, and $D_\phi = \rho_\phi$ forces $\phi \in \textup{Val}_1 \oplus \textup{Val}_0$.
     \begin{proof}
         We prove the statement for $\Delta_\phi$ first --- write 
         $\psi := \phi_k$, where $k:=\min\{j \in \{0, \dots, n-1\} : \phi_j \not \equiv 0\}$.
         We assume for the sake of contradiction that $\Delta_\phi$ satisfies a triangle inequality on $\mathcal{K}^n$, i.e. 
         \begin{align*}\Delta_\phi(K,L) \leq \Delta_\phi(K,M) + \Delta_\phi(M,L) \quad \text{for all } K, L, M \in \mathcal{K}^n.\end{align*}
         We claim then that this triangle inequality descends to $\Delta_\psi$. Indeed, note by assumption that we have for all $K, L, M \in \mathcal{K}^n$ that
         \begin{align}\label{triangle-ineq-proof}\Delta_\phi(\lambda K, \lambda L) \leq \Delta_\phi(\lambda K, \lambda M) + \Delta_\phi(\lambda M, \lambda L) \quad \text{for all } \lambda > 0,\end{align}
         where for $\lambda > 0$
         \begin{align*}
         \Delta_\phi(\lambda K, \lambda L) &= \phi(\lambda K) + \phi(\lambda L) - 2\phi(\lambda K \cap \lambda L) \\
         &= \phi(\lambda K) + \phi(\lambda L) - 2\phi(\lambda(K \cap L)) \\
         &= \sum_{j = k}^n \lambda^j \left[\phi_j(K) + \phi_j(L) - 2\phi_j(K \cap L)\right] \\
         &= \sum_{j = k}^n \lambda^j \Delta_{\phi_j}(K,L).
         \end{align*}
         Expanding (\ref{triangle-ineq-proof}) and dividing by $\lambda^k$ thus gets for all $K, L, M \in \mathcal{K}^n$ that  
         $$\Delta_\psi(K,L) + \sum_{j = k + 1}^n \lambda^{j-k} \Delta_{\phi_j}(K,L) \leq \Delta_\psi(K,M) + \sum_{j = k + 1}^n \lambda^{j-k} \Delta_{\phi_j}(K,M) + \Delta_\psi(M,L) + \sum_{j = k + 1}^n \lambda^{j-k} \Delta_{\phi_j}(M,L),$$
         where taking $\lambda \rightarrow 0$ yields then the triangle inequality for $\Delta_\psi$.
         
         We note then that, as $\psi$ is $k$-homogeneous, nonconstant, and nonnegative (as it is monotone), it cannot be \textit{simple}; i.e., we have $\psi(A) > 0$ for some $A \in \mathcal{K}^n$ with $\dim(A) < n$ --- this is another classical result of Hadwiger; see for instance \cite{Schneider2017ValuationsClassicalBasicFacts} Theorem 1.15.
         
Without loss of generality then, we assume $A$ has its affine hull a linear subspace of dimension $\ell < n$, translating if necessary. Taking a unit vector $u \in \textup{aff}(A)^\perp$, we consider for $\epsilon > 0$
         $$K := A , \quad M_\epsilon := K + \epsilon\{u\}, \quad \text{and} \quad L_\epsilon := \textup{conv}(A \cup M_\epsilon);$$
         obviously one has 
         $$K \cap M_\epsilon = \varnothing , \quad K \cap L_\epsilon = K_, \quad \text{and} \quad L_\epsilon \cap M_\epsilon = M_\epsilon,$$
         the latter two equalities especially immediate. The former is similarly clear: note for $x \in K$ that $\langle u, x \rangle = 0$, whereas for $y = x + \epsilon u \in M_\epsilon$ we have
         $$\langle u, y \rangle = \langle u, x + \epsilon u \rangle = \langle u, x \rangle + \epsilon \langle u, u \rangle = \epsilon,$$
         and so necessarily $K \cap M_\epsilon = \varnothing.$
         As we also have $d_\mathcal{H}(L_\epsilon, K)  \leq  d_\mathcal{H}(M_\epsilon, K) = \epsilon$, we may in particular 
          take $\epsilon$ small enough such that $\psi(L_\epsilon) < 2\psi(K).$
         Finally, as we thus have 
         \begin{align*}&\Delta_\psi(K, M_\epsilon) = \psi(K) + \psi(M_\epsilon) = 2\psi(K), \\
         &\Delta_\psi(K, L_\epsilon) = \psi(K) + \psi(L_\epsilon) - 2\psi(K) = \psi(L_\epsilon) - \psi(K), \quad \text{and} \\
         &\Delta_\psi(L_\epsilon, M_\epsilon) = \psi(L_\epsilon) + \psi(M_\epsilon) - 2\psi(M_\epsilon) = \psi(L_\epsilon) - \psi(M_\epsilon) = \psi(L_\epsilon) - \psi(K),
         \end{align*}
         we remark the triangle inequality for $\Delta_\psi$ would imply 
         $$2\psi(K) \leq 2[\psi(L_\epsilon) -\psi(K)],$$
         and hence $2\psi(K) \leq \psi(L_\epsilon)$, a contradiction. It follows $\Delta_\phi$ could \emph{not} have originally satisfied a triangle inequality.

         We move onto the result for $\rho_\phi$ then, which begins similarly --- appropriating our prior notation, we write
         $\psi := \phi_k$ where $k:= \max\{ j \in \{2, \dots n\} : \phi_j \not \equiv 0\}.$
         We first show that, dual to the $\Delta_\phi$ analysis, a triangle inequality for $\rho_\phi$ \emph{ascends} to $\rho_\psi$. Indeed, as 
         $$\textup{conv}(\lambda K \cup \lambda L) = \textup{conv}(\lambda (K \cup L)) = \lambda(\textup{conv}(K \cup L)) \quad \text{for all } \lambda > 0,$$
         we have that 
         $$\rho_\phi(\lambda K, \lambda L) = \sum_{j = 1}^k \lambda^j \rho_{\phi_j}(K,L) \quad \text{for all } K, L \in \mathcal{K}^n.$$
         Consequently, if a triangle inequality held for $\rho_\phi$ on $\mathcal{K}^n$, considering it for $K, L, M \in \mathcal{K}^n$ and dividing by $\lambda^k$ gets 
         $$\rho_\psi(K,L) + \sum_{j = 1}^{k-1} \lambda^{j-k} \rho_{\phi_j}(K,L) \leq \rho_\psi(K,M) + \sum_{j = 1}^{k-1} \lambda^{j-k} \rho_{\phi_j}(K,M) + \rho_\psi(M,L) + \sum_{j = 1}^{k-1} \lambda^{j-k} \rho_{\phi_j}(M,L),$$
         where taking $\lambda \rightarrow +\infty$ yields the triangle inequality for $\rho_\psi$. Recalling then by (\ref{compatibility}) that 
         $\textup{Val}_k = \textup{Val}_k^+ \oplus \textup{Val}_k^-,$ we decompose 
         $$\psi(K) = \frac{\psi(K) + \psi(-K)}{2} + \frac{\psi(K) - \psi(-K)}{2},$$
         where the former term on the right is even and the latter is odd. Obviously, the odd component vanishes on origin-symmetric sets (i.e. $K$ such that $K = -K$) --- moreover, it is that
         $$-\textup{conv}(A) = \textup{conv}(A) \quad \text{and} \quad -(A \cup B) = A \cup B \quad \text{for all } A, B \subseteq \mathbb{R}^n$$
         such that $A, B \subseteq \mathbb{R}^n$ are (origin-)symmetric, that is symmetric sets are closed under unions and the convex hull. An immediate consequence of this is that on symmetric $K, L \in \mathcal{K}^n$, writing $\omega := \frac{\psi(K) + \psi(-K)}{2},$ we have
         \begin{align*}
             \rho_\psi(K,L) &= 2\psi(\textup{conv}( K \cup L)) - \psi(K) - \psi(L) \\
             &= 2\omega(\textup{conv}(K \cup L)) - \omega(K) - \omega(L) \\
             &= \rho_\omega(K,L),
         \end{align*}
         where $\omega \in \textup{Val}^+_k$. It is thus sufficient to produce a failure of the triangle inequality for $\rho_\omega$ on symmetric bodies in $\mathcal{K}^n$. For this, note that $\omega \not \equiv 0$ on $\mathcal{K}^n$ --- if it were identically zero, $\psi$ would be odd and nonnegative, in which case $\psi \equiv 0$. By the injectivity of the Klain function then, 
         $$\omega \vert_E = c \cdot\textup{vol}_k,$$
         where $E \in \Gr_{n,k}$, $c > 0$, and $\textup{vol}_j$ denotes the restriction of the normal Hausdorff measure to $E$; this follows as $\omega$ is nonnegative itself. 
         We proceed by identifying $E \cong \mathbb{R}^k$ and setting
         $$K_T := [-T, T] \times [-1, 1]^{k-1}, \quad M_T := [-1, 1] \times [-T, T] \times [-1, 1]^{k-2}, \quad \text{and} \quad L_T:= K_T \cap M_T = [-1,1]^k,$$
         where $T > 1$ --- for the sake of shorthand we also write $H_T := \textup{conv}(K_T \cup M_T)$. Clearly, each of $K_T, M_T, L_T$ and $H_T$ are symmetric.

\begin{figure}[H]
\centering
\begin{tikzpicture}[scale=0.55, line cap=round, line join=round]

\def\T{2.4}
\def\one{1}
\def\ax{3.2}
\def\pan{7.2}

\tikzset{
  axline/.style={black!100, dotted, line width=0.8pt, ->},
  outline/.style={black!100, line width=0.9pt},
  boxline/.style={black!100, line width=0.8pt},
  fillshape/.style={fill=gray!18, opacity=0.5},
}

\newcommand{\DrawAxes}{
  \draw[axline] (-\ax,0) -- (\ax,0);
  \draw[axline] (0,-\ax) -- (0,\ax);
  \node[black!100, font=\small, below right] at (\ax,0) {$x_1$};
  \node[black!100, font=\small, above left]  at (0,\ax) {$x_2$};
}

\begin{scope}[xshift=0*\pan cm]
  \fill[fillshape] (-\T,-\one) rectangle (\T,\one);
  \draw[boxline] (-\T,-\one) rectangle (\T,\one);
  \DrawAxes
  \node[black!100, font=\small] at (0,-3.65) {$(a)\;\;K_T$};
\end{scope}

\begin{scope}[xshift=1*\pan cm]
  \fill[fillshape] (-\one,-\T) rectangle (\one,\T);
  \draw[boxline] (-\one,-\T) rectangle (\one,\T);
  \DrawAxes
  \node[black!100, font=\small] at (0,-3.65) {$(b)\;\;M_T$};
\end{scope}

\begin{scope}[xshift=2*\pan cm]
  \fill[fillshape] (-\one,-\one) rectangle (\one,\one);
  \draw[boxline] (-\one,-\one) rectangle (\one,\one);
  \DrawAxes
  \node[black!100, font=\small] at (0,-3.65) {$(c)\;\;L_T$};
\end{scope}

\begin{scope}[xshift=3*\pan cm]
  \coordinate (A1) at ( \T,  \one);
  \coordinate (A2) at ( \one, \T);
  \coordinate (A3) at (-\one, \T);
  \coordinate (A4) at (-\T,  \one);
  \coordinate (A5) at (-\T, -\one);
  \coordinate (A6) at (-\one,-\T);
  \coordinate (A7) at ( \one,-\T);
  \coordinate (A8) at ( \T, -\one);

  \fill[fillshape] (A1)--(A2)--(A3)--(A4)--(A5)--(A6)--(A7)--(A8)--cycle;
  \draw[outline] (A1)--(A2)--(A3)--(A4)--(A5)--(A6)--(A7)--(A8)--cycle;

  \DrawAxes
  \node[black!100, font=\small] at (0,-3.65) {$(d)\;\;H_T=\conv(K_T\cup M_T)$};
\end{scope}

\end{tikzpicture}
\caption{$K_T, M_T, L_T$ and $H_T$ for $k = 2$ and a choice of $T > 1$.}
\label{fig:four-panel-HT}
\end{figure}
     We note then that $$\textup{vol}_k(H_T) + \textup{vol}_k(L_T) - \textup{vol}_k(K_T) - \textup{vol}_k(M_T) = 2^{k-1}(T-1)^2 > 0;$$
     indeed, writing $R_1 := [-T,T] \times [-1,1]$, $ R_2 := [-1 , 1] \times [-T,T]$, and $C := [-1,1]^{k-2}$ we have
     \begin{align*}
         H_T &= \textup{conv}(K_T \cup M_T) 
         = \textup{conv}((R_1 \times C) \cup (R_2 \times C)) 
         = \textup{conv}((R_1 \cup R_2) \times C) 
         = \textup{conv}(R_1 \cup R_2) \times C.
     \end{align*}
     It follows $\textup{vol}(H_T) = \textup{area}\left (\textup{conv}(R_1 \cup R_2)\right) \cdot \textup{vol}_{k-2}(C) = 2^{k-2} \cdot \textup{area}\left (\textup{conv}(R_1 \cup R_2)\right)$, where clearly 
     \begin{align*}\textup{area}\left (\textup{conv}(R_1 \cup R_2)\right) &= (2T)^2 - 4 \left ( \frac{(T-1)^2}{2}\right ) \\ &= 2T^2 + 4T - 2.\end{align*}
     As we have $\textup{vol}_k(L_T)  = 2^k$, and $\textup{vol}_k(K_T) = \textup{vol}_k(M_T) = 2^kT$ then, we get 
     \begin{align*}
     \textup{vol}_k(H_T) + \textup{vol}_k(L_T) - \textup{vol}_k(K_T) - \textup{vol}_k(M_T) &= 2^{k-2}(2T^2 + 4T - 2) + 2^k - 2^{k+1}T \\
     &= 2^{k-1}(T-1)^2 \\
     &> 0.
     \end{align*}
Reversing the identification $E \cong \mathbb{R}^k$ changes none of the volume analysis of course, as does it not affect origin-symmetry, given that any linear map commutes with the involution $x \mapsto -x$.
\begin{figure}[H]
\centering
\begin{tikzpicture}[
  scale=0.45,
  line cap=round, line join=round,
  x={(1.28cm,0.00cm)},
  y={(-0.68cm,0.92cm)},
  z={(0.00cm,1.18cm)},
  axpos/.style={black!100,line width=1.0pt,->},
  axneg/.style={black!100,dotted,line width=0.95pt,->},
  faint/.style={black!100,opacity=.12,line width=.35pt},
  plane/.style={fill=white,draw=none},
  ht/.style={fill=gray!18,opacity=.5,draw=black!100,line width=1.0pt},
]

% parameters
\def\HTT{2.4}
\def\HTone{1}
\def\PlaneRad{4.0}
\def\AxPos{3.8}
\def\AxNeg{3.3}
\pgfmathsetmacro{\AxNegZ}{0.75*\AxNeg}

% tilted plane E = span(u,v) in ambient R^3
\def\ux{1}\def\uy{0.25}\def\uz{0.45}
\def\vx{-0.20}\def\vy{1}\def\vz{0.15}

\def\Ept#1#2{{(\ux)*(#1)+(\vx)*(#2)},{(\uy)*(#1)+(\vy)*(#2)},{(\uz)*(#1)+(\vz)*(#2)}}

% plane patch + grid
\path[plane]
  (\Ept{-\PlaneRad}{-\PlaneRad}) -- (\Ept{\PlaneRad}{-\PlaneRad}) --
  (\Ept{\PlaneRad}{\PlaneRad}) -- (\Ept{-\PlaneRad}{\PlaneRad}) -- cycle;

\foreach \t in {-4,-3,-2,-1,1,2,3,4}{
  \draw[faint] (\Ept{-\PlaneRad}{\t}) -- (\Ept{\PlaneRad}{\t});
  \draw[faint] (\Ept{\t}{-\PlaneRad}) -- (\Ept{\t}{\PlaneRad});
}

% H_T vertices
\coordinate (A1) at (\Ept{\HTT}{\HTone});
\coordinate (A2) at (\Ept{\HTone}{\HTT});
\coordinate (A3) at (\Ept{-\HTone}{\HTT});
\coordinate (A4) at (\Ept{-\HTT}{\HTone});
\coordinate (A5) at (\Ept{-\HTT}{-\HTone});
\coordinate (A6) at (\Ept{-\HTone}{-\HTT});
\coordinate (A7) at (\Ept{\HTone}{-\HTT});
\coordinate (A8) at (\Ept{\HTT}{-\HTone});

% H_T only
\filldraw[ht] (A1)--(A2)--(A3)--(A4)--(A5)--(A6)--(A7)--(A8)--cycle;

% axes
\draw[axneg] (0,0,0) -- (-\AxNeg,0,0);
\draw[axpos] (0,0,0) -- (\AxPos,0,0);
\node[black!100, below right] at (\AxPos,0,0) {$x_1$};

\draw[axneg] (0,0,0) -- (0,-\AxNeg,0);
\draw[axpos] (0,0,0) -- (0,\AxPos,0);
\node[black!100, above left] at (0,\AxPos,0) {$x_2$};

\draw[axneg] (0,0,0) -- (0,0,-\AxNegZ);
\draw[axpos] (0,0,0) -- (0,0,\AxPos);
\node[black!100, above] at (0,0,\AxPos) {$x_3$};

% origin
\fill[black!100] (0,0,0) circle (1.4pt);
\node[black!100, below left] at (0,0,0) {\scriptsize $0$};

% label
\node[black!100, font=\footnotesize] at (\Ept{-1.8}{-2.4}) {$H_T=\conv(K_T\cup M_T)$};

\end{tikzpicture}
\caption{$H_T$ as it lies in $E \in \Gr_{n,j}$ for $j = 2$, $n = 3$, reversing the identification $E \cong \mathbb{R}^j$.}
\end{figure}
To finish then, note that $\rho_\omega(K_T, M_T) \leq \rho_\omega(K_T, L_T) + \rho_\omega(L_T, M_T)$ is exactly 
\begin{align*}
&2\omega(H_T) - \omega(K_T) - \omega(M_T) \leq \omega(K_T) - \omega(L_T) + \left(\omega(M_T) - \omega(L_T) \right ) &\Longrightarrow \\
&\omega(H_T) +\omega(L_T) \leq \omega(K_T) + \omega(M_T),
\end{align*}
which contradicts our earlier analysis as $\omega = c \cdot \textup{vol}_k$ in this context. The triangle inequality for $\rho_\phi$ consequently fails, giving the desired result.
     \end{proof}
\end{theoremC}

Of course, this result suggests that closed forms for $D_\phi$ are somewhat elusive. As such, much of our argumentation in latter sections with rely on comparisons with more explicit metrics.   

\section{Intrinsicness and geodesicity}\label{5}

We turn now in this section to consider certain metric-geometric aspects of $D_\phi$. To start, we first formalize the intuition that the $D_\phi$ construction is like that of an intrinsic metric.

\subsection{Proof of intrinsicness}

\begin{theoremD}
    Let $\phi \in \textup{Val}$ be monotone. If $D_\phi$ is a metric on $\mathcal{K}^n_k$ for $k \in \{0, \dots, n\}$, then it is intrinsic/length.
    \begin{proof}
        We assume $D_\phi$ is a metric on $\mathcal{K}^n_k$ for $k \in \{0, \dots, n\}$ --- recall first that for $K, L \in \mathcal{K}^n$ with $K \subseteq L$ or $L \subseteq K$, we have that $D_\phi(K,L) = |\phi(K) - \phi(L)|$ by Proposition \ref{basic-props} part (i). For any comparable $K, L \in \mathcal{K}^n_k$ then, we consider
        $$\gamma: [0,1] \rightarrow \mathcal{K}^n_k, \quad \text{where} \quad \gamma(t) := (1-t)K + tL \quad \text{for all } t \in [0,1];$$ we claim that $\gamma$ is continuous (with respect to $D_\phi$) and has length $L(\gamma) = |\phi(K) - \phi(L)|$. To show this, assume without loss of generality that $K \subseteq L$ --- for $0 \leq t \leq s \leq 1$ we have $\gamma(t) \subseteq \gamma(s)$ and hence
        \begin{align*}
            D_\phi(\gamma(t), \gamma(s)) &= |\phi(\gamma(t)) - \phi(\gamma(s)) | \\
            &= \phi(\gamma(s)) - \phi(\gamma(t))
        \end{align*}
        as $\phi$ is monotone. As $\gamma$ is additionally itself Hausdorff continuous, this combined with the continuity of $\phi$ yields the desired $(D_\phi)$-continuity of $\gamma$. Moreover, if $\{t_i\}_{i = 0}^N$ is any partition of $[0,1]$, we see
        \begin{align*}
            \sum_{i = 1}^N D_\phi(\gamma(t_{i-1}), \gamma(t_i)) &= \sum_{i = 1}^N\phi(\gamma(t_i)) - \phi(\gamma(t_{i-1})) \\
            &= \phi(L) - \phi(K) \\
            &= |\phi(K) - \phi(L)|,
        \end{align*}
        and hence the length $L(\gamma)$ equals $|\phi(K) - \phi(L)|$.
        
        \begin{figure}[H]
\centering
\begin{tikzpicture}[scale=0.74,line cap=round,line join=round]
\tikzset{
  outline/.style={black,line width=1.00pt},
  inneroutline/.style={black,line width=0.90pt},
  Kfill/.style={fill=gray!18,opacity=.5},
  Lfill/.style={fill=white,opacity=.88},
  lab/.style={black,font=\small,transform shape=false},
}

\def\pan{6.85}
\def\sc{1.05}
\def\tval{0.5}

% Bodies are drawn from smooth support functions h_K and h_L.
% The middle curve is mathematically the Minkowski midpoint:
% h_{(1-t)K+tL}=(1-t)h_K+t h_L.
% For a smooth strictly convex support function h, the boundary point with
% outward normal u(theta) is h u(theta)+h' u'(theta).

\def\Kcoords{}
\def\Lcoords{}
\def\Mcoords{}

\foreach \a in {0,4,...,356}{
  % Larger body K: smooth, strictly convex, non-elliptical.
  \pgfmathsetmacro{\hK}
    {2.20 + 0.15*cos(2*\a) + 0.05*sin(3*\a)
          + 0.03*cos(5*\a) + 0.025*sin(4*\a)}
  \pgfmathsetmacro{\hpK}
    {-2*0.15*sin(2*\a) + 3*0.05*cos(3*\a)
          -5*0.03*sin(5*\a) + 4*0.025*cos(4*\a)}

  % Smaller body L: smooth, visibly not a dilate of K, and with h_L <= h_K.
  \pgfmathsetmacro{\hL}
    {1.25 + 0.08*cos(3*\a) - 0.06*sin(2*\a)
          + 0.025*cos(4*\a) + 0.015*sin(5*\a)
          + 0.18*cos(\a) - 0.10*sin(\a)}
  \pgfmathsetmacro{\hpL}
    {-3*0.08*sin(3*\a) - 2*0.06*cos(2*\a)
          -4*0.025*sin(4*\a) + 5*0.015*cos(5*\a)
          -0.18*sin(\a) - 0.10*cos(\a)}

  \pgfmathsetmacro{\xK}{\sc*(\hK*cos(\a) - \hpK*sin(\a))}
  \pgfmathsetmacro{\yK}{\sc*(\hK*sin(\a) + \hpK*cos(\a))}
  \pgfmathsetmacro{\xL}{\sc*(\hL*cos(\a) - \hpL*sin(\a))}
  \pgfmathsetmacro{\yL}{\sc*(\hL*sin(\a) + \hpL*cos(\a))}

  % True support-function/Minkowski interpolation at t=1/2.
  \pgfmathsetmacro{\xM}{(1-\tval)*\xK+\tval*\xL}
  \pgfmathsetmacro{\yM}{(1-\tval)*\yK+\tval*\yL}

  \xdef\Kcoords{\Kcoords (\xK,\yK)}
  \xdef\Lcoords{\Lcoords (\xL,\yL)}
  \xdef\Mcoords{\Mcoords (\xM,\yM)}
}

\begin{scope}[xshift=0*\pan cm]
  \path[use as bounding box] (-3.25,-3.15) rectangle (3.25,3.15);

  \fill[Kfill] plot[smooth cycle] coordinates {\Kcoords};
  \draw[outline] plot[smooth cycle] coordinates {\Kcoords};

  \fill[Lfill] plot[smooth cycle] coordinates {\Lcoords};
  \draw[inneroutline] plot[smooth cycle] coordinates {\Lcoords};

  \node[lab] at (-2.15,2.05) {$K$};
  \node[lab] at (0.10,0.05) {$L$};
\end{scope}

\begin{scope}[xshift=1*\pan cm]
  \path[use as bounding box] (-3.25,-3.15) rectangle (3.25,3.15);

  % Only the interpolated body and L are drawn here; K is not drawn.
  \fill[Kfill] plot[smooth cycle] coordinates {\Mcoords};
  \draw[outline] plot[smooth cycle] coordinates {\Mcoords};

  \fill[Lfill] plot[smooth cycle] coordinates {\Lcoords};
  \draw[inneroutline] plot[smooth cycle] coordinates {\Lcoords};

  \node[lab] at (-1.25,2.05) {$\tfrac12 K+\tfrac12 L$};
  \node[lab] at (0.10,0.05) {$L$};
\end{scope}

\begin{scope}[xshift=2*\pan cm]
  \path[use as bounding box] (-3.25,-3.15) rectangle (3.25,3.15);

  \fill[Kfill] plot[smooth cycle] coordinates {\Lcoords};
  \draw[outline] plot[smooth cycle] coordinates {\Lcoords};

  \node[lab] at (0.10,0.05) {$L$};
\end{scope}

\end{tikzpicture}
\caption{Convex interpolation between comparable bodies $K, L \in \mathcal{K}^n$ yields a $D_\phi$-continuous path with length $|\phi(K) - \phi(L)|$.}
\label{fig:intrinsic-interpolation}
\end{figure}
        
        To finish then, take any $K, L \in \mathcal{K}^n_k$ and any comparable chain $\{K_i\}_{i = 0}^N \subseteq \mathcal{K}^n$ between $K$ and $L$; note we impose no restrictions on the (affine) dimensions of the $K_i$. We form a new comparable chain
        $$\{K^\delta_i\}_{i = 0}^{N + 2} \subseteq \mathcal{K}^n_k, \quad K^\delta_i := \begin{cases}
            K & \text{if } i = 0,  \\
            K_{i-1} + \delta B^n & \text{if } 1\leq i \leq N + 1, \\
            L & \text{if } i = N + 2.
        \end{cases}$$
        for $\delta > 0$; this is indeed a comparable chain as $K \subseteq K + \delta B^n$, $L \subseteq L + \delta B^n$, and for $i \in \{1, \dots, N+1\}$
        $$K_{i - 1}^\delta \subseteq K_i^\delta \quad \Longleftrightarrow \quad K_{i-2} \subseteq K_{i-1}, \quad K_{i}^\delta \subseteq K_{i-1}^\delta \quad \Longleftrightarrow \quad K_{i-1} \subseteq K_{i-2},$$
        where $\{K_i\}_{i =0}^N$ is a zigzag. As $\phi$ is continuous then and $d_\mathcal{H}(K_i, K_i + \delta B^n) = \delta$ for each $i \in \{1, \dots, N\}$, for given $\epsilon > 0$ we might take $\delta > 0$ small enough that $\sum_{i = 0}^N |\phi(K_i + \delta B^n) - \phi(K_i)| < \nicefrac{\epsilon}{2}$. Thus 
        \begin{align*}
            \sum_{i = 1}^{N + 2} |\phi(K_{i-1}^\delta) - \phi(K^\delta_i)| &= |\phi(K + \delta B^n) - \phi(K)| + |\phi(L + \delta B^n) - \phi(L)| + \sum_{i = 2}^{N + 1} |\phi(K_{i-2} + \delta B^n) - \phi(K_{i-1} + \delta B^n)| \\
            &\leq \sum_{i = 1}^N |\phi(K_{i-1}) - \phi(K_i)| + 2\sum_{i = 0}^N |\phi(K_{i} + \delta B^n) - \phi(K_i)|\\
            &\leq \sum_{i=1}^N |\phi(K_{i-1}) - \phi(K_i)| + \epsilon
        \end{align*}
        by applying the triangle inequality. Consider then the maps $\{\gamma_i\}_{i = 1}^{N + 2}$, with $\gamma_i : [0,1] \rightarrow \mathcal{K}^n_k$ given by $t \mapsto (1-t)K^\delta _{i-1} + tK^\delta_i$. The concatenation $\bigast_{i = 1}^{N+2} \gamma_i$ is continuous by the pasting lemma and 
        \begin{align*}
        L \left (\bigast_{i = 1}^{N + 2} \gamma_i \right ) &= \sum_{i = 1}^{N + 2} L(\gamma_i)= \sum_{i = 1}^{N + 2} |\phi(K_{i-1}^\delta) - \phi(K_i^\delta)| \leq \sum_{i = 1}^N |\phi(K_{i-1}) - \phi(K_i)| + \epsilon
        \end{align*}
        by the additivity of length. As $\bigast_{i = 1}^{N + 2}$ joins $K$ and $L$ with its image entirely in $\mathcal{K}^n_k$ then , $\ol{D_\phi}(K,L) \leq L \left (\bigast_{i = 1}^{N + 2} \gamma_i \right) \leq \sum_{i = 1}^N |\phi(K_{i-1}) - \phi(K_i)| + \epsilon$. Taking $\epsilon \rightarrow 0$ and an infimum gets $\overline{D_\phi}(K,L) \leq D_\phi(K,L)$, which gives the desired result. 
    \end{proof}
\end{theoremD}

With intrinsicness comes then the question of geodesicity --- we will not have the latter in full generality, as $D_n = \Delta_{\textup{vol}}$ is easily shown to be \emph{not} geodesic when $n \geq 2$ --- nonetheless, we can study properness. 

\subsection{Properness and the supports of mixed area measures}

\begin{theoremE}
    Let $\phi \in \textup{Val}$ be monotone, and write 
    $$\phi_1(K) = \int_{S^{n-1}} h_K(u) \, dS(G_1, \dots, G_{n-1}; u) \quad \text{for all } K \in \mathcal{K}^n$$
    and fixed $G_1, \dots, G_{n-1} \in \mathcal{K}^n$. Then the support of the mixed area measure $\mu := S(G_1, \dots, G_{n-1}; \cdot)$ is not contained in any closed hemisphere, i.e. 
    \begin{align}\label{pos-on-open-hem}\mu\left(\left \{u \in S^{n-1} : \langle u, v \rangle > 0  \right\}\right) > 0 \quad \text{for all } v \in S^{n-1}, \end{align}
    if and only if for each $K \in \mathcal{K}^n$ and $r > 0$, the ball $B(K,r) := \{L \in \mathcal{K}^n : D_\phi(K, L) \leq r\}$ is Hausdorff bounded, i.e. there is $R \geq 0$ such that 
    $B(K,r) \subseteq \{L \in \mathcal{K}^n : d_\mathcal{H}(K,L) \leq R\}.$

    \begin{proof}
        We start with the forward direction --- let $K \in \mathcal{K}^n$, $r > 0$ and consider the ball 
        $B(K, r) = \{L \in \mathcal{K}^n : D_\phi(K,L) \leq r\}.$ We wish to show $B(K,r)$ is Hausdorff bounded; for this, first recall that for every $v \in \mathbb{R}^n$ with $v \neq 0$ one has 
        \begin{align*}
            \phi_1([0,v]) &= ||v|| \cdot \phi_1\left(\Big [0, \frac{v}{||v||} \Big] \right)  = ||v||\int_{S^{n-1}} \max \left \{0, \Big \langle u ,\frac{v}{||v||} \Big \rangle \right \} \, d\mu(u),
        \end{align*}
         given $1$-homogeneity of $\phi_1$. Moreover, this quantity is positive given the assumption that the support of $\mu$ is not contained in any closed hemisphere, i.e. (\ref{pos-on-open-hem}). We set then 
        $$m := \min_{w \in S^{n-1}} \left \{\int_{S^{n-1}} \max \{0, \langle u, w \rangle \}\,  d\mu(u) \right \} > 0;$$
        this minimum is achieved given the continuity of the term $F(w) : = \int_{S^{n-1}} \max\{0, \langle u, w  \rangle \} \, d\mu(u)$ in $w$ and the compactness of $S^{n-1}$. Indeed, for $w, w' \in S^{n-1}$ we have
        \begin{align*}
            \left |\int_{S^{n-1}} \max\{0, \langle u, w \rangle \} \, d\mu(u) - \int_{S^{n-1}} \max \{0, \langle u, w' \rangle \} \, d\mu(u) \right | &\leq \int_{S^{n-1}} \left |\max\{0, \langle u, w \rangle\} - \max\{0, \langle u, w' \rangle \right | \, d\mu(u) \\
            & \leq \int_{S^{n-1}} ||u|| \, ||w - w'|| \, d\mu(u) \\
            &= \mu(S^{n-1})||w - w'||
        \end{align*}
        as the function $t \mapsto \max\{0, t\}$ is $1$-Lipschitz and by Cauchy--Schwarz, and hence continuity follows. Using translation invariance then, our above work establishes that any segment $S$ has $\phi_1(S) \geq m\ell(S)$ with $\ell(S)$ the length of $S$. 
        
        Let then any $L \in B(K,r)$, and select $x \in K$ and $y \in L$ such that $||x - y|| = d_\mathcal{H}(K,L)$. By monotonicity we get 
        $$\phi_1(\textup{conv}(K \cup L)) \geq \phi_1([x,y]) \geq m||x - y|| = md_\mathcal{H}(K,L),$$
        and so
        $2md_\mathcal{H}(K,L) - \phi_1(K) - \phi_1(L) \leq \rho_{\phi_1}(K,L) \leq D_\phi(K,L);$
        rearranging and using $D_\phi(K, L) \leq r$ gets that $$d_\mathcal{H}(K,L) \leq \frac{r + \phi_1(K) + \phi_1(L)}{2m} \leq \frac{r + \phi(K) + \phi(L)}{2m}.$$ Moreover, as we know that $|\phi(K) - \phi(L)| \leq D_\phi(K, L)$, we have $\phi(L) \leq r + \phi(K)$ and thus
        $$d_\mathcal{H}(K,L) \leq \frac{r + \phi(K)}{m},$$
        i.e. $B(K,r) \subseteq \left \{L \in \mathcal{K}^n : d_\mathcal{H}(K, L) \leq \frac{r + \phi(K)}{m} \right \}$, as desired. We turn then to the reverse direction, which we demonstrate by means of the contrapositive --- for this, assume $\mu$ vanishes on some open hemisphere $H_v := \{u \in S^{n-1} : \langle u, v \rangle > 0\}$, i.e. there is $v \in S^{n-1}$ such that 
        $$\mu(H_v) = \mu \left ( \left \{u \in S^{n-1}: \langle u, v \rangle > 0 \right \} \right )  = 0.$$
        Take then any $K \in \mathcal{K}^n$ and $x \in K$, and consider $K_t : = K + t\{v\}$ for $t \geq 0$; obviously $d_\mathcal{H}(K, K_t)$ diverges as $t \rightarrow +\infty$, so it suffices to show that $D_\phi(K, K_t)$ remains bounded.

\begin{figure}[H]
\centering
\begin{tikzpicture}[scale=0.9,line cap=round,line join=round]
\tikzset{
  outline/.style={black,line width=0.95pt},
  sphere/.style={black,dashed,line width=0.95pt},
  arcglow/.style={draw=gray!28,line width=5.8pt},
  Kfill/.style={fill=gray!18,opacity=.5},
  normal/.style={black,line width=0.9pt,->},
  pt/.style={circle,fill=black,inner sep=1.25pt,transform shape=false},
  lab/.style={black,font=\small,transform shape=false},
  faintlab/.style={black!65,font=\small,transform shape=false},
  hpt/.style={circle,draw=black,fill=white,line width=0.85pt,inner sep=1.6pt,transform shape=false},
}

\def\rad{2.05}
\def\tv{4.55}
\def\Kcoords{%
(-0.72,-0.38) (-0.38,-0.70) (0.12,-0.80) (0.58,-0.63)
(0.86,-0.24) (0.90,0.18) (0.68,0.56) (0.26,0.82)
(-0.20,0.84) (-0.62,0.62) (-0.88,0.24) (-0.90,-0.12)
}

\path[use as bounding box] (-3.05,-3.70) rectangle (7.20,3.10);

% Light circular highlight of the arc H_v.
\draw[arcglow] (90:\rad) arc[start angle=90,end angle=-90,radius=\rad];

% The circle of directions.
\draw[sphere] (0,0) circle (\rad);

% A generic convex body K.
\begin{scope}[rotate=-90]
  \fill[Kfill] plot[smooth cycle] coordinates {\Kcoords};
  \draw[outline] plot[smooth cycle] coordinates {\Kcoords};
\end{scope}
\node[lab] at (0.02,0.02) {$K$};

% The translated copy K_t.
\begin{scope}[shift={(\tv,0)},rotate=-90]
  \fill[Kfill] plot[smooth cycle] coordinates {\Kcoords};
  \draw[outline] plot[smooth cycle] coordinates {\Kcoords};
\end{scope}
\node[lab] at ({\tv+0.02},0.02) {$K_t$};

% The pole of the open hemisphere, supporting S^1 normally.
\coordinate (V) at (\rad,0);
\node[pt] at (V) {};
\draw[normal] (V) -- ++(0.72,0) node[lab,right=2pt] {$v$};

% Label for the highlighted arc H_v.
\node[faintlab] at (35:{\rad+0.4}) {$H_v$};

% The mixed area measure is concentrated on the boundary of H_v.
\node[hpt] (Eminus) at (0,\rad) {};
\node[hpt] (Eplus)  at (0,-\rad) {};

\draw[normal] (Eminus.north) -- ++(0,0.72) node[lab,above=2pt] {$e_2$};
\draw[normal] (Eplus.south) -- ++(0,-0.72) node[lab,below=2pt] {$-e_2$};

\node[lab,below=36pt] at (0,-\rad)
  {$\mu(H_v)=0,\quad \supp\mu\subseteq\{\pm e_2\}$};

\end{tikzpicture}
\caption{An example where $\mu$ is the surface area measure induced by a horizontal line segment --- it vanishes on the open hemisphere $H_v$ corresponding to $v = e_1$, and so $K_t = K + t\{v\}$ can ``hide'' from it.}
\end{figure}

        For this, define the segment $S_t : = [x, x + tv]$. We have
        $$\phi_1(S_t) = \phi_1([0, tv]) = t \int_{S^{n-1}} \max\{0, \langle u, v \rangle \} \, d\mu(u) = t\int_{H_v} \langle u, v \rangle \, d\mu(u) = 0$$
        by hypothesis, and moreover $\phi_j(S_t) = 0$ for $j \geq 2$ as noted earlier --- hence $\phi(S_t) = \phi_0(S_t)$. To finish then, we consider the comparable chain $\{K, \{x\}, S_t, \{x + tv\}, K_t\}$; summing over this gets 
        $$\phi(K) + \phi(K_t)+ 2\phi(S_t) - 2\phi(\{x\}) - 2\phi(\{x + tv\}) \leq 2\phi(K),$$
        the right bound holding by translation invariance as the $\phi_0$ contributions cancel and $\sum_{j = 1}^n \phi_j(S_t) = 0$. Thus $D_\phi(K, K_t) \leq 2\phi(K)$ for all $t \geq 0$, and so the ball $B(K, 2\phi(K))$ is Hausdorff \emph{unbounded}, finishing the proof. 
\end{proof}
\end{theoremE}

Theorem \ref{bounded-balls} is interesting in and of itself as it gives an equivalent condition for an \emph{arbitrary} mixed area measure to have its support not constrained to a closed hemisphere. Indeed, if $G_1, \dots, G_{n-1} \in \mathcal{K}^n$, then the valuation $\psi \in \textup{Val}_1$ given by 
$$K \mapsto \int_{S^{n-1}} h_K(u) \, dS(G_1, \dots, G_{n-1}; u) \quad \text{for all } K \in \mathcal{K}^n$$
is monotone, and hence Theorem \ref{bounded-balls} provides a means to discuss a simple aspect of the geometry of the support of $S(G_1, \dots, G_{n-1}; \cdot)$ via a simple aspect of the geometry of the $D_\psi$-balls. 

That all notwithstanding, our primary use for Theorem \ref{bounded-balls} is that it exactly furnishes a condition under which $(\mathcal{K}^n, D_\phi)$ is \emph{proper} --- and hence complete and geodesic --- as we note now.

\subsection{Geodesicity and metrics made from $\mu_G$}

\begin{coroF}
    Let $\phi \in \textup{Val}$ be monotone, and assume $D_\phi$ is a metric on $\mathcal{K}^n$. Write
    $$\phi_1(K) = \int_{S^{n-1}} h_K(u) \, dS(G_1, \dots, G_{n-1}) \quad \text{for all } K \in \mathcal{K}^n$$
    and fixed $G_1, \dots, G_{n-1} \in \mathcal{K}^n$. If the mixed area measure $\mu := S(G_1, \dots, G_{n-1}; \cdot)$ is such that its support is not contained within any closed hemisphere, then $D_\phi$ is proper, and consequently complete and geodesic. 
    \begin{proof}
        To show that $D_\phi$ is proper, it is enough to show that the closed $D_\phi$-balls are compact. For this, let $\mathcal{K}^n$, $r > 0$, and consider again the closed ball
        $$B(K,r) = \{L \in \mathcal{K}^n : D_\phi(K, L) \leq r\};$$
        we take $\{L_i\}_{i = 0}^\infty \subseteq B(K,r)$. It is clear Theorem \ref{bounded-balls} implies $\bigcup_{L \in B(K,r)} L$ lies in some bounded Euclidean region --- in particular, one has 
        \begin{align*}d_\mathcal{H}(K,L) \leq \frac{r + \phi(K)}{m} \quad &\Longrightarrow \quad L \subseteq K + \left (\frac{r + \phi(K)}{m} \right )B^n \\
          &\Longrightarrow \max_{\ell \in L} ||\ell|| \leq \max_{k \in K} ||k|| + \frac{r + \phi(K)}{m} := M \end{align*}
        for each $L \in B(K,r)$, where $m > 0$ is as in the proof of Theorem \ref{bounded-balls} --- that is that every such $L$ lies in a Euclidean disc of diameter $2M$. 
        
        Consequently, one might apply the Blaschke selection theorem to the sequence $\{L_i\}_{i = 0}^\infty$ --- that is, we have that there is some $L \in \mathcal{K}^n$ and subsequence $\{L_{i_j}\}_{j =0}^\infty \subseteq B(K,r)$ such that 
        $$d_\mathcal{H}(L_{i_j}, L) \rightarrow 0 \quad \text{as} \quad l \rightarrow + \infty;$$
        finally then, by Proposition \ref{basic-props} (ii), we note it must be that $D_\phi(L_{i_j}, L) \rightarrow 0$ as $l \rightarrow +\infty$; as $B(K,r)$ is closed, we get then $L \in B(K,r)$, and so $B(K,r)$ is compact --- hence $D_\phi$ is proper as $K \in \mathcal{K}^n$ and $r > 0$ were arbitrary.
        
        Properness then of course has completeness, and as Theorem \ref{intrinsic} implies $D_\phi$ is intrinsic, it is also geodesic now by the metric space version of Hopf--Rinow. 
    \end{proof}
\end{coroF}

Corollary \ref{geodesicity}, of course, substantially generalizes the result of Florian in \cite{florian} that $\rho_w = D_w$ is geodesic, as the spherical Lebesgue measure $\omega$ is fully supported on $S^{n-1}$ --- in particular,
\begin{align*}w(K) &= \frac{2}{n \kappa_n} \int_{S^{n-1}} h_K(u) \, d\omega(u) \\ &= \int_{S^{n-1}} h_K(u) \, dS \Big (\frac{2}{n\kappa_n} B^n, B^n[n-2]; u \Big)\end{align*} for all $K \in \mathcal{K}^n$,
where the involved mixed area measure is clearly fully supported. That notwithstanding, the real utility of the result lies in allowing us to say something when $\phi_1$ is given by a more general \emph{surface area measure}, as we see now.

\begin{example}\label{mu-G-ex}
    Let $G \in \mathcal{K}^n_n$, and consider $\mu_G \in \textup{Val}$ where 
    \begin{align*}\mu_G(K) = \textup{vol}(G + K)  \quad \text{for all } K \in \mathcal{K}^n.\end{align*}
    Then $D_\phi$ is a complete, geodesic metric on $\mathcal{K}^n$.
    \begin{proof}
        Recall we have the Steiner formula expansion (\ref{mixed-vol-eq}) for $G$ and $K$; considering this at $t = 1$ yields
        $$\mu_G(K) = \textup{vol}(G + K) = \sum_{j = 0}^n \binom n j \, \textup{V}_j(G,K) \quad \text{for all } K \in \mathcal{K}^n,$$
        and hence the mixed area measure associated to the $1$-homogeneous component of $\phi$ is $dS(G[n-1]; \cdot)$. With that in mind, we recall the well-known fact that the surface area measure $dS(G[n-1]; \cdot)$ \emph{cannot} have its support constrained to a closed hemisphere of $S^{n-1}$ --- see for instance \cite{schneider2014} Chapter 8.

\begin{figure}[H]
\centering
\begin{tikzpicture}[scale=0.55,line cap=round,line join=round]
\tikzset{
  outline/.style={black,line width=1.0pt},
  Kfill/.style={fill=gray!18,opacity=.5},
  normal/.style={black,->,line width=0.85pt},
}

\def\pan{6.6}
\def\R{1.95}
\def\rB{1.85}
\def\nscale{1.35}
\def\triN{1.15}

\def\Rarc{2.55}
\def\gap{15}

%---------------------------------------
% Left panel: simplex (right triangle)
%---------------------------------------
\begin{scope}[xshift=0*\pan cm]
  \path[use as bounding box] (-2.55,-2.55) rectangle (2.55,2.55);

  % Affine copy of the unit simplex
  \coordinate (A) at (-1.55,-1.45);
  \coordinate (B) at ( 1.55,-1.45);
  \coordinate (C) at (-1.55, 1.65);

  \fill[Kfill] (A)--(B)--(C)--cycle;
  \draw[outline] (A)--(B)--(C)--cycle;

  % Midpoints of facets
  \coordinate (MAB) at ($(A)!0.5!(B)$);
  \coordinate (MAC) at ($(A)!0.5!(C)$);
  \coordinate (MBC) at ($(B)!0.5!(C)$);

  % Outward facet normals (all same length)
  \draw[normal] (MAB) -- ++(0,-\triN);
  \draw[normal] (MAC) -- ++(-\triN,0);
  \draw[normal] (MBC) -- ++({\triN/sqrt(2)},{\triN/sqrt(2)});
\end{scope}

%---------------------------------------
% Middle panel: pentagon
%---------------------------------------
\begin{scope}[xshift=1*\pan cm]
  \path[use as bounding box] (-2.55,-2.55) rectangle (2.55,2.55);

  \coordinate (P1) at (90:\R);
  \coordinate (P2) at (162:\R);
  \coordinate (P3) at (234:\R);
  \coordinate (P4) at (306:\R);
  \coordinate (P5) at (18:\R);

  \fill[Kfill] (P1)--(P2)--(P3)--(P4)--(P5)--cycle;
  \draw[outline] (P1)--(P2)--(P3)--(P4)--(P5)--cycle;

  \coordinate (M12) at ($(P1)!0.5!(P2)$);
  \coordinate (M23) at ($(P2)!0.5!(P3)$);
  \coordinate (M34) at ($(P3)!0.5!(P4)$);
  \coordinate (M45) at ($(P4)!0.5!(P5)$);
  \coordinate (M51) at ($(P5)!0.5!(P1)$);

  \draw[normal] (M12) -- ($(0,0)!\nscale!(M12)$);
  \draw[normal] (M23) -- ($(0,0)!\nscale!(M23)$);
  \draw[normal] (M34) -- ($(0,0)!\nscale!(M34)$);
  \draw[normal] (M45) -- ($(0,0)!\nscale!(M45)$);
  \draw[normal] (M51) -- ($(0,0)!\nscale!(M51)$);
\end{scope}

%---------------------------------------
% Right panel: disk / hemisphere picture
%---------------------------------------
\begin{scope}[xshift=2*\pan cm]
  \path[use as bounding box] (-2.90,-2.90) rectangle (2.90,2.90);

  \fill[Kfill] (0,0) circle (\rB);
  \draw[outline] (0,0) circle (\rB);

  \draw[outline] (90+\gap:\Rarc) arc[start angle=90+\gap,end angle=450-\gap,radius=\Rarc];

  \draw[normal] (0,\rB) -- (0,\Rarc);
\end{scope}

\end{tikzpicture}
\caption{One cannot find a convex body $G \in \mathcal{K}^n_n$ such that the support of its associated surface area measure $S(G[n-1]; \cdot)$ is contained within a closed hemisphere of $S^{n-1}.$}
\end{figure}
As we know $D_{\mu_G}$ is a metric on $\mathcal{K}^n$ then by Theorem \ref{metricity-conditions} part (iv), Corollary \ref{geodesicity} applies, and thus we have the desired result. 
    \end{proof}
\end{example}

It is worth noting that this result, as well as the previous result establishing the metricity of $D_{\mu_G}$ and the upcoming result concerning its topology, are applicable to the related valuation $\textup{M}_G(K) := \sum_{j = 0}^n \textup{V}_j(G, K)$ considered in \cite{us} --- here $G \in \mathcal{K}^n_n$. Indeed,
\begin{align*}
    \Delta_{\mu_G}(K,L) &= \textup{vol}(G + K) + \textup{vol}(G + L) - 2\textup{vol}(G + (K \cap L)) \\
    &= \sum_{j = 0}^n \binom n j \left ( \textup{V}_j(G, K) + \textup{V}_j(G, L) - 2\textup{V}_j(G, K \cap L) \right ) \\
    &\leq \binom{n}{\lfloor \nicefrac{n}{2} \rfloor} \sum_{j  = 0}^n \left( \textup{V}_j(G, K) + \textup{V}_j(G, L) - 2\textup{V}_j(G, K \cap L) \right)\\
    &= \binom{n}{\lfloor \nicefrac{n}{2} \rfloor} \, \Delta_{\textup{M}_G}(K,L)
\end{align*}
for any $K, L \in \mathcal{K}^n$, and applying this to the formulation for $D_{\phi}$ in Proposition \ref{monotone-props} part (a) gives the estimate $\binom{n}{\lfloor\nicefrac{n}{2}\rfloor}^{-1}D_{\mu_G}\leq D_{M_G}$, which
implies nondegeneracy of $D_{\, \textup{M}_G}$ on $\mathcal{K}^n$. As the mixed area measure for $\textup{M}_G$ in the context of Corollary \ref{geodesicity} then is just $\nicefrac{1}{n} \cdot dS(G[n-1]; \cdot)$, we see the corollary is still applicable, and hence $D_{\, \textup{M}_G}$ \emph{is} geodesic.

\begin{figure}[H]
\centering
\begin{tikzpicture}[scale=0.8,line cap=round,line join=round]
\tikzset{
  outline/.style={black,line width=1.00pt},
  Kfill/.style={fill=gray!18,opacity=.55},
  perim/.style={black,line width=1.20pt},
  edge/.style={black,line width=0.95pt},
  vtx/.style={circle,fill=black,inner sep=1.25pt},
  lab/.style={black,font=\small},
  pluslab/.style={black,font=\Large},
}

% ---------------------------------
% independent horizontal controls
% ---------------------------------
\def\xShape{1.65}       % polygon figures only
\def\xPlus{1.15}        % vertical plus signs only
\def\xMidPlus{1.15}     % plus sign between the two V_1 figures only
\def\xLabelLeft{-3.00}   % V_j labels in left panel only
\def\xLabelRight{-1.50}  % V_j labels in right panel only

% ---------------------------------
% layout
% ---------------------------------
\def\panelsep{10.8}

\def\yTop{0.00}
\def\yMid{-3.35}
\def\yBot{-6.55}

\def\yPlusTop{-1.12}
\def\yPlusBot{-4.38}

\def\yLabelOffset{0.62}

% separation of the two V_1 copies in the left panel
\def\dxVOne{2.05}
\pgfmathsetmacro{\xVOneLeft}{\xShape-\dxVOne}
\pgfmathsetmacro{\xVOneRight}{\xShape+\dxVOne}

% ---------------------------------
% polygon asset
% ---------------------------------
\newcommand{\MakePoly}{
  \coordinate (A) at (-1.42,-0.30);
  \coordinate (B) at ( 0.16,-0.08);
  \coordinate (C) at ( 0.42, 1.02);
  \coordinate (D) at (-0.32, 1.62);
  \coordinate (E) at (-1.30, 0.90);
}
\newcommand{\PolyPath}{(A)--(B)--(C)--(D)--(E)--cycle}

\newcommand{\VolAsset}[2]{%
  \begin{scope}[shift={(#1,#2)}]
    \MakePoly
    \fill[Kfill] \PolyPath;
    \draw[outline] \PolyPath;
  \end{scope}
}

\newcommand{\PerimAsset}[2]{%
  \begin{scope}[shift={(#1,#2)}]
    \MakePoly
    \draw[perim] \PolyPath;
  \end{scope}
}

\newcommand{\EulerAsset}[2]{%
  \begin{scope}[shift={(#1,#2)}]
    \MakePoly
    \draw[edge] \PolyPath;
    \node[vtx] at (A) {};
    \node[vtx] at (B) {};
    \node[vtx] at (C) {};
    \node[vtx] at (D) {};
    \node[vtx] at (E) {};
  \end{scope}
}

% ---------------------------------
% bounding box
% ---------------------------------
\path[use as bounding box] (-5.2,-7.7) rectangle ({\panelsep+5.2},2.1);

% =================================
% LEFT PANEL: V_2 + 2V_1 + V_0
% =================================
\begin{scope}[shift={(0,0)}]

  \node[lab] at (\xLabelLeft,\yTop+\yLabelOffset) {$\textup{V}_2$};
  \VolAsset{\xShape}{\yTop}

  \node[pluslab] at (\xPlus,\yPlusTop) {$+$};

  \node[lab] at (\xLabelLeft,\yMid+\yLabelOffset) {$\textup{V}_1$};
  \PerimAsset{\xVOneLeft}{\yMid}
  \node[pluslab] at (\xMidPlus,\yMid+\yLabelOffset) {$+$};
  \PerimAsset{\xVOneRight}{\yMid}

  \node[pluslab] at (\xPlus,\yPlusBot) {$+$};

  \node[lab] at (\xLabelLeft,\yBot+\yLabelOffset) {$\pi \textup{V}_0$};
  \EulerAsset{\xShape}{\yBot}

\end{scope}

% =================================
% RIGHT PANEL: V_2 + V_1 + V_0
% =================================
\begin{scope}[shift={(\panelsep,0)}]

  \node[lab] at (\xLabelRight,\yTop+\yLabelOffset) {$\textup{V}_2$};
  \VolAsset{\xShape}{\yTop}

  \node[pluslab] at (\xPlus,\yPlusTop) {$+$};

  \node[lab] at (\xLabelRight,\yMid+\yLabelOffset) {$\textup{V}_1$};
  \PerimAsset{\xShape}{\yMid}

  \node[pluslab] at (\xPlus,\yPlusBot) {$+$};

  \node[lab] at (\xLabelRight,\yBot+\yLabelOffset) {$\pi \textup{V}_0$};
  \EulerAsset{\xShape}{\yBot}

\end{scope}

\end{tikzpicture}
\caption{Comparison of $\mu_{B^n}$ vs. $\textup{M}_{B^n}$ in $\mathbb{R}^n$; the latter may be thought of as the ``unweighted'' version of the former.}
\end{figure}

Given the complex-geometric origins of $\rho_{\, \textup{M}_G}$ described in \cite{us}, it is interesting to consider the $D_{\, M_G}$ construction in the setting of \cite{darvas2021metric} --- we elaborate on this in Section \ref{7}. More simply, one might ask about some specifics of  $D_{\mu_G}$ in its native \emph{convex} setting, e.g. what its geodesics are.

\begin{openq}
    Fix $G \in \mathcal{K}^n_n$; if $\mu_G(K) = \textup{vol}(G + K)$ for all $K \in \mathcal{K}^n$, does there exist a ``good'' description of the geodesics of $D_{\mu_G}$? That is, can one construct and characterize the geodesics of $D_{\mu_G}$?
\end{openq}

We progress slightly on this question now in Section \ref{6} by showing that, for $G \in \mathcal{K}^n_n$, $D_{\mu_G}$ always gives the same topology on $\mathcal{K}^n$ as the Hausdorff distance.

\section{Topological notions}\label{6}

\subsection{Comparison of notions of convergence}

Here now, we turn to considering the topology of $D_\phi$ when it happens to be a metric. Our main result here is the following theorem, which gives some connections between different forms of convergence.

\begin{theorem}\label{convergence}
Let $\phi \in \textup{Val}$ be monotone. Then the following implications hold,
% https://q.uiver.app/#q=WzAsNCxbNCwwLCJEX1xccGhpLVxcdGV4dHtjb252ZXJnZW5jZX0iXSxbMCw0LCJcXERlbHRhX1xccGhpLVxcdGV4dHtjb252ZXJnZW5jZX0iXSxbOCw0LCJcXHJob19cXHBoaS1cXHRleHR7Y29udmVyZ2VuY2V9Il0sWzQsOCwiZF9cXG1hdGhjYWx7SH0tXFx0ZXh0e2NvbnZlcmdlbmNlfSJdLFsyLDAsIihpKSIsMix7ImN1cnZlIjo1fV0sWzEsMCwiKHYpIiwwLHsiY3VydmUiOi01fV0sWzMsMSwiKHZpKSIsMCx7ImN1cnZlIjotNX1dLFsxLDMsIihpaWkpIiwwLHsiY3VydmUiOi01fV0sWzMsMCwiKGkpIiwwLHsiY3VydmUiOi01fV0sWzAsMywiKGlpKSwgKGlpaSksIChpdiksICh2aWkpIiwwLHsiY3VydmUiOi01fV0sWzMsMiwiKGkpIiwwLHsiY3VydmUiOi01fV0sWzIsMywiKGl2KSIsMCx7ImN1cnZlIjotNX1dXQ==
\[\begin{tikzcd}
	&&&& {D_\phi-\text{convergence}} &&&& \\
	\\
	\\
	\\
	{\Delta_\phi-\text{convergence}} &&&&&&&& {\rho_\phi-\text{convergence}} \\
	\\
	\\
	\\
	&&&& {d_\mathcal{H}-\text{convergence}}
	\arrow["{(ii), (iii), (iv), (viii)}", curve={height=-30pt}, from=1-5, to=9-5]
	\arrow["{(v)}", curve={height=-30pt}, from=5-1, to=1-5]
	\arrow["{(iii),(vi)}", curve={height=-30pt}, from=5-1, to=9-5]
	\arrow["{(i)}"', curve={height=30pt}, from=5-9, to=1-5]
	\arrow["{(iv)}", curve={height=-30pt}, from=5-9, to=9-5]
	\arrow["{(i)}", curve={height=-30pt}, from=9-5, to=1-5]
	\arrow["{(vii)}", curve={height=-30pt}, from=9-5, to=5-1]
	\arrow["{(i)}", curve={height=-30pt}, from=9-5, to=5-9]
\end{tikzcd}\]

where the criteria are 
\begin{enumerate}[label=(\roman*)]
    \item on $\mathcal{K}^n$ always, 
    \item on $\mathcal{K}^n_k$ when $\phi \in \textup{Val}^{+,\infty}$ and we may choose nonnegative representing functions $f_j$ in (\ref{grassman-comp}), with at least one $f_k>0$ on $\Gr_{n,k}$ for some $k\in\{1,\dots,n\}$, and 
    \item on $\mathcal{K}^n_n$ when $\phi_n \not \equiv 0$
    \item on $\mathcal{K}^n$ when, in the context of the representation 
    $$\phi_1(K) = \int_{S^{n-1}} h_K(u) \textup{D}(D^2 h_{G_1}(u), \dots D^2 h_{G_{n-1}}(u)) \, d\omega(u) \quad \text{for all } K \in \mathcal{K}^n$$
    given some $G_1, \dots, G_{n-1} \in \mathcal{K}^n$, it is that the involved mixed discriminant is continuous and strictly positive on $S^{n-1}$ --- this holds in particular when $G_1, \dots, G_{n-1}$ are $C^2_+$,
    \item on $\mathcal{K}^n$ when $\phi_1 \equiv 0$
    \item on $\mathcal{K}_k(H)$ when $\phi_k \in \textup{Val}^+$, $k \in \{1, \dots n\}$, $H \in \Gr_{n,k}$ is a $k$-dimensional subspace, and $\textup{Kl}_{\phi_k}(H) > 0$.
    \item on $\mathcal{K}_k(H)$ when $\phi = \phi_k \in \textup{Val}^+$, $k \in \{1, \dots n\}$, $H \in \Gr_{n,k}$ is a $k$-dimensional subspace, and $\textup{Kl}_{\phi_k}(H) > 0$.
    \item on $\mathcal{K}^n$ when $\phi(K) = \mu_G(K)$ for some $G \in \mathcal{K}^n_n$
\end{enumerate}

Here $d_1-\text{convergence} \xrightarrow{(\ast)} d_2-\text{convergence}$ on $\mathcal{K}^n_k$ for $k \in \{0, \dots, n\}$ means that, if $\{K_i\}_{i = 0}^\infty \subseteq \mathcal{K}^n_k$ and $K \in \mathcal{K}^n_k$, we have 
$$d_1(K_i, K) \rightarrow 0 \quad \text{as} \quad i \rightarrow +\infty \quad \Longrightarrow \quad d_2(K_i, K) \rightarrow 0 \quad \text{as} \quad i \rightarrow +\infty$$
under the hypothesis $(\ast)$.
\end{theorem}

Most parts of Theorem \ref{convergence} are straightforward to prove using the same strategies we employed to prove Theorem \ref{metricity-conditions}. There is one technical aspect, however, that needs to be addressed if we wish to work with criteria (ii). 

Indeed, recall that in the proof of Theorem \ref{metricity-conditions} that the critical lower bound under the assumption of (ii) was
    \begin{align*}D_\phi(K,L) \geq m \cdot \int_{\Gr_{n,k}} \textup{vol}_k(\pi_H(K) \, \triangle \, \pi_H(L)) \, d\nu_k(H),\end{align*}
    where $K, L \in \mathcal{K}^n_k$ and $m > 0$. Given that the right-hand term is a metric on $\mathcal{K}^n_k$, this was enough to establish that $D_\phi$ itself was a metric on $\mathcal{K}^n_k$. To conclude something topological from this, however, we would need to show that say
    \begin{align*}\int_{\Gr_{n,k}} \textup{vol}_k(\pi_H(K_i) \, \triangle \, \pi_H(K)) \, d\nu_k(H) \rightarrow 0\end{align*} has Hausdorff convergence of the $K_i$ to $K$, where $\{K_i\}_{i = 0}^\infty \subseteq \mathcal{K}^n_k$ and $K \in \mathcal{K}^n_k$. This result, however, would exactly constitute an (affirmative) resolution to Open Question 2.2 in \cite{besau2024} --- we will give exactly this resolution now as our Theorem \ref{delta-j-haus-equiv}, for which we start by notationally condensing things. 

    \subsection{Metrics made from intrinsic volumes and Hausdorff convergence}

    \begin{defn}[\cite{besau2024} Definition 1.1]
    Let $j \in \{1, \dots, n\}$. Define
    \begin{align}\label{delta-j}\delta_j(K,L) = {n \brack j } \int_{\Gr_{n,j}} \textup{vol}_j(\pi_H(K) \, \triangle \, \pi_H(L)) \, d\nu_j(H) \quad \text{for all } K, L \in \mathcal{K}^n,\end{align}
    where ${n \brack j} := \binom n j \frac{\kappa_n}{\kappa_j \kappa_{n-j}}$ is the flag coefficient.
\end{defn}

In this language then, the desired result is the following.

\begin{theorem}[\cite{besau2024} Open Question 2.2]\label{delta-j-haus-equiv}
Let $\{K_i\}_{i = 0}^\infty \subseteq \mathcal{K}^n_j$, and $K \in \mathcal{K}^n_j$. Then 
$$\delta_j(K_i, K) \rightarrow 0 \quad \text{as} \quad i \rightarrow +\infty \quad \Longrightarrow \quad d_\mathcal{H}(K_i, K) \rightarrow 0 \quad \text{as} \quad i \rightarrow +\infty.$$
\end{theorem}

Of course, this gives the result we ultimately desire, as the flag coefficient ${n \brack j}$ is positive. To demonstrate Theorem \ref{delta-j-haus-equiv}, however, will first necessitate us to accumulate some useful lemmas regarding orthogonal projections and how they interact with support functions, which we do now.

\begin{lemma}[Basic projection lemmas]\label{basic-proj-lemmas}
 Let $j \in \{1, \dots n\}$, and take $H \in \Gr_{n,j}$. We have that 
    \begin{enumerate}[label=(\roman*)]
        \item for every $x \in \mathbb{R}^n$ and $v \in H$ that 
        $$\langle v, x\rangle = \langle v, \pi_H(x) \rangle,$$
        \item for every $K \in \mathcal{K}^n$ \begin{align}
            h_{\pi_H(K)}(v) = h_K(v) \quad \text{for all } v \in H, 
        \end{align}
        and
        \item if $K, L \in \mathcal{K}(H)$, then 
        $$d_\mathcal{H}(K,L) = \sup_{u \in S^{n-1} \cap H} |h_K(u) - h_L(u)|.$$
    \end{enumerate}
    \begin{proof}
        For (i), take $x \in \mathbb{R}^n$ and $v \in H$. By the definition of orthogonal projection, we have that 
        $x = \pi_H(x) + (x - \pi_H(x))$,
        where $\pi_H(x) \in H$ and $x - \pi_H(x) \in H^\perp$. As $v \in H$ then, it follows $\langle v, x - \pi_H(x) \rangle = 0$, and so 
        $$\langle v, x \rangle = \langle v, \pi_H(x) \rangle + \langle v, x - \pi_H(x) \rangle = \langle v, \pi_H(x) \rangle;$$
        the desired result thus follows. For (ii) then, note we have immediately
        $$h_{\pi_H(K)}(v) = \sup_{x \in K} \, \langle v, \pi_H(x) \rangle = \sup_{x \in K} \, \langle v, x \rangle = h_K(v),$$
        the central equality following from (i). Finally, for (iii), recall that 
        $$d_\mathcal{H}(K,L) = \sup_{u \in S^{n-1}} |h_K(u) - h_L(u)| \quad \text{for all } K, L \in \mathcal{K}^n;$$
        take then $K, L \in \mathcal{K}(H)$, and fix $u \in S^{n-1}$ --- we additionally write $w := \pi_H(u) \in H$. For any $v \in K \subseteq H$, we know by (i) that $\langle u, v \rangle = \langle w, v \rangle$; hence, taking suprema, we have $h_{K}(u) = h_{K}(w)$ and similarly $h_{L}(u) = h_{L}(w)$. It follows then
        \begin{align*}|h_K(u) - h_L(u)| = |h_K(w) - h_L(w)|;\end{align*}
        if $w = 0$, $h_K(w) = h_L(w) = 0$, and hence $|h_K(u) - h_L(u)| = 0$. If $w \neq 0$, we write $\hat w := \nicefrac{w}{||w||} \in S^{n-1} \cap H$ --- applying then positive homogeneity of support functions gets
        \begin{align*}|h_K(w) - h_L(w)| &= ||w|| \cdot |h_K(\hat w) - h_L(\hat w)| \\ &\leq |h_K(\hat w) - h_L(\hat w)|,\end{align*}
        as $||w|| = ||\pi_H(u)|| \leq ||u|| = 1$, given that $\pi_H$ is a contraction. It follows then in either case
        $$|h_K(u) - h_L(u)| \leq \sup_{\xi \in S^{n-1} \cap H } |h_K(\xi) - h_L(\xi)|;$$
        the assumption $j \geq 1$ of course guaranteeing $S^{n-1} \cap H$ is nonempty. Taking a supremum then, we have $d_\mathcal{H}(K, L) \leq \sup_{u \in S^{n-1} \cap H} |h_K(u) - h_L(u)|$, with the reverse inequality being immediate given $S^{n-1} \cap H \subseteq S^{n-1}$. The desired result thus follows.  
    \end{proof}
\end{lemma}

These lemmas provide the framework by which to establish a useful sufficient (and necessary) condition for a sequence $\{K_i\}_{i = 0}^\infty \subseteq \mathcal{K}^n$ to converge to fixed $K \in \mathcal{K}^n$ in the Hausdorff distance. For our purposes, it will be most useful to phrase this condition in the contrapositive. 

\begin{lemma}\label{far-away-projs}
    Let $\{K_i\}_{i = 0}^\infty \subseteq \mathcal{K}^n$ and $K \in \mathcal{K}^n$ such that $K_i$ does not converge to $K$ in the Hausdorff distance, and let $j \in \{1, \dots, n\}.$ Then there exists an $\epsilon > 0$, a subsequence $\{K_{i_m}\}_{m = 0}^\infty$, and a $\nu_j$-measurable $S \subseteq \Gr_{n,j}$ such that $\nu_j(S) > 0$ and
    $$d_\mathcal{H}(\pi_H(K_{i_m}), \pi_H(K)) \geq \epsilon \quad \text{for all } m \in \mathbb{N}, \quad H \in S.$$

\begin{figure}[H]
\centering
\begin{tikzpicture}[scale=1,line cap=round,line join=round]
\tikzset{
  ambient/.style={black,line width=0.9pt},
  cone/.style={fill=none},
  coneedge/.style={black,line width=0.65pt},
  outline/.style={black,line width=0.95pt},
  Kfill/.style={fill=gray!18,opacity=.5},
  Kifill/.style={fill=gray!18,opacity=.5},
  projK/.style={black,line width=2.0pt},
  projKi/.style={black,line width=2.0pt},
  guide/.style={black,dashed,line width=0.8pt},
  measure/.style={black,line width=0.9pt,decorate,
    decoration={brace,amplitude=4.5pt}},
  lab/.style={black,font=\small},
}

\def\R{5.2}
\def\ang{14}
\def\halfcone{7}
\pgfmathsetmacro{\ux}{cos(\ang)}
\pgfmathsetmacro{\uy}{sin(\ang)}
\pgfmathsetmacro{\nx}{-sin(\ang)}
\pgfmathsetmacro{\ny}{ cos(\ang)}

\fill[cone] (0,0) --
  ({\R*cos(\ang-\halfcone)},{\R*sin(\ang-\halfcone)})
  arc[start angle=\ang-\halfcone,end angle=\ang+\halfcone,radius=\R] -- cycle;
\fill[cone] (0,0) --
  ({\R*cos(180+\ang-\halfcone)},{\R*sin(180+\ang-\halfcone)})
  arc[start angle=180+\ang-\halfcone,end angle=180+\ang+\halfcone,radius=\R] -- cycle;

\draw[coneedge] (0,0) -- ({\R*cos(\ang-\halfcone)},{\R*sin(\ang-\halfcone)});
\draw[coneedge] (0,0) -- ({\R*cos(\ang+\halfcone)},{\R*sin(\ang+\halfcone)});
\draw[coneedge] (0,0) -- ({\R*cos(180+\ang-\halfcone)},{\R*sin(180+\ang-\halfcone)});
\draw[coneedge] (0,0) -- ({\R*cos(180+\ang+\halfcone)},{\R*sin(180+\ang+\halfcone)});

\draw[ambient] ({-\R*\ux},{-\R*\uy}) -- ({\R*\ux},{\R*\uy});
\node[lab,above right=1pt] at ({4.15*\ux},{4.15*\uy}) {$H$};

\coordinate (PA1) at (-3.31,-0.824);
\coordinate (PA2) at (-1.47,-0.366);
\coordinate (PB1) at (1.22,0.304);
\coordinate (PB2) at (3.16,0.787);
\coordinate (PAm) at (-2.39,-0.595);
\coordinate (PBm) at (2.19,0.546);

\draw[projK]  (PA1)--(PA2);
\draw[projKi] (PB1)--(PB2);

\coordinate (LBmid) at ($0.5*(PA2)+0.5*(PB1)+(0.24*\nx,0.24*\ny)$);
\coordinate (LB1) at ($(LBmid)+(-0.87*\ux,-0.87*\uy)$);
\coordinate (LB2) at ($(LBmid)+( 0.87*\ux, 0.87*\uy)$);
\draw[measure] (LB1) -- (LB2);
\node[lab] at ($(LBmid)+(-0.08,0.40)$) {$>\nicefrac{\epsilon}{2}$};

\def\Kcoords{
(-2.068,1.491) (-2.216,1.682) (-2.369,1.808) (-2.546,1.894)
(-2.822,1.958) (-3.204,1.961) (-3.572,1.883) (-3.798,1.781)
(-3.873,1.725) (-3.902,1.689) (-3.954,1.578) (-4.006,1.358)
(-4.009,1.111) (-3.971,0.929) (-3.911,0.800) (-3.792,0.651)
(-3.568,0.470) (-3.271,0.323) (-3.012,0.259) (-2.846,0.255)
(-2.707,0.286) (-2.502,0.385) (-2.247,0.585) (-2.044,0.831)
(-1.953,1.008) (-1.934,1.081) (-1.935,1.138) (-1.967,1.279)
}
\fill[Kfill] plot[smooth cycle] coordinates {\Kcoords};
\draw[outline] plot[smooth cycle] coordinates {\Kcoords};

\def\Lcoords{
(1.651,3.533) (1.517,3.555) (1.370,3.544) (1.142,3.474)
(0.877,3.323) (0.669,3.136) (0.550,2.967) (0.485,2.804)
(0.451,2.590) (0.469,2.329) (0.532,2.113) (0.589,2.008)
(0.618,1.975) (0.696,1.924) (0.922,1.836) (1.289,1.779)
(1.670,1.803) (1.950,1.885) (2.131,1.986) (2.277,2.120)
(2.407,2.307) (2.488,2.505) (2.511,2.640) (2.507,2.710)
(2.476,2.808) (2.361,3.006) (2.133,3.255) (1.861,3.446)
}
\fill[Kifill] plot[smooth cycle] coordinates {\Lcoords};
\draw[outline] plot[smooth cycle] coordinates {\Lcoords};

\node[lab] at (-3.02,1.2) {$K$};
\node[lab] at (1.46,2.75) {$K_{i_{m'}}$};

\draw[guide] (-2.95,0.95) -- (PAm);
\draw[guide] (1.45,2.5) -- (PBm);

\coordinate (UB1) at (-1.76,1.80);
\coordinate (UB2) at ($(UB1)+(2.10*\ux,2.10*\uy)$);
\draw[measure] (UB1) -- (UB2);
\node[lab] at ($(UB1)!0.5!(UB2)+(0.00,0.3)$) {$>\epsilon$};

\end{tikzpicture}
\caption{We can assuredly find a subspace and a subspace $H$ and a subsequence $\{K_{i_m}\}_{m = 0}^\infty$ where the projections of the latter onto the former remain persistently ``far away'' from $\pi_H(K)$. Intuitively, this property holds for subspaces ``nearby'' $H$.}
\end{figure}

\begin{proof}
    Take $\{K_i\}_{i = 0}^\infty \subseteq \mathcal{K}^n$ and $K \in \mathcal{K}^n$ such that $K_i$ does \emph{not} converge to $K$ in the Hausdorff distance. Of course, we might then select a subsequence $\{K_{i_m}\}_{m = 0}^\infty$ such that, for some $\epsilon > 0$,
        $$d_\mathcal{H}(K_{i_m}, K) \geq \epsilon \quad \text{for all } m \in \mathbb{N}.$$
        Set then $R := \max_{x \in K} ||x||$ --- we recall that the support function $h_K$ is $R$-Lipschitz on $S^{n-1}$. Indeed, for any $u, v \in S^{n-1}$, we might pick some $x_u \in K$ such that $h_K(u) = \langle u, x_u \rangle $ via compactness. It follows
        \begin{align*}
            h_K(u) - h_K(v) &= \langle u, x_u \rangle - h_K(v) \\ 
            &\leq \langle u, x_u\rangle - \langle v, x_u \rangle \\
&= \langle u - v, x_u \rangle \\
            &\leq R||u - v||
        \end{align*}
        by Cauchy--Schwarz, and so $h_K$ is indeed $R$-Lipschitz. Write then $d_m := d_\mathcal{H}(K_{i_m}, K) \geq \epsilon$ --- clearly $K_{i_m} \subseteq K + d_mB^n$ by definition of the Hausdorff distance, and so any element of $K_{i_m}$ might be expressed as $x + y$ for $x \in K$ and $y \in d_mB^n$, where 
        $$||x + y|| \leq ||x|| + ||y|| \leq R + d_m.$$
        It follows each $h_{K_{i_m}}$ is $(R + d_m)$-Lipschitz on $S^{n-1}$, as is seen by our same argument for $K$. Take then a sequence $\{u_m\}_{m = 0}^\infty \subseteq S^{n-1}$ such that $u_m$ achieves $d_m$, that is select $u_m$ such that
        $$d_\mathcal{H}(K_{i_{m}}, K) = \max_{u \in S^{n-1}} |h_{K_{i_m}}(u) - h_K(u)| = |h_{K_{i_m}}(u_m) - h_K(u_m)| \quad \text{for all } m \in \mathbb{N};$$
        this is, of course, again possible by compactness. We simultaneously define $H_m : S^{n-1} \rightarrow \mathbb{R}$ by $H_m(u) := |h_{K_{i_m}}(u) - h_K(u)|$; it is easy to see that $H_m$ is itself $(2R + d_m)$-Lipschitz. Indeed, for any $u, v \in S^{n-1}$ we have
        \begin{align*}
            |H_m(u) - H_m(v)| &= \left ||h_{K_{i_m}}(u) - h_K(u) | - |h_{K_{i_m}}(v) - h_K(v)| \right | \\
            &\leq \left | (h_{K_{i_m}}(u) - h_K(u)) - (h_{K_{i_m}}(v) - h_K(v)) \right | \\
            & = \left |(h_{K_{i_m}}(u) - h_{K_{i_m}}(v)) - (h_K(u) - h_K(v)) \right| \\
            &\leq |h_{K_{i_m}}(u) - h_{K_{i_m}}(v)| + |h_K(u) - h_K(v)| \\ 
            &\leq (2R + d_m)||u - v||
        \end{align*}
        by applying the triangle inequality and our previous work. In particular, we get $|H_m(u_m) - H_m(v)| \leq (2R + d_m)||u_m - v||$ for every $v \in S^{n-1}$, which immediately has that
        \begin{align}\label{Hm-bound}H_m(u_m) - (2R + d_m)||u_m - v|| \leq H_m(v) \quad \text{for all } v \in S^{n-1}.\end{align}
        As $H_m(u_m) = d_m$ by definition, this is then exactly that $H_m(v) \geq d_m - (2R + d_m)||u_m - v||$ --- we set then $\delta := \frac{\epsilon}{2(2R + \epsilon)} > 0$. If $v \in S^{n-1}$ is such that $||u_m - v|| < \delta$, (\ref{Hm-bound}) refines to $H_m(v) \geq d_m - (2R + d_m)\delta$. Moreover, $\delta < \frac{d_m}{2(2R + d_m)}$ for each $m \in \mathbb{N}$ as, crossmultiplying, $$2\epsilon(2R + d_m) = 4\epsilon R+ 2\epsilon d_m \leq  4Rd_m + 2\epsilon d_m =2d_m(2R + \epsilon),$$

        and hence $(2R + d_m)\delta < \frac{d_m}{2}$ for each $m \in \mathbb{N}$. Combining this with our previous work has for all $m \in \mathbb{N}$ that
        \begin{align}\label{Hm(v)-bound}H_m(v) \geq d_m - \frac{d_m}{2} = \frac{d_m}{2} \geq \frac{\epsilon}{2} \quad \text{when} \quad v \in S^{n-1}, \quad  ||u_m - v || < \delta.\end{align}
        Owing again to the compactness of $S^{n-1}$ then, there is some subsequence $\{u_{m_k}\}_{k = 0}^\infty$ where $u_{m_k} \rightarrow u^\ast \in S^{n-1}$ as $k \rightarrow +\infty$ --- we drop finitely many terms of this subsequence such that $||u_{m_k} - u^\ast|| < \nicefrac{\delta}{2}$ for all $k \in \mathbb{N}$ and define the open spherical cap 
        $$U := \left \{v \in S^{n-1} : ||u^\ast - v|| < \frac{\delta}{2} \right \};$$
        of course if $v \in U$, we have for each $k \in \mathbb{N}$ that
        $||u_{m_k} -v|| \leq ||u_{m_k} - u^\ast|| + ||u^\ast - v|| < \frac{\delta}{2} + \frac{\delta}{2} = \delta,$
        and hence $H_{m_k}(v) \geq \nicefrac{\epsilon}{2}$ for each $v \in U$ and $k \in \mathbb{N}$, given (\ref{Hm(v)-bound}). With this in mind, we set
        $$S := \{H \in \Gr_{n,j} : H \cap U \neq \varnothing\};$$
        as $j \geq 1$, it is clear that $S$ is a nonempty open subset of $\Gr_{n,j}$, and hence $\nu_j$-measurable with $\nu_j(S) > 0$ --- to quickly verify this, we might metrize the standard topology on $\Gr_{n,j}$ by means of the operator norm metric 
        $$d(H,I) := ||\pi_H - \pi_I||_{\textup{op}} = \sup_{u \in S^{n-1}} ||\pi_H(u) - \pi_I(u)|| \quad \text{for all } H,I \in \Gr_{n,j}.$$
        We can observe then that $S = \left \{H \in \Gr_{n,j} : ||\pi_H(u^\ast)||  > 1 - \nicefrac{\delta^2}{8} \right \}$, as for $v \in S^{n-1}$ we have $||u^\ast - v||^2 = 2 - 2\langle u^\ast, v\rangle$; it follows
        $$||u^\ast - v|| < \frac{\delta}{2} \quad \Longleftrightarrow \quad 2 - 2\langle u^\ast, v \rangle < \frac{\delta^2}{4} \quad \Longleftrightarrow \quad  1 - \frac{\delta^2}{8} < \langle u^\ast, v \rangle,$$
        and so $U = \left \{v \in S^{n-1} :  \langle u^\ast, v \rangle > 1 - \nicefrac{\delta^2}{8} \right \}$. If we had $H \in S$ then, we can take $v \in H \cap U$ and get
        \begin{align*}1 - \frac{\delta^2}{8} &< \langle u^\ast, v \rangle \\ &= \langle \pi_H(u^\ast), v \rangle \\ &\leq ||\pi_H(u^\ast)|| \cdot ||v|| \\ &= ||\pi_H(u^\ast)||,\end{align*}
        by application of Lemma \ref{basic-proj-lemmas}s part (i); conversely, if $H \in \Gr_{n,j}$ satisfies $||\pi_H(u^\ast)|| > 1 - \nicefrac{\delta^2}{8}$, we might take $v := \frac{\pi_H(u^\ast)}{||\pi_H(u^\ast)||} \in S^{n-1} \cap H$ --- assuming without loss of generality that $\delta > 0$ is small enough that $||\pi_H(u^\ast)|| > 0$. Then 
        \begin{align*}\langle u^\ast, v \rangle  &= \frac{\langle u^\ast, \pi_H(u^\ast) \rangle}{||\pi_H(u^\ast)||} \\ &= \frac{||\pi_H(u^\ast)||^2}{||\pi_H(u^\ast)||} \\ &= ||\pi_H(u^\ast)|| \\ &> 1 - \frac{\delta^2}{8}\end{align*}
        via application of Lemma \ref{basic-proj-lemmas} again,
        and so $v \in H \cap U$ and $H \cap U \neq \varnothing$. The openness of $S$ then follows, as we may then render it as the preimage of $\left (1 - \nicefrac{\delta^2}{8}, \infty \right )$ under the continuous mapping $F : \Gr_{n,j} \rightarrow \mathbb{R}$ given by $H \mapsto ||\pi_H(u^\ast)||$ --- indeed $F$ is $1$-Lipschitz as
        \begin{align*}|F(H) - F(I)| &= |||\pi_H(u^\ast)|| - ||\pi_I(u^\ast)||| \\ &\leq ||\pi_H(u^\ast) - \pi_I(u^\ast)|| \\ &\leq ||\pi_H - \pi_I||_{\textup{op}} \\
        &= d(H, I).\end{align*}
        Nonemptiness of $S$ when $j \geq 1$ is even more straightforward: we may select any $v \in U$ and choose $j - 1$ vectors $\{v_i\}_{i = 0}^{j - 2}$ where the collection $\{v\} \cup \{v_i\}_{i = 0}^{j-2}$ is linearly independent. Their span then exactly produces an $H \in \Gr_{n,j}$ with $H \in S$.
        
        To finish the proof of the lemma then, we take $H \in S$ and $v \in H \cap U \subseteq S^{n-1} \cap H$. By Lemma \ref{basic-proj-lemmas} part (iii) we know that 
        $$d_\mathcal{H}(\pi_H(K_{i_{m_k}}), \pi_H(K)) = \sup_{u \in S^{n-1} \cap H} \left|h_{\pi_H(K_{i_{m_k}})}(u) - h_{\pi_H(K)} (u)\right|,$$
        where by Lemma \ref{basic-proj-lemmas} part (ii) this reduces to 
        \begin{align*}d_\mathcal{H}(\pi_H(K_{i_{m_k}}), \pi_H(K)) &= \sup_{u \in S^{n-1} \cap H} \left|h_{K_{i_{m_k}}}(u) - h_{K} (u)\right| \\ &= \sup_{u \in S^{n-1} \cap H} H_{m_k}(u);\end{align*}
        we have, though, that $v \in S^{n-1} \cap H$, so the supremum above must be greater than or equal to $H_{m_k}(v) \geq \nicefrac{\epsilon}{2}$. The desired claim thus holds with the subsequence $\{K_{i_{m_k}}\}_{k = 0}^\infty$ and the choice of ``$\epsilon$'' as $\nicefrac{\epsilon}{2}$.
\end{proof}

\end{lemma}

We are in a position now to prove Theorem \ref{delta-j-haus-equiv}.

\begin{proof}[Proof of Theorem \ref{delta-j-haus-equiv}]
    Let $j \in \{1, \dots, n\}$ and take $\{K_i\}_{i = 0}^\infty \subseteq \mathcal{K}^n_j$. Say $K \in \mathcal{K}^n_j$ is such that 
    $\delta_j(K_i, K) \rightarrow 0$ as $i \rightarrow + \infty$, but assume for the sake of contradiction that $K_i$ does \emph{not} converge to $K$ in the Hausdorff distance.

    Following Lemma \ref{far-away-projs}, there is then some $\epsilon > 0$, $\nu_j$-measurable set $S \subseteq \Gr_{n,j}$ with $\nu_j(S) > 0$, and a subsequence $\{K_{i_m}\}_{m = 0}^\infty$ such that
    \begin{align}\label{proj-contra}d_\mathcal{H}(\pi_H(K_{i_m}), \pi_H(K)) \geq \epsilon \quad \text{for all } m \in \mathbb{N}, \quad H \in S;\end{align}
     of course, in spite of this, one still has $\delta_j$-convergence of $K_{i_m}$ to $K$, i.e.
    \begin{align}\label{cnvgtozero} \delta_j(K_{i_m}, K) = {n \brack j} \int_{\Gr_{n,j}} \textup{vol}_j(\pi_H(K_{i_m}) \, \triangle \, \pi_H(K)) \, d\nu_j(H) \rightarrow 0 \quad \text{as} \quad m \rightarrow + \infty.\end{align}
    We write then $f_m : \Gr_{n,j} \rightarrow \mathbb{R}$, where $f_m(H) : = \textup{vol}_j(\pi_H(K_{i_m}) \, \triangle \, \pi_H(K))$. As ${n \brack j} > 0$, (\ref{cnvgtozero}) implies that the sequence $\{f_m\}_{m = 0}^\infty$ converges in $L^1(\Gr_{n,j}, \nu_j)$ to the function which is identically zero on $\Gr_{n,j}$. It follows then there is some subsequence $\{f_{m_k}\}_{k = 0}^\infty$ such that $f_{m_k}$ converges to zero pointwise $\nu_j$-a.e., i.e. 
    $$\textup{vol}_j(\pi_H(K_{i_{m_k}}) \, \triangle \, \pi_H(K)) \rightarrow 0 \quad \text{as} \quad k \rightarrow +\infty \quad \text{for all } H \notin Z,$$
    where $Z \subseteq \Gr_{n,j}$ is a $\nu_j$-null set. Define then 
    $$P_m := \{H \in \Gr_{n,j} : \dim \pi_H(K_{i_m}) < j\} \quad \text{and} \quad P := \{H \in \Gr_{n,j} : \dim \pi_H(K) < j\}.$$
    Then because $\dim(K) \geq j$ and $\dim(K_{i_m}) \geq j$ for each $m \in \mathbb{N}$, $P$ is $\nu_j$-null and $P_m$ is $\nu_j$-null for all $m \in \mathbb{N}$ --- see for instance \cite{mattila2018hausdorffdimensionprojectionsintersections} Theorem 3.1. Thus for $H \in \Gr_{n,j}$ in the complement of the $\nu_j$-null set $\left (\bigcup_{m =0}^\infty P_m \right) \cup P \cup Z$, i.e. for 
    $$H \in \left ( \left(\bigcup_{m =0}^\infty P_m \right) \cup P \cup Z \right )^c = \left (\bigcap_{m = 0}^\infty P_m^c \right )\cap P^c \cap Z^c,$$ we have $$\textup{vol}_j(\pi_H(K_{i_{m_k}}) \, \triangle \, \pi_H(K)) \rightarrow 0 \quad \text{as} \quad k \rightarrow +\infty \quad \Longrightarrow \quad d_\mathcal{H}(\pi_H(K_{i_{m_k}}), \pi_H(K)) \rightarrow 0 \quad \text{as} \quad k \rightarrow +\infty$$
    by \cite{https://doi.org/10.1112/S0025579300005179} Theorem 7, as here $\pi_H(K)$ and $\pi_H(K_{i_{m_k}})$ for all $k \in \mathbb{N}$ are full-dimensional \emph{relative} to the $j$-dimensional subspace $H$. 
    
    This of course yields a contradiction, however, as (\ref{proj-contra}) prevents such Hausdorff convergence of projections on a set of positive measure. The desired result thus follows.
\end{proof}

With this result in hand then, we have indeed resolved \cite{besau2024} Open Question 2.2 positively, and may now give our proof of Theorem \ref{convergence}.

\begin{proof}[Proof of Theorem \ref{convergence}]
    Let $\phi \in \textup{Val}$ be monotone. We will organize our proof by specifying a type of convergence and some successive arrows emitting from it in the diagram in the statement of Theorem \ref{convergence}.

    \begin{figure}[h]
        \centering
        % https://q.uiver.app/#q=WzAsNCxbMCwyLCJkX1xcbWF0aGNhbHtIfS1cXHRleHR7Y29udmVyZ2VuY2V9Il0sWzIsMCwiXFxEZWx0YV9cXHBoaS1cXHRleHR7Y29udmVyZ2VuY2V9Il0sWzIsMiwiRF9cXHBoaS1cXHRleHR7Y29udmVyZ2VuY2V9Il0sWzIsNCwiXFxyaG9fXFxwaGktXFx0ZXh0e2NvbnZlcmdlbmNlfSJdLFswLDEsIih2aWkpIiwwLHsiY3VydmUiOi01fV0sWzAsMiwiKGkpIl0sWzAsMywiKGkpIiwyLHsiY3VydmUiOjV9XV0=
\[\begin{tikzcd}
	&& {\Delta_\phi-\text{convergence}} \\
	\\
	{d_\mathcal{H}-\text{convergence}} && {D_\phi-\text{convergence}} \\
	\\
	&& {\rho_\phi-\text{convergence}}
	\arrow["{(vii)}", curve={height=-30pt}, from=3-1, to=1-3]
	\arrow["{(i)}", from=3-1, to=3-3]
	\arrow["{(i)}"', curve={height=30pt}, from=3-1, to=5-3]
\end{tikzcd}\]
        \caption{For example, we start with the claimed implications of Hausdorff convergence under various criteria.}
    \end{figure}

    We start with $d_\mathcal{H}$-convergence, from which two arrows are obvious; indeed, they are shown the proof of Proposition \ref{basic-props} (ii).  The remaining one, that is the one towards $\Delta_\phi$-convergence, is also immediate --- under hypothesis (vii) one has 
    \begin{align*}
    \phi \vert_{\mathcal{K}(H)} = \textup{Kl}_{\phi}(H) \textup{vol}_k \quad \text{for all } H \in \Gr_{n,k} \Longrightarrow 
    \Delta_{\phi} \vert_{\mathcal{K}(H)} = \textup{Kl}_\phi(H) \Delta_{\textup{vol}_k} \quad \text{for all } H \in \Gr_{n,k} ;
    \end{align*}
    if we take then $\{K_i\}_{i = 0}^\infty \subseteq \mathcal{K}_k(H)$ and $K \in \mathcal{K}_k(H)$, given the $K_i$ and $K$ are full-dimensional \emph{relative} to $H$, we have 
    \begin{align}\label{equiv-on-subspace}d_\mathcal{H}(K_i, K) \rightarrow 0 \quad \Longleftrightarrow \quad \Delta_{\textup{vol}_k}(K_i, K) \rightarrow 0\end{align}
    by \cite{https://doi.org/10.1112/S0025579300005179} Theorem 7. As $\textup{Kl}_\phi(H) > 0$ then, the desired implication follows. 

    We look at the implications coming from $\Delta_\phi$-convergence then --- that it has Hausdorff convergence under condition (vi) also follows from (\ref{equiv-on-subspace}), as in this case we have
    \begin{align}\label{delta-bound-below}\Delta_{\phi}(K_i,K) &= \sum_{j = 0}^n \Delta_{\phi_j}(K_i,K) \\ &\geq \Delta_{\phi_k}(K_i,K) \notag \\ &= \textup{Kl}_{\phi_k}(H)\Delta_{\textup{vol}_k}(K_i, K); \notag\end{align}
    this follows from monotonicity of the $\phi_j$. If we instead took condition (iii) to hold, it is clear the exact same argument applies, as then $\phi_n$ need be a positive multiple of volume and hence even. Finally, it is clear that we have $D_\phi$-convergence here under condition (v), as then $$D_\phi(K,L) \leq \Delta_\phi(K,L) \quad \text{for all } K,L \in \mathcal{K}^n$$ by Proposition \ref{monotone-props} (v).

    We move then to the arrows coming from $\rho_\phi$-convergence --- obviously we have $D_\phi$-convergence under (i) as $$D_\phi(K,L) \leq \rho_\phi(K,L) \quad \text{for all } K, L \in \mathcal{K}^n$$ by Proposition \ref{monotone-props} (iv). As to why we have $d_\mathcal{H}$-convergence under (iv), we need turn to Florian in \cite{florian} --- first, analogous to (\ref{delta-bound-below}), we have for some fixed $G_1, \dots, G_{n-1} \in \mathcal{K}^n$, $\{K_i\}_{i = 0}^\infty \subseteq \mathcal{K}^n$ and $K \in \mathcal{K}^n$ that
    \begin{align*}
        \rho_\phi(K_i, K) &= \sum_{j = 1}^n \rho_{\phi_j}(K_i, K) \\
        &\geq \rho_{\phi_1}(K_i, K)  \\
        &= \int_{S^{n-1}} \left (2\max\{h_{K_i}(u), h_K(u)\} - h_{K_i}(u) - h_K(u) \right ) \textup{D}(D^2 h_{G_1}(u), \dots D^2 h_{G_{n-1}}(u)) \, d\omega(u) \\
        &= \int_{S^{n-1}} |h_{K_i}(u) - h_K(u)| \textup{D}(D^2 h_{G_1}(u), \dots D^2 h_{G_{n-1}}(u)) \, d\omega(u);
    \end{align*}
    by hypothesis, the mixed discriminant $\textup{D}(D^2 h_{G_1}(u), \dots D^2 h_{G_{n-1}}(u))$ is continuous and strictly positive on $S^{n-1}$. By compactness then, there is some $m > 0$ where $m = \min_{u \in S^{n-1}} \textup{D}(D^2 h_{G_1}(u), \dots D^2 h_{G_{n-1}}(u))$. It follows 
    $$\rho_\phi(K_i, K) \geq m \cdot \int_{S^{n-1}} |h_{K_i}(u) - h_K(u)| \, d\omega(u),$$
    where $\int_{S^{n-1}} |h_{K_i}(u) - h_K(u)| \, d\omega(u)$ is proportional to the mean width metric $\rho_w$. As $\rho_w$-convergence has $d_\mathcal{H}$-convergence --- see \cite{florian} Theorem 2 --- we have then Hausdorff convergence here. 

    Lastly then, we turn to the implications emerging from $D_\phi$-convergence. $d_\mathcal{H}$-convergence here under conditions (iii) and (iv) is obvious, as by Proposition \ref{monotone-props} (ii) and (iii) we have 
    $$\rho_{\phi_1}(K,L) \leq D_\phi(K,L) \quad \text{and} \quad \Delta_{\phi_n}(K,L) \leq D_\phi(K,L) \quad \text{for all } K, L \in \mathcal{K}^n,$$
    from which the desired results follow given our previous work. As for Hausdorff convergence under condition (ii), this of course follows from our work on Theorem \ref{delta-j-haus-equiv} and the lower bound $$D_\phi(K,L) \geq m\delta_k(K,L) \quad \text{for all } K, L \in \mathcal{K}^n $$ that we developed in Theorem \ref{metricity-conditions}, where $m > 0$. 

    To finish then, we need just show $d_\mathcal{H}$-convergence under condition (viii). For this, recall the definition of $\partial_G$ in (\ref{eqn:partialG}) and the subsequent lower bound
    \begin{align*}D_{\phi}(K,L) &\geq \partial_G(K,L) \\ &= \textup{vol}(G + K \, \triangle \, G + L) \end{align*}
    for all $K, L \in \mathcal{K}^n$,
    where $\phi(K) = \mu_G(K) = \textup{vol}(G + K)$ for some fixed $G \in \mathcal{K}^n_n$.  From this it follows, if we had a sequence $\{K_i\}_{i = 0}^\infty \subseteq \mathcal{K}^n$ and $K \in \mathcal{K}^n$, that
    \begin{align*}
        &D_\phi(K_i,K) \rightarrow 0 \quad \text{as} \quad i \rightarrow + \infty \\
        \Longrightarrow\quad&\textup{vol}(G +K_i \, \triangle \, G + K) \rightarrow 0 \quad \text{as} \quad i \rightarrow + \infty \\
        \Longrightarrow\quad&d_\mathcal{H}(G + K_i, G + K) \rightarrow 0 \quad \text{as} \quad i \rightarrow +\infty;
    \end{align*}
    the last implication of course following here as $\{G + K_i\} \subseteq \mathcal{K}^n_n$ and $G + K \in \mathcal{K}^n_n$, ensuring that \cite{https://doi.org/10.1112/S0025579300005179} Theorem 7 again applies. Finally then, we remark
    $$d_\mathcal{H}(G + K_i, G + K) = d_\mathcal{H}(K_i, K) \quad \text{for all } i \in \mathbb{N},$$
    given that 
    \begin{align*}
        |h_{G + K_i}(u) - h_{G + K}(u)| &= \left |h_{G}(u) + h_{K_i}(u) - (h_{G}(u) + h_K(u)) \right | \\
        &= |h_{K_i}(u) - h_K(u)|
    \end{align*}
    for each $u \in S^{n-1}$ and that $d_\mathcal{H}(K,L) = \sup_{u \in S^{n-1}} |h_K(u) - h_L(u)|$ for all $K, L  \in \mathcal{K}^n$. It follows $K_i$ converges to $K$ in the Hausdorff distance, and the proof is complete.
    \end{proof}

Just as was the case with Theorem \ref{metricity-conditions} then, we get an easy corollary --- namely that the metrics produced by nonconstant, monotone, $\textup{SO}(n)$-invariant $\phi \in \textup{Val}$ (recall Corollary \ref{intrinsic-volumes-give-metrics}) happen to give the same topology as the Hausdorff distance. 

\begin{coroG}
    Let $\phi \in \textup{Val}$ be monotone, nonconstant, and $\textup{SO}(n)$-invariant. Then for the same $k \in \{1, \dots n\}$ in Corollary \ref{intrinsic-volumes-give-metrics} for which $D_\phi$ is a metric on $\mathcal{K}^n_k$, $D_\phi$ induces the topology of the Hausdorff distance on $\mathcal{K}^n_k$, i.e.
    $$D_\phi(K_i, K) \rightarrow 0 \quad \text{as} \quad i \rightarrow +\infty \quad \Longleftrightarrow \quad d_\mathcal{H}(K_i, K) \rightarrow 0 \quad \text{as} \quad i \rightarrow +\infty,$$
    where $K \in \mathcal{K}^n_k$ and $\{K_i\} \subseteq \mathcal{K}^n_k$. In particular, each $D_j$ induces the topology of the Hausdorff distance on $\mathcal{K}^n_{f(j)}$, where $f(j) = j - \delta_{j,1}$.
\end{coroG}

This result simultaneously generalizes the facts that $D_n=\Delta_{\textup{vol}}$ and $D_1 \varpropto \rho_w$ are both topologically equivalent to the Hausdorff distance, while also unifying the proofs of these equivalences --- indeed, the leveraged lower bound
$$D_j(K,L) \geq \delta_j(K,L) =  {n \brack j} \int_{\Gr_{n,j}} \textup{vol}_j(\pi_H(K) \,\triangle \, \pi_H(L)) \quad \text{for all } K, L \in \mathcal{K}^n, \quad j \in \{1, \dots, n\}$$
connects each $D_j$ to the symmetric difference metric on $j$-dimensional slices, motivically identical to how each $\textup{V}_j$ is understood as volume on $j$-dimensional slices.

Of course, a result similar to Corollary \ref{intrinsic-volumes-give-haus-top} obviously holds for the class of valuations $\{\mu_G : G \in \mathcal{K}^n_n\}$, as we document here.

\begin{coro}
    Let $G \in \mathcal{K}^n_n$. Then $D_{\mu_G}$ (and $D_{\textup{M}_G}$) induces the same topology as the Hausdorff distance on $\mathcal{K}^n$ --- that is, the $$\{D_{\mu_G} : G \in \mathcal{K}^n_n\} \quad (\text{resp. } \{D_{\textup{M}_G} : G \in \mathcal{K}^n_n\})  $$ give a parametrized family of complete, geodesic metrics on $\mathcal{K}^n$ that are all topologically equivalent to $d_\mathcal{H}$. 
\end{coro}

Such a corollary merits revisiting an aside made in Example \ref{mu-G-ex} --- that is, the \emph{complex}-geometric intrigue of the $D_{\mu_G}$/$D_{\textup{M}_G}$ construction. We develop this point now in Section \ref{7}.

\section{Application to the metric geometry of singularity types}\label{7}

\subsection{Toric singularity types and convex bodies}

Let $(X, \omega)$ be a compact K\"ahler manifold --- here we will assume some familiarity with compact K\"ahler manifolds and their pluripotential theory; see \cite{guedj2017degenerate} for some relevant exposition.

\begin{defn}
    We say a function $u : X \rightarrow \mathbb{R} \cup \{-\infty\}$ is quasi-plurisubharmonic (qpsh) if locally one may write $u = \rho + \varphi$, where $\rho$ is smooth and $\varphi$ is plurisubharmonic. It is $\omega$-plurisubharmonic ($\omega$-psh) if it is quasi-plurisubharmonic and $$\omega_u := \omega + i \partial \bar \partial u \geq 0$$ in the weak sense of currents. The space of the latter type of functions that are not identically $-\infty$ is denoted $\textup{PSH}(X, \omega)$.
\end{defn}

It is natural to introduce an equivalence relation on $\textup{PSH}(X, \omega)$ --- indeed, given $u, v \in \textup{PSH}(X, \omega)$, we say that $u$ is \emph{more singular} than $v$ if there is $C \in \mathbb{R}$ such that $u \leq v + C$; we denote this $u \preceq v$. If 
$$u \preceq v \quad \text{and} \quad v \preceq u,$$
we say $u$ and $v$ have the same \emph{singularity type}; this is a bona fide equivalence relation ``$\sim$''. In practice, we will only be interested in singularity types of \emph{model} potentials $\mathcal{M}(X, \omega)$; such are the potentials $\phi \in \textup{PSH}(X, \omega)$ that satisfy
\begin{align}\label{c-model}\phi &= C(\phi), \quad C(\phi) := \textup{usc} \left (\sup \mathcal{F}_\phi \right ), \quad \text{where} \\
&\mathcal{F}_\phi := \left \{u \in \textup{PSH}(X, \omega) : \phi \preceq u, u \leq 0, \int_X \omega^j_u \wedge \omega^{n-j}= \int_X \omega_\phi^j \wedge \omega^{n-j} \text{ for all } j \in \{0, \dots, n\} \right \} \notag;\end{align}
here $\textup{usc}$ is the upper semicontinuous regularization, i.e. $f^*(x) = \limsup_{y \rightarrow x} f(y)$ for all $x \in X$. We will write then $\mathcal{S}(X, \omega)$ to denote the space of \emph{all} singularity types of model potentials in $\textup{PSH}(X, \omega)$, that is 
\begin{align*}\mathcal{S}(X, \omega) &:= \pi(\mathcal{M}(X, \omega)), \quad \text{where} \quad \pi: \textup{PSH}(X, \omega) \rightarrow \nicefrac{\textup{PSH}(X,\omega)}{\sim}, \quad u \mapsto [u].
\end{align*}

The space $\mathcal{S}(X,\omega)$ is given a metric $d_{\mathcal{S}}$ in \cite{darvas2021metric} by Darvas, Di Nezza, and Lu --- its construction is somewhat complicated, but it may be studied via an expression involving non-pluripolar products of potentials, i.e.
\begin{align}\label{dS-quasi}d_\mathcal{S}([u], [v]) &\leq \sum_{j = 0}^n \left (2\int_X \omega ^j \wedge \omega ^{n-j}_{\max\{u,v\}} - \int_X \omega^j \wedge \omega^{n-j}_u - \int_X \omega^j \wedge \omega^{n-j}_v \right ) \\
&\leq Cd_\mathcal{S}([u], [v]), \quad \text{for all } [u],[v] \in \mathcal{S}(X,\omega); \notag\end{align}
here $C >0$ is a dimension-dependent constant. The central term in (\ref{dS-quasi}) of course beckons to a $\rho_\phi$-type expression --- this is not incidental, as we now explain by recalling some results of Darvas, Di Nezza and Lu. 

\begin{theorem}[\cite{darvas2020logconcavityvolumecomplexmongeampere} Theorem 7.2]
    Say $\phi \in \textup{PSH}(\mathbb{CP}^n, \omega_{FS})$ is toric if it is invariant under the $(S^1)^n$ action induced by the $(\mathbb{C}^\ast)^n$ action on $\mathbb{CP}^n$. Here, $\omega_{FS}$ is the Fubini--Study metric. Then there is a one-to-one correspondence
    \begin{align}\label{toric-correspondence}
        \mathcal{S}_{\textup{tor}}(\mathbb{CP}^n, \omega_{FS}) := \{[\phi] \in \mathcal{S}(\mathbb{CP}^n, \omega_{FS}) : [\phi] = [\psi], \quad \psi \text{ is toric}\} \quad  \longleftrightarrow \quad \mathcal{K}(\Delta),
    \end{align}
    where $\Delta$ is the unit simplex in $\mathbb{R}^n$. 
\end{theorem}

Formally, (\ref{toric-correspondence}) is given by a bijection $\Phi : \mathcal{K}(\Delta) \rightarrow \mathcal{S}_\textup{tor}(\mathbb{CP}^n, \omega_{FS})$ defined by $K \mapsto [\phi_K]$, where 
    \begin{align*}
        \phi_K(z) := H_K(z) - \frac{1}{2} \log(1 + |z|^2), \quad
        \text{and} \quad   H_K(z):= h_K \circ L; 
    \end{align*}
    $L : (\mathbb{C}^\ast)^n \rightarrow \mathbb{R}^n$ is itself defined by $z \mapsto (\log|z_1|, \dots, \log|z_n|)$. $\phi_K$ then extends from $(\mathbb{C}^\ast)^n$ to all of $\mathbb{CP}^n$ as an $\omega_{FS}$-psh function --- it is furthermore toric, as can be gleaned from its construction, and of model type. 

This correspondence is itself not merely ``set level'', and offers some useful features --- namely, the ability to interpret non-pluripolar volume through mixed volumes. 
    \begin{prop}[\cite{darvas2020logconcavityvolumecomplexmongeampere}]
        We have 
        \begin{align}\label{vol-is-mixed-vol}
            \frac{2^n}{n!} \cdot \int_{\mathbb{CP}^n} \omega_{\phi_{K_1}} \wedge \cdots \wedge \omega_{\phi_{K_n}} = \textup{V}(K_1, \dots, K_n) \quad \text{for all } \{K_j\}_{j = 1}^n \subseteq \mathcal{K}(\Delta).
        \end{align}
    \end{prop}
    Accordingly, given the additional fact that $\max\{\phi_K, \phi_L\} = \phi_{\textup{conv}(K \cup L)}$, we finally recognize that central term in (\ref{dS-quasi}) as $\nicefrac{n!}{2^n} \cdot  \rho_{\textup{M}_\Delta}$, given that then
    \begin{align*}
        &\sum_{j = 0}^n \left (2\int_{\mathbb{CP}^n} \omega ^j \wedge \omega ^{n-j}_{\max\{\phi_K,\phi_L\}} - \int_{\mathbb{CP}^n} \omega^j \wedge \omega^{n-j}_{\phi_K} - \int_{\mathbb{CP}^n} \omega^j \wedge \omega^{n-j}_{\phi_L} \right )  \\&= 2 \sum_{j = 0}^n \left( \int_{\mathbb{CP}^n} \omega^j \wedge \omega^{n-j}_{\max\{\phi_K, \phi_L\}} \right ) - \sum_{j = 0}^n \left( \int_{\mathbb{CP}^n} \omega^j \wedge \omega^{n-j}_{\phi_K} \right ) - \sum_{j = 0}^n \left( \int_{\mathbb{CP}^n} \omega^j \wedge \omega^{n-j}_{\phi_L} \right ) \\
        &= \frac{n!}{2^n} \big [2\textup{M}_\Delta(\textup{conv}(K \cup L)) - M_\Delta(K) - M_\Delta(L)\big] \\ &= \frac{n!}{2^n} \cdot \rho_{M_\Delta}(K,L)
    \end{align*}
for each $K, L \in \mathcal{K}(\Delta)$ --- this makes explicit the complex-geometric allusion noted in Example \ref{mu-G-ex}.

\subsection{A geodesic metric bilipschitz to $d_\mathcal{S}$}

    Our ability then to pushforward an expression ``equivalent'' to $d_\mathcal{S}$ naturally raises interest in comparisons with constructions native to the \emph{convex} side, e.g. the Hausdorff distance --- this is approached in \cite{us}.
    \begin{theorem}[\cite{us} Theorems 1.1 and 1.3]
        Let $G \in \mathcal{K}^n_n$. Then there is $G$-dependent $D > 1$ such that
        $$\frac{1}{D} d_\mathcal{H}(K,L)^n \leq \rho_{\textup{M}_G}(K,L) \leq Dd_\mathcal{H}(K,L) \quad \text{for all } K, L \in \mathcal{K}(G);$$
        the exponent $n$ is optimal for convex polytopes and might be improved to $\nicefrac{n + 1}{2}$ when $G$ has $C^2$ boundary --- the latter is the best possible exponent, i.e. the smallest that holds for any $G$.
\begin{figure}[H]
\centering
\begin{tikzpicture}[scale=1,line cap=round,line join=round]
\tikzset{
  outline/.style={black!100,line width=0.95pt},
  segline/.style={black!100,line width=1.1pt},
  Deltafill/.style={fill=gray!40,opacity=.65},
  Kfill/.style={fill=gray!18,opacity=.5},
  sumfill/.style={fill=gray!18,opacity=.5},
  lab/.style={black!100,font=\footnotesize}
}

\def\k{0.78}

\begin{scope}[xshift=-2cm]
  \fill[Deltafill] (0,0) -- (2,0) -- (0,2) -- cycle;
  \draw[outline] (0,0) -- (2,0) -- (0,2) -- cycle;
  \node[lab] at (1,-0.35) {$\Delta$};
\end{scope}

\begin{scope}[xshift=0.95cm]
  \fill[Kfill] (0,0) -- (\k,0) -- (0,\k) -- cycle;
  \draw[outline] (0,0) -- (\k,0) -- (0,\k) -- cycle;
  \draw[segline] (0,0) -- (0,\k);
  \node[lab,above] at ({\k/1.5},{\k/2.25}) {$K$};
  \draw[decorate,decoration={brace,amplitude=4pt}]
    (-0.11,0) -- (-0.11,\k);
  \node[lab,anchor=east] at (-0.18,0.39) {$L$};
  \draw[decorate,decoration={brace,amplitude=4pt,mirror}]
    (0,-0.11) -- (\k,-0.11);
  \node[lab] at ({\k/2},-0.45) {$d_{\mathcal H}(K,L)$};
\end{scope}

\begin{scope}[xshift=-2cm,yshift=-3.15cm]
  \fill[sumfill] (0,0) -- ({2+\k},0) -- (0,{2+\k}) -- cycle;
  \fill[Deltafill] (0,0) -- (2,0) -- (0,2) -- cycle;
  \fill[Kfill] (2,0) -- ({2+\k},0) -- (2,\k) -- cycle;
  \draw[outline] (0,0) -- ({2+\k},0) -- (0,{2+\k}) -- cycle;
  \draw[outline] (0,0) -- (2,0) -- (0,2) -- cycle;
  \draw[outline] (2,0) -- ({2+\k},0) -- (2,\k) -- cycle;
  \draw[segline] (2,0) -- (2,\k);
  \node[lab,anchor=west] at (-0.35,-0.5)
    {$\Delta+\conv(K\cup L)=\Delta+K$};
\end{scope}

\begin{scope}[xshift=2.35cm,yshift=-3.15cm]
  \fill[sumfill] (0,0) -- (2,0) -- (2,\k) -- (0,{2+\k}) -- cycle;
  \fill[Deltafill] (0,0) -- (2,0) -- (0,2) -- cycle;
  \draw[outline] (0,0) -- (2,0) -- (2,\k) -- (0,{2+\k}) -- cycle;
  \draw[outline] (0,0) -- (2,0) -- (0,2) -- cycle;
  \draw[segline] (2,0) -- (2,\k);
  \node[lab] at (1,-0.5) {$\Delta+L$};
\end{scope}
\end{tikzpicture}
\vspace{4pt}
\caption{Appealing to the Steiner formula, a Lipschitz bound below for $\rho_{\textup{M}_G}$ in terms of $d_\mathcal{H}$ holds if and only if another Lipschitz bound holds for $\rho_{\mu_G}$ --- the illustrated example shows it cannot.}
\label{fig:PKL-conv}
\end{figure}
    \end{theorem}
    Though this result determines the topology $d_\mathcal{S}$ induces on $\mathcal{S}_\textup{tor}(\mathbb{CP}^n, \omega_{FS})$, it rules out $d_\mathcal{S}$ being bilipschitz to the pullback of the Hausdorff distance --- a geodesic metric --- on that same domain. Moreover, recall
    \begin{align*}&d_\mathcal{S}([u], [v]) := d_1^c(\{r[u]_t\}_t, \{r[v]_t\}_t) \quad \text{for all } [u], [v] \in \mathcal{S}(\mathbb{CP}^n, \omega_{FS}),\\ &\text{where} \quad r: \mathcal{S}(\mathbb{CP}^n, \omega_{FS}) \rightarrow \mathcal{R}(\mathbb{CP}^n, \omega_{FS}), \quad u \mapsto \{r[u]_t\}_t\end{align*} sends a singularity type to its associated finite-energy $d_1$-geodesic ray emanating from the zero potential --- see \cite{darvas2021metric} Section 3 for notation etc. Given then that the metric
    $$d_1^c(\{u_t\}_t, \{v_t\}_t) := \lim_{t \rightarrow + \infty} \frac{d_1(u_t, v_t)}{t} \quad \text{for all } \{u_t\}_t, \{v_t\}_t \in \mathcal{R}(\mathbb{CP}^n, \omega_{FS})$$
    makes $\mathcal{R}(\mathbb{CP}^n, \omega_{FS})$ into a geodesic space --- see \cite{Darvas_2020} --- it is \emph{also} that $d_\mathcal{S}$ is, by construction, isometric to a subspace of a geodesic space; this all leaves us with a natural question. 
    \begin{openq}
        Is $d_\mathcal{S}$ itself bilipschitz to a geodesic metric on $\mathcal{S}(\mathbb{CP}^n, \omega_{FS})$? That is, is there a metric $d$ and $D > 0$ such that $(\mathcal{S}(\mathbb{CP}^n, \omega_{FS}), d)$ is geodesic and
        $$\frac{1}{D} d_\mathcal{S}([u], [v]) \leq d([u], [v]) \leq Dd_\mathcal{S}([u], [v]) \quad \text{for all } [u], [v] \in \mathcal{S}(\mathbb{CP}^n, \omega_{FS})?$$
    \end{openq}

    We resolve this partially now, by applying our study to show this \emph{is} indeed the case on the subspace $\mathcal{S}_\textup{tor}(\mathbb{CP}^n, \omega_{FS})$.

    \begin{theoremH}
        Define 
        \begin{align*}\mathcal{D}(K,L) := \inf \left \{
        \sum_{i = 1}^N |\textup{M}_\Delta(K_{i-1}) - \textup{M}_\Delta(K_i)| : \{K_i\}_{i = 0}^N \subseteq \mathcal{K}(\Delta) \text{ is a comparable chain between } K \text{ and } L \right \}\end{align*}
        for each $K, L \in \mathcal{K}(\Delta)$, i.e. $\mathcal{D}$ is $D_{\textup{M}_\Delta}$ with its defining chains restricted to $\mathcal{K}(\Delta)$. Then the pushforward metric $\mathcal{D}_\ast$ on $\mathcal{S}_{\textup{tor}}(\mathbb{CP}^n, \omega_{FS})$, that is $$\mathcal{D}_\ast := \frac{n!}{2^n(n + 1)} (\Phi)_\ast \mathcal{D}, \quad \Phi : \mathcal{K}(\Delta) \rightarrow \mathcal{S}_{\textup{tor}}(\mathbb{CP}^n, \omega_{FS}), \quad K \mapsto [\phi_K],$$
        is geodesic and bilipschitz to $d_\mathcal{S}$.
    \end{theoremH}

    To prove this, we merely require one additional fact concerning the $d_\mathcal{S}$ metric, which we note now. 

    \begin{prop}[\cite{darvas2021metric} Lemma 3.4]
        If $[u], [v] \in \mathcal{S}(\mathbb{CP}^n, \omega_{FS})$ are such that $[u] \leq [v]$, i.e. $u \preceq v$, then 
        \begin{align*}
            d_\mathcal{S}([u], [v]) = \frac{1}{n + 1} \sum_{j = 0}^n \left (\int_{\mathbb{CP}^n} \omega^j \wedge \omega_v^{n-j} - \int_{\mathbb{CP}^n} \omega^j \wedge \omega_u^{n-j}  \right ),
        \end{align*}
        where $\omega := \omega_{FS}$.
    \end{prop}
    This proposition is particularly useful to us as, for $K, L \in \mathcal{K}(\Delta)$, $K \subseteq L$ has $\phi_K \preceq \phi_L$. Hence
    \begin{align}\label{comp-form}
        d_\mathcal{S}([\phi_K], [\phi_L]) &= \frac{1}{n + 1} \sum_{j = 0}^n \left (\int_{\mathbb{CP}^n} \omega^j \wedge \omega_{\phi_L}^{n-j} - \int_{\mathbb{CP}^n} \omega^j \wedge \omega_{\phi_K}^{n-j}  \right ) \notag \\
        &= \frac{n!}{2^n(n+1)}\sum_{j = 0}^n [\textup{V}(\Delta[j], L[n-j]) - \textup{V}(\Delta[j], K[n-j])] \notag \\
        &= \frac{n!}{2^n(n + 1)} \cdot |\textup{M}_\Delta(K) - \textup{M}_\Delta(L)|
    \end{align}
    after reindexing appropriately. We can now proceed with the proof of the theorem. 
    \begin{proof}[Proof of Theorem \ref{geo-bilip-to-dS}]
        As $\Delta$ is convex, interpolants of the form 
        $$(1-t)K + tL \quad \text{where} \quad K,L \in \mathcal{K}(\Delta), \quad t \in [0,1],$$
        remain in $\mathcal{K}(\Delta)$, and thus the proof of Theorem \ref{intrinsic} applies to show that $\mathcal{D}$ is intrinsic --- note inflating elements of a given chain by some $\delta B^n$ for $\delta > 0$ is unnecessary here. 

        Moreover, it is that $\mathcal{D}$ turns $\mathcal{K}(\Delta)$ into a proper metric space; this may be viewed as a consequence of Theorem \ref{geodesicity}, or of the more elementary observation that 
        $$\mathbbm{1} : (\mathcal{K}(\Delta), d_\mathcal{H}) \rightarrow (\mathcal{K}(\Delta), \mathcal{D}), \quad K \mapsto K,$$
        is a continuous bijection from a compact space to a Hausdorff space by Proposition \ref{basic-props} (ii), and hence a homeomorphism --- thus $\mathcal{K}(\Delta)$ is compact in the $\mathcal{D}$-topology. 
        
        From this it then follows that 
        $$\mathcal{D}_\ast([u], [v]) = \frac{n!}{2^n(n + 1)} \mathcal{D}(\Phi^{-1}([u]), \Phi^{-1}([v])) \quad \text{for all } [u], [v] \in \mathcal{S}_{\textup{tor}}(\mathbb{CP}^n, \omega_{FS})$$
        is geodesic, given it is isometric to $(\mathcal{K}(\Delta), \mathcal{D})$, the latter now geodesic by Hopf--Rinow. All, thus, that remains is to show that $\mathcal{D}_\ast$ is bilipschitz to $d_\mathcal{S}$ --- for this, take $[u], [v] \in \mathcal{S}_\textup{tor}(\mathbb{CP}^n, \omega_{FS})$ and set 
        $$K := \Phi^{-1}([u]), \quad L := \Phi^{-1}([v]), \quad \text{i.e.} \quad [\phi_K] = [u] \quad \text{and} \quad [\phi_L] = [v].$$
        We consider then any comparable chain in $\mathcal{K}(\Delta)$ joining $K$ and $L$, say $\{K_i\}_{i = 1}^N$. Pulling the $\nicefrac{n!}{2^n(n+1)}$ factor into the infimum defining $\mathcal{D}$ and summing over this chain gets 
        \begin{align*}
             \frac{n!}{2^n(n+1)}\sum_{i = 1}^N |\textup{M}_\Delta(K_{i-1}) - \textup{M}_\Delta(K_i)| &= \sum_{i = 1}^N \sum_{j = 0}^n \frac{1}{n  + 1} \left | \int_{\mathbb{CP}^n} \omega^j \wedge \omega^{n-j}_{\phi_{K_{i-1}}} - \int_{\mathbb{CP}^n} \omega^j \wedge \omega^{n-j}_{\phi_{K_{i}}} \right | \\
             &= \sum_{i = 1}^N d_\mathcal{S}([\phi_{K_{i-1}}], [\phi_{K_i}]) \\
             &\geq d_\mathcal{S}([\phi_K], [\phi_L])
        \end{align*}
        by application of (\ref{comp-form}) and the triangle inequality for $d_\mathcal{S}$ --- taking an infimum thus has $d_\mathcal{S} \leq \mathcal{D}_\ast$. Conversely, via the comparable chain $\{K, \textup{conv}(K \cup L), L\}$, there is an upper bound
        \begin{align*}
  \mathcal{D}_\ast([\phi_K], [\phi_L]) &\leq \frac{n!}{2^n(n + 1)} \left [2\textup{M}_\Delta(\textup{conv}(K \cup L)) - \textup{M}_\Delta(K) - \textup{M}_\Delta(L) \right] \\
  &= \frac{1}{n+1} \sum_{j = 0}^n \left (2 \int_{\mathbb{CP}^n} \omega^j \wedge \omega^{n-j}_{\phi_{\textup{conv}(K \cup L)}} - \int_{\mathbb{CP}^n} \omega^j \wedge \omega^{n-j}_{\phi_{K}} - \int_{\mathbb{CP}^n} \omega^j \wedge \omega^{n-j}_{\phi_L} \right ) \\
  &\leq \frac{C}{n + 1} d_\mathcal{S}([\phi_K],[\phi_L]),
\end{align*}
where $C > 0$ is the dimension-dependent constant from (\ref{dS-quasi}). This completes the proof. 
    \end{proof}
    
\appendix

\section{Differing notions of being ``model''}

In claiming the correspondence (\ref{toric-correspondence}) a considerable abuse of precision occurs, which we rectify here. In particular, there is another notion of being ``model'', which is the notion used in \cite{darvas2020logconcavityvolumecomplexmongeampere} for (\ref{toric-correspondence}).

\begin{defn}
    Let $\phi \in \textup{PSH}(X, \omega)$. We say that $\phi$ is $P$-model if $\phi = P(\phi)$, where
    \begin{align}\label{p-model}P(\phi) := \textup{usc}(\sup\{u \in \textup{PSH}(X, \omega) : u \preceq \phi, u \leq 0\}).\end{align}
\end{defn}
One sees immediately that $P$ is \emph{invariant} in the sense that $P(\phi) = P(\psi)$ if $[\phi] = [\psi]$, i.e. $\phi,\psi \in \textup{PSH}(X,\omega)$ have the same singularity type --- we will nonetheless avoid writing $P[\phi]$. 

Our main result here then is an elementary proof that this notion of being model is the ``same'' in our setting. This may hold \emph{in general} due to recent work in \cite{pang2026capacitystabilitycomplexmongeampere}, though we retain our work here given its comparable simplicity. 

\begin{lemma}\label{C-mod-is-P-mod}
    Let $\phi \in \textup{PSH}(\mathbb{CP}^n, \omega_{FS})$ be toric. Then 
    $C(\phi) = P(\phi).$
    \begin{proof}
        Consider toric $\phi \in \textup{PSH}(\mathbb{CP}^n, \omega_{FS})$  --- we assume without loss of generality that $\phi \leq 0$. On the torus $(\mathbb{C}^\ast)^n$ then inside the affine chart $\{z_0 \neq 0\}$ we write 
         \begin{align*}z = [1 : z_1 : \cdots : z_n] \quad \text{and} \quad L(z) := (\log|z_1|, \dots, \log|z_n|).\end{align*}
         In this setting, the convex function $F_\phi : \mathbb{R}^n \rightarrow \mathbb{R}$ associated to $\phi$ is 
         \begin{align*}
            F_\phi(x) &:= \phi([1 : e^{x_1} : \dots : e^{x_n}]) + F_0(x) \quad  \text{for all } x\in \mathbb{R}^n, \quad \text{where} \\
            F_0(x) &:= \frac{1}{2} \log \left (1 + \sum_{j = 1}^n e^{2x_j} \right) \quad \text{for all } x \in \mathbb{R}^n;
        \end{align*}
        see for instance \cite{guedj2018toricpluripotentialtheory} Example 3.8. We take here $G_\phi : \mathbb{R}^n \rightarrow \mathbb{R} \cup \{+\infty\}$ to denote the Legendre transform of $F_\phi$, i.e. $G_\phi(y) := F_\phi^*(y) := \sup_{x \in \mathbb{R}^n} ( \langle x, y \rangle - F_\phi(x))$ for all $y \in \mathbb{R}^n$ --- following this, the \emph{Newton body} of $\phi$ is defined as
        $$\Delta(\omega_{FS}, \phi) := \overline{\textup{dom}(G_\phi)} \subseteq \Delta,$$
        where $\textup{dom}(G_\phi) := \{y \in \mathbb{R}^n : G_\phi(y) < +\infty\}$; see \cite{Xia2025} Definition 5.2.1. The Newton body of $\phi$ is relevant to us here given \cite{Xia2025} Proposition 5.2.6, which holds that 
        \begin{align*}&G_{P(\phi)}(x) = \begin{cases}
            G_0(x) & \text{if } x \in \Delta(\omega_{FS}, \phi) \\ 
            +\infty &\text{otherwise}
        \end{cases} \quad \text{for all } x \in \mathbb{R}^n, &\Longrightarrow \\
        &G_{P(\phi)} = G_0 + X_K, \quad \text{where } X_K : \mathbb{R}^n \rightarrow \mathbb{R} \cup \{+\infty\}, \quad X_K(x) := \begin{cases}
            0 & \text{if } x \in K \\
            +\infty & \text{otherwise}
        \end{cases}; \notag\end{align*}
        here $K := \Delta(\omega_{FS},\phi).$ We take then $0 < \epsilon < 1$ and set $\phi_\epsilon := P((1-\epsilon)\phi)$ --- it is then immediate that
        \begin{align*}
            F_{(1-\epsilon)\phi}(x) &= (1-\epsilon)\phi([1 : e^{x_1} : \cdots e^{x_n}]) + F_0(x) \\
            &= (1-\epsilon)(F_\phi(x) - F_0(x)) + F_0(x) 
            = (1-\epsilon)F_\phi(x) + \epsilon F_0(x)
        \end{align*}
        for each $x \in \mathbb{R}^n$, from which it follows that $\Delta(\omega_{FS}, (1-\epsilon)\phi) = (1-\epsilon)K + \epsilon \Delta := K_\epsilon$; indeed, this is a consequence of Equation (A.8) in \cite{Xia2025}, as we have
       \begin{align*}
           \Delta(\omega_{FS},(1-\epsilon)\phi) = \overline{\textup{dom}(G_{(1-\epsilon)\phi})} 
           &= \overline{\textup{dom}(((1-\epsilon)F_\phi + \epsilon F_0)^\ast)} \\
           &= \overline{\textup{dom}(((1-\epsilon)F_\phi)^\ast)} + \overline{\textup{dom}((\epsilon F_0)^\ast)} 
           =  (1-\epsilon)K + \epsilon\Delta
       \end{align*}
       using some elementary properties of the effective domain and of the closure; applying Xia's formula again to $\phi_\epsilon$ then gets $G_{\phi_\epsilon} = G_0 + X_{K_\epsilon}$. Now take any sequence $\{\epsilon_i\}_{i = 0}^\infty \subseteq (0,1)$ with $\epsilon_i \searrow 0$ as $i \rightarrow + \infty$ --- given $K \subseteq \Delta$, $$K \subseteq (1-\epsilon)K + \epsilon \Delta \quad \text{for each} \quad \epsilon \in (0,1),$$ and as $d_\mathcal{H}(K_{\epsilon_i}, K) \rightarrow 0$ as $i \rightarrow +\infty$, we have
       $$\bigcap_{i = 0}^\infty K_{\epsilon_i} = \bigcap_{i =0}^\infty ((1-\epsilon_i)K + \epsilon_i\Delta) = K.$$
\begin{figure}[H]
\centering
\begin{tikzpicture}[scale=1.05,line cap=round,line join=round]
\tikzset{
  outline/.style={black!100,line width=1pt},
  Kfill/.style={fill=gray!40,opacity=.65},
  epsfill/.style={fill=gray!18,opacity=.5}
}

\def\a{1.35}
\def\s{4.20}
\def\pan{3.15}

\newcommand{\Panel}[2]{
  \begin{scope}[xshift=#2*\pan cm]
    \pgfmathsetmacro{\om}{1-#1}

    \coordinate (KSW) at (0,0);
    \coordinate (KSE) at (\a,0);
    \coordinate (KNE) at (\a,\a);
    \coordinate (KNW) at (0,\a);

    \coordinate (DSW) at (0,0);
    \coordinate (DSE) at (\s,0);
    \coordinate (DNW) at (0,\s);

    \coordinate (P1) at ($ \om*(KSW) + #1*(DSW) $);
    \coordinate (P2) at ($ \om*(KSE) + #1*(DSE) $);
    \coordinate (P3) at ($ \om*(KNE) + #1*(DSE) $);
    \coordinate (P4) at ($ \om*(KNE) + #1*(DNW) $);
    \coordinate (P5) at ($ \om*(KNW) + #1*(DNW) $);

    \path[use as bounding box] (-0.18,-0.18) rectangle (3.00,3.00);

    \fill[epsfill] (P1) -- (P2) -- (P3) -- (P4) -- (P5) -- cycle;
    \fill[Kfill] (KSW) rectangle (KNE);

    \draw[outline] (P1) -- (P2) -- (P3) -- (P4) -- (P5) -- cycle;
    \draw[outline] (KSW) rectangle (KNE);
  \end{scope}
}

\Panel{0.42}{0}
\Panel{0.30}{1}
\Panel{0.21}{2}
\Panel{0.14}{3}
\Panel{0.08}{4}
\end{tikzpicture}
\vspace{6pt}
\caption{An example where $K$ (dark gray) is an axis-aligned square contained in $\Delta$: for $\epsilon_i\searrow 0$, the convex bodies $K_{\epsilon_i}=(1-\epsilon_i)K+\epsilon_i\Delta$ (light gray) decrease to $K$.}
\end{figure}
       Moreover, by \cite{darvas2021metric} Lemma 2.4, one has 
       $$\phi_{\epsilon_i} = P((1- \epsilon_i)\phi) \searrow C(\phi) \quad \text{as} \quad i \rightarrow +\infty,$$
       and hence $C(\phi)$ is toric (as a limit of toric potentials) with $F_{\phi_{\epsilon_i}} \searrow F_{C(\phi)}$ as $i \rightarrow +\infty$. This itself implies $G_{\phi_{\epsilon_i}} \nearrow G_{C(\phi)}$ as $i \rightarrow +\infty$ by say \cite{guedj2018toricpluripotentialtheory} Lemma 2.1 (iii), which gets then $G_{C(\phi)} = \sup_i(G_{\phi_{\epsilon_i}})$; in particular, 
       \begin{align*}G_{C(\phi)} = \sup_i\left(G_{\phi_{\epsilon_i}}\right) = &\sup_i\left(G_0 + X_{K_{\epsilon_i}}\right) \\ &= G_0 + X_{\bigcap_{i =0}^\infty K_{\epsilon_i}} =G_0 + X_K = G_{P(\phi)}.\end{align*}
       Finally then, taking Legendre transforms, we have $F_{C(\phi)} = F_{P(\phi)}$ --- thus for each $z \in (\mathbb{C}^\ast)^n$ we get
       \begin{align*}C(\phi)(z) = F_{C(\phi)}(L(z)) - F_0(L(z)) 
       = F_{P(\phi)}(L(z)) - F_0(L(z)) = P(\phi)(z);\end{align*}
       see for instance \cite{guedj2018toricpluripotentialtheory} Section 3.1. This is that $C(\phi)$ and $P(\phi)$ are qpsh and agree a.e. on $\mathbb{CP}^n$ though, and so they must agree everywhere by \cite{Klimek1991} Corollary 2.9.8 --- that is, $C(\phi)$ equals $P(\phi)$ outright. 
    \end{proof}
\end{lemma}

This may seem to outright resolve our problem, but there is one minor inconsistency to patch. Indeed, recall we said $\phi \in \textup{PSH}(X, \omega)$ was of \emph{model type} if
$$[\phi] \in \pi(\mathcal{M}(X,\omega)), \quad \pi : \textup{PSH}(X, \omega) \rightarrow \nicefrac{\textup{PSH}(X,\omega)}{\sim}, \quad u \mapsto [u];$$
notably, this differs from the notion of being of ``model type'' in the $P$-model context, which is generally defined by the requirement \begin{align}\label{P-mod-type}[\phi] = [P(\phi)];\end{align} this is ``weaker'' in the sense that such a type needn't \emph{a priori} have a $P$-model representative. Nonetheless, Lemma \ref{C-mod-is-P-mod} says this isn't problematic in our setting as, for toric $\phi \in \textup{PSH}(\mathbb{CP}^n, \omega_{FS})$, one has
\begin{align*}
    &[\phi] = [P(\phi)] &\Longrightarrow \\
    &[\phi] = [C(\phi)] \quad \text{where} \quad C(C(\phi)) = C(\phi),
\end{align*}
owing to the idempotence of the $C$ (ceiling) operator established in \cite{darvas2021metric} Proposition 2.6 (iii) --- hence any toric $\phi \in \textup{PSH}(\mathbb{CP}^n, \omega_{FS})$ satisfying (\ref{P-mod-type}) lies in $\pi(\mathcal{M}(\mathbb{CP}^n, \omega_{FS}))$. Conversely, if toric $\phi \in \pi(\mathcal{M}(\mathbb{CP}^n, \omega_{FS}))$,
\begin{align*}
&[\phi] = [\psi] \quad \text{for} \quad \psi \in \textup{PSH}(\mathbb{CP}^n, \omega_{FS}), \quad C(\psi) = \psi &\Longrightarrow \\
&[\phi] = [C(\psi)] = [C(\phi)] = [P(\phi)],
\end{align*}
given that the ceiling operator descends to $\nicefrac{\textup{PSH}(\mathbb{CP}^n, \omega_{FS})}{\sim}$ and that Lemma \ref{C-mod-is-P-mod} holds. Thus the differing notions of being ``model type'' exactly coincide for us.

\printbibliography

@book{hug2020lectures,
  title={Lectures on convex geometry},
  author={Hug, Daniel and Weil, Wolfgang},
  volume={286},
  year={2020},
  publisher={Springer}
}

@book{schneider2014,
    title = "Convex Bodies: The Brunn-Minkowski Theory",
    author = "Rolf Schneider",
    publisher = "Cambridge University Press",
    year = "2014"
}

@article{us,
  author = {Ayo Aitokhuehi and Benjamin Braiman and David Owen Horace Cutler and Tamás Darvas and Robert Deaton and Prakhar Gupta and Jude Horsley and Vasanth Pidaparthy and Jen Tang},
  title = {The Hausdorff distance and metrics on toric singularity types},
  journal = {Bulletin des Sciences Mathématiques},
  year = {2025},
  volume = {204}
}

@article{darvas2021metric,
  author = {Darvas, Tam{\'a}s and Di Nezza, Eleonora and Lu, Hoang-Chinh},
  title = {The metric geometry of singularity types},
  journal = {Journal f{\"u}r die reine und angewandte Mathematik (Crelles Journal)},
  year = {2021},
  volume = {2021},
  number = {771},
  pages = {137--170}
}

@article{groemer,
  author = {Groemer, H.},
  title = {On the symmetric difference metric for convex bodies.},
  journal = {Beiträge zur Algebra und Geometrie},
  year = {2000},
  volume = {41},
  number = {1},
  pages = {107--114},
  url = {http://eudml.org/doc/225705},
  language = {eng}
}

@article{florian,
  author = {A. Florian},
  title = {On a metric for the class of compact convex sets},
  journal = {Geometriae Dedicata},
  year = {1989},
  volume = {30},
  pages = {69--80},
  url = {https://api.semanticscholar.org/CorpusID:119651984}
}

@book{alesker2014,
  title={Integral geometry and valuations},
  author={Alesker, Semyon and Fu, Joseph},
  year={2014},
  publisher={Springer}
}

@article{besau2024,
  author = {Besau, Florian and Hoehner, Steven},
  title = {An intrinsic volume metric for the class of convex bodies in $\mathbb{R}^n$},
  journal = {Communications in Contemporary Mathematics},
  year = {2024},
  volume = {26},
  number = {03},
  pages = {2350006}
}

@article{firey1976functional,
  author = {William J. Firey},
  title = {A functional characterization of certain mixed volumes},
  journal = {Israel Journal of Mathematics},
  year = {1976},
  volume = {24},
  pages = {274--281},
  month = {12},
  doi = {10.1007/BF02834758}
}

@article{bernig2010hermitianintegralgeometry,
  author = {Andreas Bernig and Joseph H. G. Fu},
  title = {Hermitian integral geometry},
  journal = {Annals of Mathematics},
  year = {2010},
  url = {https://arxiv.org/abs/0801.0711},
  eprint = {0801.0711},
  eprinttype = {arXiv},
  eprintclass = {math.DG}
}

@article{article,
  author = {Chrz\c{a}szcz, Katarzyna and Jachymski, Jacek and Turoboś, Filip},
  title = {On characterizations and topology of regular semimetric spaces},
  journal = {Publicationes Mathematicae Debrecen},
  year = {2018},
  volume = {93},
  pages = {87--105},
  month = {07},
  doi = {10.5486/PMD.2018.8049}
}

@article{Besau_2019,
  author = {Besau, Florian and Hoehner, Steven and Kur, Gil},
  title = {Intrinsic and Dual Volume Deviations of Convex Bodies and Polytopes},
  journal = {International Mathematics Research Notices},
  year = {2019},
  volume = {2021},
  number = {22},
  pages = {17456--17513},
  month = {12},
  doi = {10.1093/imrn/rnz277},
  issn = {1687-0247}
}

@article{https://doi.org/10.1112/S0025579300014625,
  author = {Klain, Daniel A.},
  title = {A short proof of Hadwiger's characterization theorem},
  journal = {Mathematika},
  year = {1995},
  volume = {42},
  number = {2},
  pages = {329--339},
  doi = {10.1112/S0025579300014625},
  url = {https://londmathsoc.onlinelibrary.wiley.com/doi/abs/10.1112/S0025579300014625},
  eprint = {https://londmathsoc.onlinelibrary.wiley.com/doi/pdf/10.1112/S0025579300014625}
}

@article{https://doi.org/10.1112/S0025579300005179,
  author = {Shephard, G. C. and Webster, R. J.},
  title = {Metrics for sets of convex bodies},
  journal = {Mathematika},
  year = {1965},
  volume = {12},
  number = {1},
  pages = {73--88},
  doi = {10.1112/S0025579300005179},
  url = {https://londmathsoc.onlinelibrary.wiley.com/doi/abs/10.1112/S0025579300005179},
  eprint = {https://londmathsoc.onlinelibrary.wiley.com/doi/pdf/10.1112/S0025579300005179}
}

@misc{alesker2003hardlefschetztheoremvaluations,
      title={Hard Lefschetz theorem for valuations, complex integral geometry, and unitarily invariant valuations}, 
      author={Semyon Alesker},
      year={2003},
      eprint={math/0209263},
      archivePrefix={arXiv},
      primaryClass={math.MG},
      url={https://arxiv.org/abs/math/0209263}, 
}

@misc{weil2017grassmannmeasuresconvexbodies,
      title={Grassmann measures of convex bodies}, 
      author={Wolfgang Weil},
      year={2017},
      eprint={1706.07651},
      archivePrefix={arXiv},
      primaryClass={math.MG},
      url={https://arxiv.org/abs/1706.07651}, 
}

@incollection{Schneider2017ValuationsClassicalBasicFacts,
  author    = {Schneider, Rolf},
  title     = {Valuations on Convex Bodies: The Classical Basic Facts},
  booktitle = {Tensor Valuations and Their Applications in Stochastic Geometry and Imaging},
  editor    = {Jensen, Eva B. Vedel and Kiderlen, Markus},
  series    = {Lecture Notes in Mathematics},
  volume    = {2177},
  pages     = {1--25},
  publisher = {Springer},
  address   = {Cham},
  year      = {2017},
  doi       = {10.1007/978-3-319-51951-7_1}
}

@article{Shenfeld_2019,
  author = {Shenfeld, Yair and van Handel, Ramon},
  title = {Mixed volumes and the Bochner method},
  journal = {Proceedings of the American Mathematical Society},
  year = {2019},
  volume = {147},
  number = {12},
  pages = {5385--5402},
  month = {06},
  doi = {10.1090/proc/14651},
  issn = {0002-9939}
}

@misc{mattila2018hausdorffdimensionprojectionsintersections,
      title={Hausdorff dimension, projections, intersections, and Besicovitch sets}, 
      author={Pertti Mattila},
      year={2018},
      eprint={1712.09199},
      archivePrefix={arXiv},
      primaryClass={math.CA},
      url={https://arxiv.org/abs/1712.09199}, 
}

@article{ca8d7fa2304c4b91a128f4574c71e5ed,
  author = "Semyon Alesker",
  title = "Description of translation invariant valuations on convex sets with solution of P. Mcmullen's conjecture",
  journal = "Geometric and Functional Analysis",
  year = "2001",
  volume = "11",
  number = "2",
  pages = {244--272},
  doi = {10.1007/PL00001675},
  issn = "1016-443X"
}

@book {MR102775,
    AUTHOR = {Hadwiger, H.},
     TITLE = {Vorlesungen \"uber {I}nhalt, {O}berfl\"ache und
              {I}soperimetrie},
 PUBLISHER = {Springer-Verlag, Berlin-G\"ottingen-Heidelberg},
      YEAR = {1957},
     PAGES = {xiii+312},
   MRCLASS = {52.00 (28.00)},
  MRNUMBER = {102775},
}

@book {guedj2017degenerate,
    AUTHOR = {Guedj, Vincent and Zeriahi, Ahmed},
     TITLE = {Degenerate complex {M}onge-{A}mp\`ere equations},
    SERIES = {EMS Tracts in Mathematics},
    VOLUME = {26},
 PUBLISHER = {European Mathematical Society (EMS), Z\"{u}rich},
      YEAR = {2017},
     PAGES = {xxiv+472},
      ISBN = {978-3-03719-167-5},
   MRCLASS = {32W20 (32Q20 32U15 32U20 32U40 35J96)},
  MRNUMBER = {3617346},
MRREVIEWER = {Slimane Benelkourchi},
       DOI = {10.4171/167},
}

@misc{darvas2020logconcavityvolumecomplexmongeampere,
      title={Log-concavity of volume and complex Monge-Amp\`ere equations with prescribed singularity}, 
      author={Tamás Darvas and Eleonora Di Nezza and Chinh H. Lu},
      year={2020},
      eprint={1807.00276},
      archivePrefix={arXiv},
      primaryClass={math.DG},
      url={https://arxiv.org/abs/1807.00276}, 
}

@misc{guedj2018toricpluripotentialtheory,
      title={Toric Pluripotential Theory}, 
      author={Vincent Guedj and Ahmed Zeriahi and Dan Coman and Sibel Sahin},
      year={2018},
      eprint={1804.03387},
      archivePrefix={arXiv},
      primaryClass={math.CV},
      url={https://arxiv.org/abs/1804.03387}, 
}

@book{Klimek1991,
  author    = {Klimek, Maciej},
  title     = {Pluripotential Theory},
  series    = {London Mathematical Society Monographs. New Series},
  volume    = {6},
  publisher = {The Clarendon Press, Oxford University Press},
  address   = {New York},
  year      = {1991},
  isbn      = {978-0-19-853568-3}
}

@unpublished{Xia2025,
  author = {Xia, Mingchen},
  title  = {Singularities in Global Pluripotential Theory},
  year   = {2025},
  note   = {Book manuscript, accepted for publication in the
            Lecture Notes in Mathematics series},
  url    = {https://mingchenxia.github.io/Lectures/SGPT_final.pdf}
}

@misc{pang2026capacitystabilitycomplexmongeampere,
      title={Capacity Stability of Complex Monge-Amp\`ere Equations with Moving Prescribed Singularities}, 
      author={Kai Pang and Haoyuan Sun and Zhiwei Wang and Xiangyu Zhou},
      year={2026},
      eprint={2607.12797},
      archivePrefix={arXiv},
      primaryClass={math.CV},
      url={https://arxiv.org/abs/2607.12797}, 
}

@article{Darvas_2020,
  author = {Darvas, Tamás and Lu, Chinh H},
  title = {Geodesic stability, the space of rays and uniform convexity in Mabuchi geometry},
  journal = {Geometry and Topology},
  year = {2020},
  volume = {24},
  number = {4},
  pages = {1907--1967},
  month = {11},
  doi = {10.2140/gt.2020.24.1907},
  issn = {1465-3060}
}

\end{document}